\documentclass[11pt,oneside,reqno]{amsart}
\usepackage{mathpazo} 
\normalfont
\usepackage[T1]{fontenc}

\usepackage[utf8]{inputenc}
\usepackage{amsmath}
\usepackage{amsfonts}
\usepackage{amssymb}
\usepackage{amsthm}
\usepackage{mathtools}
\usepackage{stmaryrd}
\usepackage{tikz}
\usepackage{tikz-cd}
\usepackage[all]{xy}
\usepackage{enumitem}
\usepackage[margin=1.2in]{geometry} 
\usepackage{multicol}

\makeatletter
\let\origsection\section
\renewcommand\section{\@ifstar{\starsection}{\nostarsection}}

\newcommand\nostarsection[1]
{\sectionprelude\origsection{#1}\sectionpostlude}

\newcommand\starsection[1]
{\sectionprelude\origsection*{#1}\sectionpostlude}

\newcommand\sectionprelude{%
  \vspace{1em}
}

\newcommand\sectionpostlude{%
  \vspace{1em}
}
\makeatother

\usepackage{skak}
\usepackage{accents}
\newcommand{\crown}[1]{\overset{\symking}{#1}}

\newtheorem{thm}{Theorem}[section]
\newtheorem{theorem}[thm]{Theorem}
\newtheorem{lemma}[thm]{Lemma}

\newtheorem{coro}[thm]{Corollary}

\newtheorem{prop}[thm]{Proposition}
\newtheorem{proposition}[thm]{Proposition}

\theoremstyle{definition}
\newtheorem{defi}[thm]{Definition}
\newtheorem{definition}[thm]{Definition}
\newtheorem{remark}[thm]{Remark}
\newtheorem{example}[thm]{Example}
\newtheorem{notation}[thm]{Notation}

\newtheorem{app}[thm]{Application}
\newtheorem{keyAssumption}[thm]{Key Assumption}

\newcommand\xqed[1]{%
  \leavevmode\unskip\penalty9999 \hbox{}\nobreak\hfill
  \quad\hbox{#1}}
\newcommand\exEnd{\xqed{$\diamondsuit$}}

\newcommand{\A}{\mathbb{A}}
\newcommand{\B}{\mathbb{B}}
\newcommand{\BB}{\B}
\newcommand{\R}{\mathbb{R}}
\newcommand{\T}{\mathbb{T}}
\newcommand{\TT}{\T}
\newcommand{\Z}{\mathbb{Z}}
\newcommand{\ZZ}{\Z}
\newcommand{\N}{\mathbb{N}}
\newcommand{\Q}{\mathbb{Q}}

\newcommand{\calG}{\mathcal{G}}

\newcommand{\calK}{\mathcal{K}}
\newcommand{\calk}{\calK}

\newcommand{\propStar}{\text{property } (*)}

\newcommand{\comd}[1]{$$\xymatrix{#1}$$}
\newcommand{\inj}[0]{\ar@{^{(}->}}
\newcommand{\surj}[0]{\ar@{->>}}
\newcommand{\bij}[0]{\ar@{^{(}->>}}
\newcommand{\lbij}[0]{\ar@{_{(}->>}}
\newcommand{\linj}[0]{\ar@{_{(}->}}
\newcommand{\parr}[0]{\ar@{.>}}
\newcommand{\pinj}[0]{\ar@{^{(}.>}}
\newcommand{\psurj}[0]{\ar@{.>>}}
\newcommand{\pbij}[0]{\ar@{^{(}.>>}}
\newcommand{\lin}[0]{\ar@{-}}
\newcommand{\llin}[0]{\ar@{=}}
\newcommand{\usu}[0]{\ar@{->}}
\newcommand{\cmapsto}{\ar@{|->}}

\newcommand{\toup}[1]{\stackrel{#1}{\longrightarrow}}
\newcommand{\calT}{\mathcal{T}}
\newcommand{\frakT}{\mathfrak{T}}
\newcommand{\id}{\mathrm{id}}

\newcommand{\dsum}{\displaystyle\sum}
\newcommand{\dprod}{\displaystyle\prod}
\newcommand{\dcup}{\displaystyle\bigcup}
\newcommand{\doplus}{\displaystyle\bigoplus}
\newcommand{\sdrop}{\backslash}
\newcommand{\into}{\hookrightarrow}
\newcommand{\onto}{\twoheadrightarrow}
\newcommand{\eps}{\varepsilon}
\newcommand{\ph}{\varphi}
\newcommand{\angbra}[1]{\left\langle #1\right\rangle}

\newcommand{\Spec}{\operatorname{Spec}}

\newcommand{\Cnvg}[1]{\operatorname{Cnvg}(#1)}
\newcommand{\cnvg}[1]{\Cnvg{#1}}
\newcommand{\CLS}[1]{\Cnvg{#1}}
\newcommand{\supp}{\operatorname{supp}}
\newcommand{\Frac}{\operatorname{Frac}}
\newcommand{\Terms}[1]{\operatorname{Terms}(#1)}

\newcommand{\nocontentsline}[3]{}
\newcommand{\tocless}[2]{\bgroup\let\addcontentsline=\nocontentsline#1{#2}\egroup}

\newcommand{\SpecBase}[2]{\ContBase{#1} #2}

\newcommand{\Spa}{\Cont}
\newcommand{\Cont}{\operatorname{Cont}}
\newcommand{\ContInt}{\Cont^{\circ}}
\newcommand{\ContBase}[1]{\Cont_{#1}}
\newcommand{\ContBaseInt}[1]{\ContInt_{#1}}
\newcommand{\ContIntBase}[1]{\ContBaseInt{#1}}
\newcommand{\BasicOpen}[2]{U(#1,#2)}
\newcommand{\basicopen}[2]{\BasicOpen{#1}{#2}}

\newcommand{\Qnorm}[1]{|#1|_{P}}

\def\Mon{M}
\def\base{S}
\newcommand{\GeneralSeries}[2]{#1\llbracket#2\rrbracket}

\newcommand{\GeorgeStory}[1]{}
\newcommand{\Hom}{\operatorname{Hom}}
\newcommand{\RnLexSemifield}[1]{\T^{(#1)}}
\newcommand{\rk}{\operatorname{rk}}
\newcommand{\fred}{Z}
\newcommand{\dlim}{\displaystyle\lim}
\newcommand{\Ima}{\operatorname{Im}}

\newcommand{\trdeg}{\operatorname{tr.deg}}

\newcommand{\Top}{\mathcal{T}\text{op}}
\newcommand{\Alg}{\mathcal{A}\text{lg}}
\newcommand{\TopAlg}{\mathcal{T}\text{op}\mathcal{A}\text{lg}}
\newcommand{\AlgInt}{\Alg^\circ}
\newcommand{\TopAlgInt}{\TopAlg^\circ}
\newcommand{\TOS}{\mathcal{TOS}}
\newcommand{\TOAG}{\mathcal{TOAG}}
\newcommand{\bigpi}{\text{\LARGE\(\pi\)}}
\newcommand{\diagbigpi}{\text{\Large\(\pi\)}}
\newcommand{\HT}{\operatorname{ht}}

\newcommand{\wt}[1]{\widetilde{#1}}
\newcommand{\lex}{\mathrm{lex}}

\usepackage{lipsum}

\makeatletter
\renewcommand\subsection{\@startsection{subsection}{2}%
  \z@{-.5\linespacing\@plus-.7\linespacing}{.5\linespacing}%
  {\centering\normalfont\scshape}}
\renewcommand\subsubsection{\@startsection{subsubsection}{3}%
  \z@{.5\linespacing\@plus.7\linespacing}{-.5em}%
  {\normalfont\scshape}}
\makeatother

\makeatletter
\let\origsubsection\subsection
\renewcommand\subsection{\@ifstar{\starsubsection}{\nostarsubsection}}

\newcommand\nostarsubsection[1]
{\subsectionprelude\origsubsection{#1}\subsectionpostlude}

\newcommand\starsubsection[1]
{\subsectionprelude\origsubsection*{#1}\subsectionpostlude}

\newcommand\subsectionprelude{%
  \vspace{0.5em}
}

\newcommand\subsectionpostlude{%
  \vspace{0.5em}
}
\makeatother

\title{Tropical adic spaces I: \\ \scriptsize{Topological semirings}
}
\title[Tropical adic spaces I:\quad The continuous spectrum of a topological semiring]{Tropical adic spaces I: \\ \scriptsize{The continuous spectrum of a topological semiring}}

\author{Netanel Friedenberg }
\address{Department of Mathematics, Tulane University, New Orleans, LA 70118, USA}
\email{nfriedenberg@tulane.edu}

\author[K.~Mincheva]{Kalina~Mincheva}
\address{Department of Mathematics, Tulane University, New Orleans, LA 70118, USA}
\email{kmincheva@tulane.edu}

\begin{document}


\begin{abstract}
Towards building tropical analogues of adic spaces, we study certain spaces of prime congruences as a topological semiring replacement for the space of continuous valuations on a topological ring. This requires building the theory of topological idempotent semirings, and we consider semirings of convergent power series as a primary example. We consider the semiring of convergent power series as a topological space by defining a metric on it. 
We check that, in tropical toric cases, the proposed objects carry meaningful geometric information. In particular, we show that the dimension behaves as expected. 
We give an explicit characterization of the points in terms of classical polyhedral geometry in a follow up paper. 

\ \\
\end{abstract}
\maketitle
\vspace{-2em}
\tableofcontents

\section{Introduction}

\vspace{-.5em}

\subsection{Context of this work}

Tropical varieties are often described as combinatorial shadows of algebraic varieties. However, they do not naturally carry any scheme structure. In tropical geometry most algebraic computations are done on the classical side, using the algebra of the original variety. The theory developed so far has explored the geometric aspect of tropical varieties as opposed to the underlying (semiring) algebra and there are still many commutative algebra tools and notions without a tropical analogue. 

In recent years, there has been a lot of effort dedicated to developing the necessary tools for semiring commutative algebra and tropical scheme theory using many different frameworks.
 Among these are prime congruences \cite{JM17}, blueprints \cite{Lor15},  tropical ideals \cite{MR18}, tropical schemes \cite{GG16}, super-tropical algebra \cite{IR10}, and systems \cite{R18}.
The goal of these programs is to endow the varieties with extra structure analogous to the extra structure schemes carry over classical varieties. 
This, in turn, would allow for the further exploration of the properties of tropicalized spaces independent of their classical algebraic counterparts.


One peculiarity of these frameworks is that the framework itself is purely algebraic, but the theory does not recover the scheme points or classical points of a variety. Instead, some of these theories can reconstruct the underlying topological space of the Berkovich analytification of an affine variety. Motivated by this discrepancy, we work towards developing an analytic theory of tropical spaces, namely tropical adic spaces. The present paper is the first stage of this program.

Tropical methods have gained a lot of popularity in recent years because they provide a new set of purely combinatorial tools which have been used to approach classical problems, notably in the study of curves and their moduli spaces \cite{FJP20}, \cite{JR21}, \cite{Mik05}, \cite{Mik06}, motivic integration \cite{NP19}, \cite{NPS18}, toric degenerations and Newton-Okounkov bodies \cite{BMNc21}, \cite{Bos21}, \cite{EH22}, \cite{KM16}, mirror symmetry \cite{Gro11}, \cite{HJMM22}, \cite{NXY19}, and tropical approach to a generalized Hodge conjecture \cite{BH17}.
Even though tropical geometry has proved useful in many cases, there are some natural limitations to the theory. The goal of developing tropical commutative algebra and scheme theory is to push the methods further, beyond the study of curves, and to help control the combinatorics of the classical tropical varieties. 
The commutative algebra of semirings has also been shown to be relevant in other situations, such as the classification of torus equivariant flat families of finite type over a toric variety \cite{KM19}.
Connes and Consani's program to prove the Riemann Hypothesis \cite{CC14}, \cite{CC16} has worked deeply with the commutative algebra of semirings.
Additionally, the connections between semiring algebra and classical tropical geometry have been further developed in \cite{JS21}. 

One of the motivations for seeking scheme structure on a tropical variety is to get better control over the tropical objects via a functorial construction. Some potential applications of this would be as follows.

(1) While tropical geometry has been very useful in enumerative geometry, there are situations where there are tropical sub-objects that cannot be realized as tropicalizations of algebraic sub-objects and situations where multiple algebraic sub-objects tropicalize to the same tropical sub-object. By building a tropical theory with more nuanced algebraic information, one would be able both to distinguish non-realizable tropical sub-objects and to count tropical sub-objects with multiplicities.

(2) Tropical methods have been very successful in showing that specific subsets of moduli spaces of curves are nonempty by constructing algebraic curves with certain properties. These constructions proceed by first finding an abstract tropical curve with corresponding properties. However, the existing notions of abstract tropical varieties in higher dimensions (such as those found in \cite{DC20} and \cite{KS15}) have not been successfully used since the combinatorially given gluing data is cumbersome. Our construction of tropical adic spaces will replace this complicated gluing data with a sheaf of semirings.

(3) One fundamental limitation on the use of tropical geometry so far has been that the tropicalization of a variety is dependent upon an embedding of that variety in a toric variety. Therefore, not every property of the tropicalization is an inherent property of the variety. However, the Abel-Jacobi map of a totally degenerate pointed curve canonically embeds it into its Jacobian, which in turn can be tropicalized. Unfortunately, no existing tropical scheme theory encompasses this tropicalization, as it inherently uses non-archimedean analytic geometry. Our theory, however, will be analytic in nature and so should enable us to recover more information about (and invariants of) the curve.


(4) Toroidal morphisms between toroidal embeddings are a ubiquitous class of morphisms of varieties. They are used, for example, in semisistable reduction results (see \cite{KKMSD73}, \cite{AK00}, and \cite{ALT19}) which are key for the study of compactifications of moduli problems. In this and other contexts, the way toroidal embeddings are used is to transform a question about morphisms of varieties into a question about morphisms of cone complexes. However, it is not known what property of a map of cone complexes should correspond to properness of the map of toroidal embeddings. We expect that our theory will provide a semiring version of the valuative criteria of properness which will allow us to translate properness of a toroidal morphism into a condition on the proposed tropical adic spaces.

Much like our work, another, very recent, paper gives a polyhedral-geometric interpretation of a space of valuations.
Motivated by potential applications to certain moduli spaces and to Newton-Okounkov bodies, Amini and Iriarte study in \cite{AI22} spaces of quasi-monomial valuations to $\R^k_{\mathrm{lex}}$. 
They study these spaces from the tropical perspective, and show that they can be interpreted as tangent cones of classical tropical spaces. 
In toric cases, the valuations considered in \cite{AI22} are analogous to the valuations we consider in this paper, but there are several important differences. 
First, Amini and Iriarte work over a trivially valued field, whereas we consistently suppose that our base semifield is infinite, which corresponds to imposing that the valuation on the base field is nontrivial. 
As the previous sentence alludes, they treat valuations as maps from a ring to a particular totally ordered abelian group while our valuations are morphisms in the category of semirings.
Second, they consider valuations to $\R^k_{\lex}$. In contrast, while our value groups always admit an embedding in some $\R^k_{\lex}$, we consider equivalence classes of arbitrary rank valuations and so work independently of both $k$ and the particular embedding into $\R^k_{\lex}$.
This makes the spaces considered in \cite{AI22} more akin to the Hahn spaces of \cite{FR15} than adic spaces. 



While a notion of adic tropicalization has been developed in \cite{FP19} by Foster and Payne it bears no relation to ours. Their adic tropicalizations impose extra structure on the exploded tropicalizations of originally defined in \cite{Pay09b} and carry sheaves of rings based on the particular variety being tropicalized. Our spaces live entirely in the framework of idempotent semirings and the points are of a completely different nature.

The relation of tropical geometry to Berkovich spaces (and so Huber adic spaces) has been well-studied \cite{BS15}, \cite{FGP14}, \cite{GRW16}. This connection motivates us to look at adic geometry to build a tropical scheme theory.  

Another motivation for working towards a scheme theory with analytic flavor is the work of Gubler. In \cite{Gub07A} he constructs tropicalization of analytic subvarieties of tori and uses it in \cite{Gub07B} to prove the Bogomolov conjecture for an abelian variety over a function field which is totally degenerate at a place. A tropical scheme theory that is completely algebraic will never be able to include these tropicalizations, but a tropical scheme theory with an analytic side will.

The main idea of our approach is to use prime congruences as points of the tropical scheme rather than prime ideals. Unlike in the case of rings, there is no bijection between the sets of ideals and congruences.
In \cite{MR18}, the authors define a notion of tropical ideals and show that sending a tropical ideal to its bend congruence gives a correspondence between these and certain tropical schemes as defined in \cite{GG16}. However, in the Zariski topology on the set of prime ideals in the sense of \cite{GG16}, the open sets are too large - the condition that a semiring element does not evaluate to zero is much less restrictive than in algebraic geometry over fields and rings. 
We make first steps towards a more suitable topology, which is reminiscent of the topology of adic space. While there have been some attempts to localize by certain congruences \cite{Izh19}, the intuition from analytic geometry is that we should look for a completion rather than localization. 

These kinds of prime congruences have shown up, in their incarnation as monomial (pre)orders, in recent works on tropical differential equations \cite{MG23}, integral models \cite{BFGN}, and Gr\"{o}bner theory 
\cite{VV23}.


When working over a trivially valued base field, the space of valuations on a finitely generated algebra becomes the Riemann-Zariski space, which has a lot of algebraic flavor. 
Trivial valuations can be viewed as taking values in the semifield $\B=\{0,1\}$ whose multiplication operation is usual multiplication and whose addition operation is $\max(\cdot,\cdot)$.
Just as in the ring case, the set of prime congruences on a toric monoid $\B$-algebra form the combinatorial Riemann-Zariski space of \cite{EI06}. 
We therefore restrict our attention to working with totally ordered semifields different from $\B$. 


\subsection{Constructions and results}

In constructing adic spaces, before defining the adic spectrum of a pair, one defines the continuous spectrum of a topological ring. 
We recall that the continuous spectrum of a topological (usually f-adic\footnote{An f-adic or Huber adic ring $A$ is is a topological ring admitting an open subring $A_0 \subset A$ on which the subspace topology is the $I$-adic topology, for some finitely generated ideal $I$ of $A_0$. F-adic rings are of central importance to the theory of adic spaces.}) ring $R$, denoted $\operatorname{Cont}R$, is the set of equivalence classes of continuous valuations $v$ on $R$. 
The valuation $v$ is considered continuous if the map $R\to\Gamma_{v}\cup\{0\}$ is continuous, where $\Gamma_{v}$ is the totally ordered abelian group generated by $\Ima(v)\sdrop\{0\}$ and $\Gamma_v\cup\{0\}$ is given a topology first defined by Huber in \cite{Hub93}. The structure on $\Gamma_v\cup\{0\}$ is that of a totally ordered additively idempotent semifield\footnote{See Theorem~\ref{thm:TOSAndTOAGEquiv} in Appendix~\ref{app:TotOrdAbGps1} for further explanation.}, and we refer to the topology on it as the Huber topology. 
The (idempotent) semiring analogue of a valuation is a homomorphism to a totally ordered semifield, and the data of an equivalence class of valuations is given by a prime congruence. 
As a first step towards building tropical adic spaces, for any topological semiring $A$ we consider the set of prime congruences $P$ on $A$ for which $A/P\neq\B$ and the map from $A$ to the residue semifield $\kappa(P)$ is continuous. 
Here $\kappa(P)$ is the semifield generated by $A/P$, equipped with a Huber topology; we write $|\bullet|_P$ for the induced map $A\to\kappa(P)$. 
We call that set the continuous spectrum of $A$, and denote it by $\Cont A$.
There is also a direct link with adic spaces: a continuous generalized valuation in the sense of \cite[Definition 2.5.1]{GG16} from a ring $R$ to a semiring $A$ induces a map from $\Cont{A}$ to $\Cont{R}$. 

Since we are motivated by tropical geometry, and because it greatly simplifies the theory, we often restrict to $\base$-algebras as opposed to general topological semirings. Here $\base\neq\B$ is a sub-semifield of the tropical semifield $\T$. In future work, we will relax this hypothesis to consider semiring analogues of Tate rings in the sense of \cite[Section~6.2]{Wed12} and possibly more general topological semirings.

In Section~\ref{ch:ContTopSpace} we specify a topology on $\Cont A$ that is a direct analogue of the topology on the continuous spectrum of a topological ring. That is, we declare that the sets of the form $U(f, g) = \{P \in \Cont A: |f|_P \leq |g|_P \neq 0_{\kappa(P)}\}$ with $f, g \in A$ are a subbase for the topology. Combining the results of Section~~\ref{ch:ContTopSpace} and Proposition~\ref{prop:Cont_is_functor} we get our first result:

\begin{theorem}\label{thm: Intro-Cont-Functor}
    The collection of sets $U(f,g)$ is a base for the topology on $\Cont A$ which is closed under intersection. Moreover, $\Cont$ is a contravariant functor from the category of topological $\base$-algebras to the category of topological spaces. 
\end{theorem}
This result allows us, in upcoming work, to define tropicalization as a morphism in a category of tropical adic spaces.



Just as in the case of the continuous spectrum of a topological ring, one does not expect there to be a nice sheaf on $\Cont A$; 
in Huber's theory, the correct pre-sheaf is defined on the adic spectrum of an affinoid ring which is a pair $(A, A^+)$ where $A$ is an f-adic ring and $A^+\subseteq A$ a ring of integral elements\footnote{In the theory of adic spaces, a subring $A^+\subseteq A$ is called a \emph{ring of integral elements of $A$} if it is integrally closed and open in $A$ and consists of power bounded elements in $A$.} of $A$.

If $\base$ is a totally ordered semifield and $A$ is an $\base$-algebra, we also consider the set $\ContBase{\base}A$ of prime congruences $P$ on $A$ such that the map from $\base$ to the residue semifield of $P$ is continuous. By endowing $A$ with a universal topology introduced in Appendix~\ref{sect: universal-topology}, we interpret this as $\Cont A$. For technical reasons it is also convenient to consider the subset $\ContBaseInt{\base} A$ consisting of those $P\in\ContBase{\base}A$ such that the inverse image of $0\in A/P$ is just $0\in A$.
We prove analogous results to Theorem~\ref{thm: Intro-Cont-Functor} for $\ContBase{\base}$ and $\ContIntBase{\base}$. 

\vspace{0.2em}
Our next focus is to get a concrete handle on the points of the space $\ContBase{\base}A$. When $\Mon$ is a toric monoid corresponding to a cone $\sigma$, we say that $\ContBase{\TT}\TT[\Mon]$ is the set of adic points of the tropical toric variety $N_{\R}(\sigma)$. To justify this, we recall what happens in the classical scheme case. Suppose that $X$ is an affine variety over a rank one valued field $k\to\T$ with valuation ring $k^\circ$ and that $R$ is the coordinate ring of $X$. In this case the Berkovich analytification of $X$ can be identified with the set of valuations (multiplicative seminorms) $v:R\to\T$ extending the valuation on $k$ \cite[Remark~3.4.2]{Ber90}. On the other hand, the Huber adic analytification of $X$ consists of\footnote{While this is well-known to experts, we could not find this statement in the literature. We have provided a proof in Appendix~\ref{app:AdicAnalytification}.} equivalence classes of those valuations on $R$ (with target an arbitrary totally ordered additively idempotent semifield) which are continuous when restricted to $k$ and which are bounded above by $1$ on $k^\circ$. The semiring analogue of this last condition always holds because the ``valuation semiring'' of $\base$ is $\{a\in\base\;:\; a\leq 1_S\}$ and semiring homomorphisms preserve the order. So, because $N_{\R}(\sigma)$ can be viewed as $\Hom(\T[\Mon],\T)$ (see \cite[Theorem 6.3.1]{GG16}), we find that $N_{\R}(\sigma)$ is analogous to a Berkovich analytification and $\ContBase{\TT}\TT[\Mon]$ would be the corresponding adic analytification.

In \cite{JM17} it is shown that, since quotients of $\T[\N^n]$ and $\T[\Z^n]$ by a prime congruence $P$ are totally ordered, there is an equivalence between prime congruences and valuated monomial (pre)oreders, which in turn can be described by (not generally unique) real-valued matrices with $n+1$ columns. More precisely given such a matrix $C$, for a monomial $m = t^ax_1^{k_1} \dots x_n^{k_n}$ (where $t^a$ is the tropical number corresponding to the real number $a$), we give $m$ the weight $C \begin{pmatrix} a \\ k_1 \\ \vdots \\ k_n \end{pmatrix}$ and compare the weights using the lexicographic order. Two polynomials are then equivalent modulo $P$ if their leading terms have the same weight. 
We explore this relation in depth in \cite{FM23}.

For the case of $\T[\Mon]$ with $\Mon$ a general toric monoid, where matrices are not a viable option for understanding prime congruences (see Remark~\ref{rem:NotMatrices,HahnEmbeddings}), we have to use the more technical machinery of Hahn embeddings developed in Section~\ref{subsec:HahnEmbeddings}.
Specific cases of Corollary~\ref{coro: matrix_form} and Section~\ref{subsec:HahnEmbeddings} give us the following result.

\begin{theorem}[How to give an adic point of a tropical toric variety]\label{thm:IntroGiveAdicPoint}
For $\Mon=\N^n$ or $\Z^n$ the prime corresponding to a matrix $C$ is in $\ContBase{\T}\T[\Mon]$ if and only if the $(1,1)$-entry of $C$ is $1$. For $\Mon$ a general toric monoid, every point $P$ of $\ContBase{\T}\T[\Mon]$ can be given by a monoid homomorphism $\T^\times\times\Mon\to(\RnLexSemifield{k},\times)$ 
satisfying certain properties.
\end{theorem}
\vspace{-0.2cm}\noindent Here, $\RnLexSemifield{k}$ is the semifield whose multiplicative group is $\R^k$ ordered lexicographically.\\

The results on the continuous spectrum in Section~\ref{subsec:HahnEmbeddings} also allow us to describe specializations of points in a simple way. Unlike in the theory of adic spaces, where specializations in $\Cont R$ have to be understood by decomposing them into a horizontal specialization and a vertical specialization, we are able to give a simple characterization of specializations in $\Cont A$ as inclusions of prime congruences; see Proposition~\ref{prop:SpecializationsInCont}.



Our primary focus is on understanding the continuous spectra of topological semirings of convergent power series.
In these series we allow coefficients in our sub-semifield $\base\neq\B$ of $\T$ and exponents in a toric monoid $\Mon$. While the set of all power series is not necessarily a semiring, the set $\Cnvg{P}$ of power series which converge at $P\in\ContBaseInt{\base}\base[\Mon]$ is. 
{We equip $\Cnvg{P}$ with a metric and show that $\Cnvg{P}$ is a topological semiring with the topology coming from the metric. }

\smallskip

The next main result of the paper characterizes the continuous spectrum of the semiring of convergent power series (at a point) and relates it to the continuous spectrum of a {generalized polynomial} 
semiring. 

\begin{theorem}[Theorem~\ref{thm:crown}]\label{IntroThm:Crown}
Let $P \in\ContIntBase{\base}{\base[\Mon]}$ and let $\iota: \base[\Mon] \rightarrow \Cnvg{P}$ be the inclusion map. The induced map $\iota^*:\Cont(\Cnvg{P}) \rightarrow \SpecBase{\base}{\base[\Mon]}$ 
is a bijection onto the set of those prime congruences  $P'\in\SpecBase{\base}{\base[\Mon]}$ for which there exists a nonzero $c \in \base$ such that for any term $m \in \base[\Mon]$, $c|m|_P \geq |m|_{P'}$.
\end{theorem}

We also give an explicit construction of any $Q\in\Cont(\Cnvg{P})$ from the corresponding point of $\ContBase{\base}\base[\Mon]$; see Corollary~\ref{coro:propStarImpliesSpa}.



%

We give a geometric interpretation to the algebraic characterization of $\Cont(\Cnvg{P})$ in Theorem~\ref{thm:crown_geom}. Let $\sigma$ be the cone giving rise to the toric monoid $\Mon$. 
Recall that the adic points of the tropical toric variety $N_\R(\sigma)$ are the points of $\ContBase{\T}\TT[\Mon]$.
We define a natural map $\ContBase{\T}\TT[\Mon]\to N_{\R}(\sigma)$ that is is a tropical analogue of the map from the Huber adic analytification of a variety to the Berkovich analytification of that variety.

\begin{theorem}[Analytic regions in tropical toric varieties]
The map $\iota^*$ gives a bijection from $\Cont\Cnvg{P}$ to the set of adic points of the tropical toric variety $N_{\R}(\sigma)$ that map to a region of $N_{\R}(\sigma)$ given by imposing inequalities on the coordinates.
\end{theorem}

\begin{example}[The unit disc in tropical affine space]
Let $P$ be the prime congruence on $\T[x_1,\ldots,x_n]$ given by the matrix $\begin{pmatrix}1&0&\cdots&0\end{pmatrix}$. Then $\iota^*$ gives a bijection from $\Cont\Cnvg{P}$ to the set of adic points of tropical affine space that map to the tropical unit disc in (classical) tropical affine space.
\end{example}

In fact, Theorem~\ref{thm:crown_geom} works with $\T$ replaced by any sub-semifield $\base\neq\B$ and also explicitly describes the region that is the range of $\iota^*$. 
Specifically, it states that $P'$ is in the image of $\iota^*$ if and only if the image of $P'$ in the tropical toric variety $N_\R(\sigma)$ lies in $\omega + \bar{\sigma}$. Here $\omega$ is the image of $P$ in $N_\R(\sigma)$ and $\bar{\sigma}$ is the closure of $\sigma$ in $N_{\R}(\sigma)$.
In the polynomial and Laurent polynomial cases, corresponding to affine space and the torus, we give alternative explicit descriptions in terms of defining matrices for the primes; see Corollaries~\ref{coro:crown_torus} and \ref{coro:crown_affine}.


\begin{app}
An immediate application of the tools developed here appears in our follow up paper \cite{FM23}, where we give an explicit geometric characterization of the points of $\ContBase{\base}\base[\Mon]$ when $\Mon$ is a toric monoid. When combined with Theorem~\ref{thm:crown_geom}, this gives an interpretation of the points of $\Cont(\Cnvg{P})$ in terms of classical polyhedral geometry. This interpretation gives an application of our theory to Gr\"{o}bner theory over valued fields by classifying the ``valuated term orders'' used in that context. 
Our geometric classification in \cite{FM23} can be seen as complementary to an earlier algebraic result of Robbiano in \cite{Rob85} in the non-valuated case which classified usual monomial orders in terms of their defining matrices. 
\end{app}

Classically, the dimension of a Noetherian local ring and its completion are the same. { In Proposition~\ref{prop:ConvergentSeriesAreLimitsOfPolynomials} we show that, for $P\in\ContBase{\base}\base[\Mon]$, $\Cnvg{P}$ is the completion of $\base[\Mon]$ with respect to a metric induced by $P$. }
Corollary~\ref{coro:DimWithTCoeffs} shows a direct analog of the classical dimension result in the case $\base=\T$; the following result combines that with a geometric interpretation from \cite[Theorem 7.2.1]{Min16}.
\begin{theorem} Let $\Mon$ be a toric monoid corresponding to a cone $\sigma$ and let $P $ a prime congruence in $ \ContBaseInt{\T}{\T[\Mon]}$, then
$$\dim_{\mathrm{top}}\Cnvg{P}=\dim_\T \T[\Mon].$$ 
This is also equal to the (usual geometric) dimension of the tropical toric variety $N_\R(\sigma)$. 
\end{theorem}
\noindent Here, the dimension $\dim_\base A$ is the number of strict inclusions in the longest chain of prime congruences in $\ContBase{\base}{A}$ and $\dim_{\mathrm{top}} A$ is the number of strict inclusions in the longest chain of prime congruences in $\Cont {A}$, when $A$ is a topological $\base$-algebra.

Without the restriction on the coefficients we have the following more general statement. 

\begin{theorem}[Corollary~\ref{coro:DimIneqsAroundCont}]\label{IntroThm:Dimension}
Let $\base\neq\B$ be a sub-semifield of $\T$, $\Mon$ be a toric monoid, and $P $ a prime congruence in $ \ContBaseInt{\base}{\base[\Mon]}$. Then 
$$
\dim_{\base}\base[\Mon]\geq\dim_{\mathrm{top}}\Cnvg{P}\geq\dim_{\base}\base[\Mon] - \rk(\kappa(P)^\times/\base^\times).
$$

\end{theorem}

\noindent These inequalities are sharp; see Examples~\ref{ex:DimEx1}, \ref{ex:DimEx2}, and \ref{ex:DimEx3} and Proposition~\ref{prop:dim-group}.


In Appendix~\ref{app: trdeg} we show that the term $\rk(\kappa(P)^\times/\base^\times)$ can be interpreted as the transcendence degree of the semifield extension $\kappa(P)/\base$. This allows us to restate Abhyankar's inequality so that all three terms are transcendence degrees as follows.

\begin{coro}[Corollary to Proposition \ref{prop:TrDegIsRk} and Abhyankar's inequality]
Let $K_2/K_1$ be an extension of valued fields and let $\wt{K}_2/\wt{K}_1$ be the extension of residue fields. Also, let $\kappa_2/\kappa_1$ be the extension of value semifields, i.e., the extension of semifields corresponding the extension of value groups. Then
$$\trdeg(K_2/K_1)\geq\trdeg(\wt{K}_2/\wt{K}_1)+\trdeg(\kappa_2/\kappa_1).$$ 
\end{coro}

\subsection{Organization of the paper}

The paper is organized as follows: in Section~\ref{sect:prelim} we review the necessary background, starting with the basics of semirings and prime congruences on them. 
We give concrete examples and recall how to specify prime congruences on the polynomial and Laurent polynomial semirings in a concrete way. 
In addition, we discuss some of the basics of tropical toric varieties. In Section~\ref{sect: Cont} we introduce the continuous spectrum of a topological semiring and realize it both as a topological space and a functor. 
In Section~\ref{sect: Cnvg} we study the semiring of convergent power series at a point $P$.
We define a metric on $\Cnvg{P}$ in Section~\ref{sect: topology-Cnvg} and prove that the topology induced by this metric makes $\Cnvg{P}$ into a topological semiring.
In Section~\ref{sect: canExt} we study a canonical extension of any prime $P \in \ContBaseInt{\base}{(\base[\Mon])}$ to a prime on $\Cnvg{P}$. 
We then use these canonical extensions in Section~\ref{section:TheoremCrown} to give algebraic and geometric descriptions of $\Cont{(\Cnvg{P})}$. Specifically, it is there that we prove Theorem~\ref{IntroThm:Crown}. In Section~\ref{sect: dim-Cnvg} we investigate an analogue of Krull dimension for $\Cnvg{P}$, compute it exactly in many cases, and prove Theorem~\ref{IntroThm:Dimension}. 

We also have several short appendices. In Appendix~\ref{app:TotOrdAbGps1} we discuss some  basics of totally ordered abelian groups and establish notation {for them}. In Appendix~\ref{app:TotOrdAbGps2} {we} recall Hahn's embedding theorem and describe how to modify one proof of it to give a slightly stronger result which we use in our study of $\Cont A$. In Appendix~\ref{sect: universal-topology} we define a topology with a certain universal property on any semiring. In Appendix~\ref{app: trdeg} we introduce the transcendence degree of a totally ordered semifield extension, and observe that it is the same as a quantity that arose in our study of dimension in Section~\ref{sect: dim-Cnvg}. 
In Appendix~\ref{app:AdicAnalytification} we justify the relation between our theory and the adic analytifications of varieties stated between Theorems~\ref{thm: Intro-Cont-Functor} and \ref{thm:IntroGiveAdicPoint}.


\addtocontents{toc}{\protect\setcounter{tocdepth}{-1}}
\section*{Acknowledgements}

The authors thank the anonymous referee for the insightful comments, corrections, and suggestions how to make the paper better. The authors thank Sam Payne for pointing them in the direction of this project and Hernán Iriarte for exciting discussions. 

The first author worked on this project while at MSRI for the program on Birational Geometry and Moduli Spaces in spring 2019 and while at ICERM for the program on Combinatorial Algebraic Geometry in spring 2021. He thanks the organizers of those programs for the hospitality. The second author is partially supported by Louisiana Board of Regents Targeted Enhancement Grant 090ENH-21.

\addtocontents{toc}{\protect\setcounter{tocdepth}{2}}
\section{Preliminaries}\label{sect:prelim}


\begin{definition}
By a {\it semiring} $R$ we mean a commutative semiring with multiplicative unit. That is, $R$ is a set which is a commutative monoid with respect to each of two binary operations: an addition operation $+_R$, whose identity is denoted $0_R$, and a multiplication operation $\cdot_R$, whose identity is denoted $1_R$. Furthermore, we require that $0_R\neq1_R$, $0_R$ is multiplicatively absorbing, and multiplication distributes over addition. We omit the subscripts from the operations whenever this will not cause ambiguity. 
\end{definition}

\begin{definition}
    A {\it semifield} is a semiring in which all nonzero elements have a multiplicative inverse.
\end{definition}

\begin{defi}
We call a semiring $R$ {\it additively idempotent} if $a+a = a$ for all $a \in R$. We refer to additively idempotent semirings as just {\it idempotent}.
\end{defi}

\noindent If $R$ is an idempotent semiring, then the addition defines a partial order on $R$ in the following way: $a \geq b \Leftrightarrow a + b = a$.
With respect to this order, $a+b$ is the least upper bound of $a$ and $b$. Any time we mention an order on an idempotent semiring, it is this partial order. When we consider a totally ordered idempotent semiring, we mean an idempotent semiring for which the order coming from addition is a total order. 
\par\medskip
\noindent Some of the idempotent semifields that we use throughout the paper are:
\begin{itemize}
\item The {\it Boolean semifield} $\mathbb{B}$ is the semifield with two elements $\{1,0\}$, where $1$ is the multiplicative identity, $0$ is the additive identity and $1+1 = 1$. 
    \item The {\it tropical semifield} $\T$ is defined on the set $\R  \cup \{-\infty\} $, by setting the $+$ operation to be the usual maximum and the multiplication operation to be the usual addition, with $-\infty = 0_\T$.
    \item The semifield $\Q_{\max}$ is a sub-semifield of $\T$. As a set it is $\Q \cup \{-\infty\} $ and the operations are the restrictions of the operations on $\T$. We also use the same notation when $\Q$ is replaced by any other additive subgroup of $\R$.
    \item The semifield $\RnLexSemifield{n}$ is defined on the set $\R^n  \cup \{-\infty\} $, by setting the addition operation to be lexicographical maximum and the multiplication operation to be the usual pointwise addition.
\end{itemize}

Throughout this paper all monoids, rings and semirings are assumed to be commutative. All rings and semirings have a multiplicative unit. 
Whenever we use the word semiring, 
we refer to an additively idempotent semiring.
Our focus will be on totally ordered semifields which are different from $\B$. 

\begin{keyAssumption}
Henceforth, when we refer to totally ordered semifields, we implicitly assume that they are not isomorphic to $\B$. 
\end{keyAssumption}

\noindent In particular, whenever we refer to a sub-semifield of $\T$, we implicitly assume that it is not $\B$. Totally ordered semifields can be seen as the image of a non-archimedean valuation. The semifield $\B$ is then the image of the trivial valuation.


In order to distinguish when we are considering a real number as being in $\R$ or being in $\T$ we introduce some notation. For a real number $a$, we let $t^a$ denote the corresponding element of $\T$. In the same vein, given $a\in\T$, we write $\log(a)$ for the corresponding element of $\R\cup\{-\infty\}$. This notation is motivated as follows.
Given a non-archimedean valuation $\nu: K \rightarrow \R\cup\{-\infty\}$ on a field and $\lambda \in \R$ with $\lambda>1$, we get a non-archimedean absolute value $|\cdot|_\nu:K\rightarrow[0, \infty)$ by setting $|x|_{\nu}=\lambda^{\nu(x)}$. Since $\T$ is isomorphic to the semifield $\big([0,\infty),\max,\cdot_{\R}\big)$, we use a notation for the correspondence between elements of $\R\cup\{-\infty\}$ and elements of $\T$ that is analogous to the notation for the correspondence between $\nu(x)$ and $|x|_{\nu}$. This notation is also convenient as we get many familiar identities such as $\log(1_{\T})=0_{\R}$ and $t^a t^b=t^{a+_{\R}b}$.

As defined in \cite[Section 4]{Gol92} a semifield $S$ is called \textit{algebraically closed} if, for all $a\in S$ and all $n\in\Z_{>0}$, the equation $x^n=a$ has a solution in $S$. Note that, if $S$ is totally ordered, $x^n=a$ has at most one solution in $S$. If $L/S$ is an extension of semifields, the algebraic closure of $S$ in $L$ is $\{x\in L\ :\ \text{for some }n\in\Z_{>0},\  x^n\in S\}$.

\begin{defi}
Let $R$ be a semiring and let $a \in R\sdrop\{0\}$. We say that $a$ is a \emph{cancellative element} if for all $b, c \in R$, whenever $ab=ac$ then $b=c$. If all elements of $R\sdrop\{0\}$ are cancellative, then we say that $R$ is a \textit{cancellative semiring}.
\end{defi}


\begin{definition}\label{def: prime_domain}
A \textit{congruence} on a semiring $R$ is an equivalence relation on $R$ that respects the operations of $R$.
The \emph{trivial congruence} on $R$ is the diagonal $\Delta\subseteq R\times R$, for which $R/\Delta\cong R$.
We call a proper congruence $P$ of a semiring $R$  {\it prime} if $R/P$ is totally ordered and cancellative (cf.\ Definition 2.3 and Proposition 2.10 in \cite{JM17}).
\end{definition}

In view of the previous definition we introduce the following notation. 
\begin{notation}
Let $R$ be semiring and let $P$ be a prime on $R$. For two elements $r_1, r_2 \in R$ we say that $r_1 \leq_P r_2$ (resp. $r_1 <_P r_2$, resp. $r_1 \equiv_P r_2$) whenever $\overline{r_1} \leq \overline{r_2}$ (resp. $\overline{r_1} <  \overline{r_2}$, resp. $\overline{r_1} = \overline{r_2}$) in $R/P$. Here by $\overline{r}$ we mean the image of $r$ in $R/P$.
\end{notation}

Note that $P$ is determined by the relation $\leq_P$ on $R$. Moreover, if $\mathcal{G}$ is a set of additive generators for $R$, then $P$ is determined by the restriction of $\leq_P$ to $\mathcal{G}$. 

Given a homomorphism $\ph:R\to R'$ of semirings and a congruence $E$ on $R'$, the \emph{pullback of $E$ to $R$} is the congruence $\ph^*(E):=(\ph\times\ph)^{-1}(E)\subseteq R\times R$. The homomorphism $\ph$ induces an injective homomorphism $R/\ph^*(E)\into R'/E$. In particular, if $E$ is prime, then $\ph^*(E)$ is also prime.

The following is a restatement in the language of congruences of a widely known fact.
\begin{theorem}[The correspondence theorem for congruences]
Let $R$ be a semiring and let $E$ be a congruence on $R$. Pulling back along the quotient map $R\to R/E$ gives an inclusion preserving bijection between (prime) congruences on $R/E$ and (prime) congruences on $R$ that contain $E$. In particular, there exists a unique congruence on $R/E$, namely $\Delta$, whose pullback to $R$ is $E$.
\end{theorem}

A large class of semirings that we will be interested in are the monoid $\base$-algebras $\base[\Mon]$ where $\base$ is a totally ordered semifield and $\Mon$ a monoid.  Elements of $\base[\Mon]$ are finite sums of expressions of the form $a\chi^u$ with $a\in\base$ and $u\in\Mon$, which are called \emph{monomials} or \emph{terms}.
We call elements of $\base[\Mon]$ \emph{generalized semiring polynomials} or \emph{polynomials} when no confusion will arise.

Consider a prime congruence $P$ on $\base[\Mon]$ and $f=\dsum_{i=1}^na_i\chi^{u_i}\in\base[\Mon]$. We say that $a_j\chi^{u_j}$ is a \emph{$P$-leading term} of $f$ if $a_j\chi^{u_j}\geq_{P}a_i\chi^{u_i}$ for all $i$. Note that if $a_j\chi^{u_j}$ is a $P$-leading term of $f$, then $f\equiv_P a_j\chi^{u_j}$.


When $R = \base[\Mon]$, where $\base \subseteq \TT$ and $\Mon = \N^n$ or $\ZZ^n$ we can explicitly describe all of the prime congruences on $R$. 
By letting $x_i=\chi^{e_i}$ with $e_i$ the $i^\text{th}$ standard basis vector, 
$\base[\N^n]$ and $\base[\Z^n]$ are identified with a polynomial semiring and a Laurent polynomial semiring, respectively, each over $\base$ and in $n$ variables.

\begin{example}\label{ex: primes}
%

Let $R = \base[\ZZ^n]$ where $\base$ is a subsemifield of $\TT$. 
A $k\times (n+1)$ real valued matrix $C$ such that the first column of $C$ is lexicographically greater than or equal to the zero vector gives a prime congruence on $R$ as follows.
For any monomial $m=t^a\chi^u\in\base[\ZZ^n]$ we let $\Psi(m)=C\begin{pmatrix}a \\ u\end{pmatrix}\in\R^k$, where we view $u\in\Z^n$ as a column vector. We call $\begin{pmatrix}a \\ u\end{pmatrix}$ the \emph{exponent vector} of the monomial $m$. For any nonzero $f\in\base[\ZZ^n]$, write $f$ as a sum of monomials $m_1,\ldots,m_r$ and set $\Psi(f)=\max_{1\leq i\leq r}\Psi(m_i)$, where the maximum is taken with respect to the lexicographic order. Finally, set $\Psi(0)=-\infty$. We specify the prime congruence $P$ by saying that $f$ and $g$ are equal modulo $P$ if $\Psi(f)=\Psi(g)$. In this case we say that $C$ is a \emph{defining matrix} for $P$.

In this situation, we say that the first column of $C$ is the \emph{column corresponding to the coefficient} or the \emph{column corresponding to $\base$}. For $1\leq i\leq n$, we say that the $(i+1)^{\text{st}}$ column of $C$ is the \emph{column corresponding to $x_i$}.

Every prime congruence on $\base[\Z^n]$ has a defining matrix; see \cite{JM17}, where it is shown that the matrix can be taken to be particularly nice. \exEnd\end{example}

Before we discuss the corresponding construction for $\base[\N^n]$, we introduce a few more concepts.

Let $E$ be a congruence on a semiring $R$. The \emph{ideal-kernel} of $E$ is the set of elements $a \in R$ such that $(a,0_R) \in E$. It is easy to see that the ideal-kernel of $E$ is an ideal of $R$. If the ideal-kernel of $E$ is the zero ideal, then we say that $E$ has \emph{trivial ideal-kernel}.

Let $\ph:R\to R'$ be a homomorphism of semirings. The \emph{congruence-kernel} of $\ph$ 
, denoted $\ker(\ph)$, is the pullback $\ph^*(\Delta)$ of the trivial congruence on $R'$.
If $R'$ is totally ordered and cancellative, then $\ker(\ph)$ is prime. The \emph{ideal-kernel} of $\ph$ is the set of those $a\in R$ such that $\ph(a)=0_{R'}$. Clearly, the ideal-kernel of $\ph$ is the same as the ideal-kernel of $\ker(\ph)$. 

\begin{remark}\label{remark:IdealKernelAndGenerators}
If $\calG$ is a set of additive generators for $R$ and $E$ is a congruence on $R$, then the ideal-kernel $I$ of $E$ is determined by $I\cap\calG$. Specifically, $I$ is the additive closure of $I\cap\calG$.
\end{remark}

\begin{example}
When $R = \base[\ZZ^n]$ every prime congruence has trivial ideal-kernel. More generally, if $R$ is generated by units, then every prime congruence on $R$ has trivial ideal-kernel.
\end{example}

\begin{example}\label{ex: primes-affine}
When $R = \base[\N^n]$ the prime congruences with trivial ideal-kernel are restrictions of primes in $\base[\ZZ^n]$. The primes with nontrivial ideal-kernel correspond to prime congruences with trivial ideal-kernel on $\base[\ZZ^m]$ for some $m < n$, together with a choice of $m$ of $x_1,\ldots,x_n$. For more details we refer the reader to \cite{JM17}.\footnote{ While this is true for all sub-semifields $\base$ of $\T$, this was originally proved for $\base = \T$ in \cite{JM17}.} 
Alternatively, if $P$ is prime congruence on $\base[\N^n]$ with nontrivial ideal-kernel, then we can define it by a matrix in the following way. 
We allow some entries of the matrix $C$ to be $-\infty$, but if $-\infty$ occurs as one entry, then it must occur as all entries in the same column.
One defines the multiplication of a matrix $C$ with the exponent vector of a monomial as usual by extending the multiplication in the following way: $0\cdot(-\infty) = 0$ and for all $a > 0$, $a\cdot(-\infty) = -\infty$.
We then define $\Psi$ as in the Laurent polynomial case; if $C\begin{pmatrix}a\\u\end{pmatrix}$ has all of its entries $-\infty$, then we set $\Psi(t^a\chi^u)=-\infty$. We then let $P$ be defined as before, and call $C$ a defining matrix for $P$.

In this situation, $x_i$ is in the ideal-kernel of $P$ if and only if the column of $C$ corresponding to $x_i$ is $-\infty$. The ideal-kernel of $P$ is the ideal of $\base[\N^n]$ generated by such $x_i$. \exEnd\end{example}

\begin{example}\label{ex: order-primes}
When $R = \base[\Mon]$, where $\base \subseteq \TT$ and $\Mon = \N^n$ or $\ZZ^n$ we can explicitly describe the order that $P$ defines. Let $C$ be a defining matrix for $P$ as in Example~\ref{ex: primes} or Example~\ref{ex: primes-affine} and let $\Psi$ be the corresponding map. 
For any $f,g\in R$, $f\leq_{P}g$ exactly if $\Psi(f)\leq_{lex}\Psi(g)$. \exEnd\end{example}


It will be convenient to know that we can choose a defining matrix with a particular form. To accomplish this, we will use the following lemma.

\begin{lemma}[Downward gaussian elimination]\label{lemma:rref}
Let $\base \subseteq \T$ and $\Mon = \N^n$ or $\ZZ^n$ and let $P$ be a prime congruence of $\base[\Mon]$ with defining matrix $C$. Then the following elementary row operations do not change the prime.
\begin{itemize}
    \item multiplying a row by a positive constant.
    \item adding a multiple of a row to any row below it. 
\end{itemize}
\end{lemma}

\begin{proof}
Consider the prime congruence $P$ with defining matrix $C$ with rows some vectors $v_1, \ldots, v_k \in \R^{n+1}$. The prime (via matrix multiplication by $C$) defines an order on the set of monomials of $\base[\Mon]$ as in Example~\ref{ex: order-primes}. Let $m$ and $m'$ be monomials of $\base[\Mon]$ and denote their exponent vectors by $w$ and $w'$. Without loss of generality assume that $m \geq_P m'$. This means that $Cw \geq_{lex} Cw'$. Multiplying any of the rows of $C$ by a positive constant preserves the inequality. 

Consider $i < j \leq k$, so $v_i$ is a row of $C$ that is higher than $v_j$. If $Uw \geq_{lex} Uw'$, then there exists $q$, such that $v_lw = v_lw'$, for all $l<q$ and $v_q w \geq v_q w'$. If $j > q$, then adding a multiple of $v_i$ to $v_j$ does not affect $Cw \geq_{lex} Cw'$, because $v_j$ does not contribute to the above condition. If $j \leq q$ then $v_j w \geq v_j w'$ and $v_i w = v_i w'$. Then for any $a\in \R$ we still have that $(av_i+v_j)w \geq (av_i +v_j)w'$.
\end{proof}

Note that, if a row of the matrix is all zero, then removing the row does not change the prime congruence that the matrix defines. Thus, using downward gaussian elimination, we can always choose a defining matrix in which the rows are linearly independent (over $\R$).\\



Let $R$ be a semiring and $A$ a submonoid of the multiplicative monoid of cancellative elements in $R$. We define an equivalence relation $\sim$ on $A\times R$ as follows:  $(a, r) \sim (a', r')$ if and only if there exist $u, u' \in R$, such that $ur = u'r'$ and  $ua = u'a' \in A$. We think of the equivalence class of $(a,r)$ as the fraction $\frac{r}{a}$.

\begin{defi}\label{def: semiring-of-fractions}
Let $R$ be a semiring and $A$ the set of all cancellative elements in $R$. The usual operations on fractions make $(A \times R)/\sim$ into a semiring, called the \textit{total semiring of fractions} and denoted by $\Frac(R)$. There is a natural injective homomorphism $\eta:R\to\Frac(R)$ which we consider as an inclusion map.
\end{defi}

This semiring is well-defined and satisfies the usual universal property, namely: if $R$ is a semiring and $\ph:R\to B$ is a morphism of semirings, then $\ph$ factors through $\eta:R\to\Frac(R)$ if and only if $\ph$ maps every cancellative element of $R$ to an invertible element of $B$. Moreover, in this case it factors uniquely. For a detailed study of the semiring of fractions and its properties we refer the reader to \cite{Gol92} or \cite{HW98}. Note that $\Frac(R)$ is a semifield if and only if $R$ is cancellative.

\begin{example}
Let $\base$ be a totally ordered semifield and let $\Mon$ be a cancellative monoid. Then $\Frac(\base[\Mon])\cong\base[\Lambda]$, where $\Lambda$ is the groupification of $\Mon$.
\end{example}

The following result will be used in Section~\ref{sect: dim-Cnvg}. While we do not use the result in its full generality, this setting makes the proof much more straightforward.

\begin{proposition}\label{prop:primes-total-sFrac}
Let $R$ be a semiring generated by a set of cancellative elements. Let $\Frac(R)$ be the total semiring of fractions. Denote by $\eta:R \rightarrow \Frac(R)$ be the natural inclusion and 
and let $\eta^*$ be the pullback map on prime congruences.
Then $\eta^*$ is an injection and its image is the set of primes on $R$ with trivial ideal-kernel.
\end{proposition}
%
%
%
\begin{proof}
For any prime $P$ on $R$, define 
$$Q_P=\{(f,g)\in\Frac(R)^2\;:\; (af,ag)\in P\text{ for some cancellative }a\in R\}.$$
It is elementary to show that $Q_P$ is a congruence on $\Frac(R)$. Since each cancellative element of $R$ is a unit in $\Frac(R)$, then any prime congruence $Q$ on $\Frac(R)$ for which $\eta^*(Q)=P$ contains $Q_P$.

Suppose $P$ is a prime on $R$ which has nontrivial ideal-kernel. Because the ideal-kernel of $P$ is determined by its intersection with any set of additive generators, the ideal-kernel of $P$ contains some cancellative element $a\in R$, i.e., $(0_R,a)\in P$. Since $(0_R,a)=(a\cdot 0_R,a\cdot 1_R)$, we have $(0_R,1_R)\in Q_P$. So there is no prime $P$ on $\Frac(R)$ containing $Q_P$. Thus, $P$ is not in the image of $\eta^*$.

Now suppose that $P$ is a prime on $R$ which has trivial ideal-kernel. Then the images in $R/P$ of cancellative elements of $R$ are again cancellative. Thus, cancellative elements of $R$ map to units under the map $R\to R/P\to\Frac(R/P)$, so, by the universal property of $\Frac(R)$, we get a unique commutative diagram
\comd{
R\ar[d]_{\pi}\ar[r]^{\eta}&\Frac(R)\ar[d]^{\ph}\\
R/P\ar[r]_{\eta'}&\Frac(R/P).
}
The congruence-kernel $Q=\ker\ph$ of $\ph$ is a prime on $\Frac(R)$, and, by the commutativity of the diagram, $\eta^*(Q)=P$. In particular, $P$ is in the image of $\eta^*$. 

We claim that $Q=Q_P$. We already know that $Q\supseteq Q_P$. So say $(f,g)\in Q$, i.e., $\ph(f)=\ph(g)$. Write $f=\frac{f'}{a}$ and $g=\frac{g'}{b}$ for some $a,b,f',g'\in R$ with $a,b$ cancellative. Note that $abf=bf'\in R$, $abg=ag'\in R$, and 
$$\eta'\circ\pi(bf')=\ph\circ\eta(bf')=\ph\circ\eta(ag')=\eta'\circ\pi(ag').$$
Since $\eta'$ is injective, this gives us that $\pi(bf')=\pi(ag')$, so $(bf',ag')\in\ker\pi=P$. Thus $(abf,abg)=(bf',ag')\in P$ with $ab$ cancellative in $R$, so $(f,g)\in Q_P$. Thus, $Q=Q_P$. Also, note that $\Frac(R)/Q$ is isomorphic to $\Frac(R/P)$ as a $\Frac(R)$-algebra.

Now suppose $Q'$ is another prime on $\Frac(R)$ with $\ph^*(Q')=P$. Then $Q'\supseteq Q_P=Q=\ker\ph$. 
Therefore there is a prime $\wt{Q}$ on $\Frac(R/P)$ such that $\ph^*(\wt{Q})=Q'$. Let $\wt{P}=(\eta')^*(\wt{Q})$. By the commutativity of the diagram, $\pi^*(\tilde{P})=P$. But there is only one prime on $R/P$ whose pullback along $\pi$ is $P$, so $\wt{P}=\Delta_{R/P}$. By \cite[Proposition~3.8]{JM15}, $(\eta')^*$ is injective, so $\wt{Q}=\Delta_{\Frac(R/P)}$. Thus $Q'=\ph^*(\wt{Q})=\ph^*(\Delta_{\Frac(R/P)})=\ker\ph=Q$.
\end{proof}
%
%


\subsection{Basics of tropical toric varieties}
In this section we recall the definition of classical tropical toric varieties. 
These arise as the target of the extended Kajiwara-Payne tropicalization map from Berkovich analytifications of toric varieties \cite{Kaj08}, \cite{Pay09}. In Section~\ref{sect: geom-crown} we use a relationship between our spaces and tropical toric varieties that is analogous to the relationship between adic spaces and Berkovich spaces.

\par\medskip  
For a general introduction to toric varieties see \cite{CLS} or \cite{Ful93}. 
In this paper we follow conventions from \cite{Rab10}.
Let $\Lambda$ be a finitely generated free abelian group and denote by $\Lambda_\R $ the vector space $ \Lambda\otimes_{\Z}\R$. We write $N:=\Lambda^*=\Hom(\Lambda,\Z)$ and $N_\R=\Hom(\Lambda,\R)\cong N\otimes_{\Z}\R$. We will use the pairing $N_\R\times\Lambda_{\R}\to\R$ given by $(v,u)\mapsto\angbra{v,u}:=v(u)$. 
A cone $\sigma$ is a \textit{strongly convex rational polyhedral cone} in $N_\R$ if $\sigma = \sum\limits_{i=1}^{r}\R_{\geq 0}v_i$ for $v_i \in N$ and $\sigma$ contains no line. By a \textit{cone} we will always mean a strongly convex rational polyhedral cone. We denote by $\sigma^\vee$ the dual cone of $\sigma$,
$$\sigma^\vee = \{u\in \Lambda_\R : \left< v,u \right> \leq 0, \forall v\in \sigma\} = \bigcap\limits_{i = 1}^r \{u\in \Lambda_\R : \left< v_i,u \right> \leq 0\}.$$

\begin{definition}
    We call a monoid of the form $\Mon = \sigma^\vee \cap \Lambda$ a \textit{toric monoid}. When we want to specify $\sigma$, we say that $\Mon$ is the \textit{toric monoid corresponding to} $\sigma$.
\end{definition}

Because $\sigma$ is strongly convex, we can identify $\Lambda$ with the groupification of $M$.

\begin{definition} 
Let $\sigma$ be a cone in $N_\R$. We denote by $N_{\R}(\sigma)$ the set of monoid maps $\Hom(\sigma^\vee \cap\Lambda, \T)$, where we consider $\T$ as a multiplicative monoid. We call $N_{\R}(\sigma)$ the \emph{tropical toric variety} corresponding to the cone $\sigma$.
\end{definition}

The tropical toric variety $N_{\R}(\sigma)$ is endowed with a topology, given as follows. 
First, observe that $(\T, \cdot_{\T}) \cong (\R_{\geq 0}, \cdot_\R)$  as monoids. 
Thus, we can use the usual topology on $\R_{\geq0}$ to make $\T$ into a topological monoid. So we can give $N_{\R}(\sigma)=\Hom(\sigma^\vee\cap\Lambda,\T)\subseteq\T^{\sigma^\vee\cap\Lambda}$ the subspace topology where the space $\T^{\sigma^\vee\cap\Lambda}$ of all functions $\sigma^\vee\cap\Lambda\to\T$ carries the product topology. This makes $N_{\R}(\sigma)$ into a topological monoid using the operation $\cdot_{\T}$, although we will denote this operation $+_{\R}$. 
For any face $\tau$ of $\sigma$, the inclusion $\sigma^\vee\cap\Lambda\subseteq\tau^\vee\cap\Lambda$ induces a map $N_{\R}(\tau)\to N_{\R}(\sigma)$, which identifies $N_{\R}(\tau)$ with an open topological submonoid of $N_{\R}(\sigma)$.
The copy of $N_{\R}\cong N_{\R}(\{0\})$ in $N_{\R}(\sigma)$ is exactly the set of units for the operation $+_{\R}$.

\begin{remark}
More generally, for any fan $\Sigma$ in $N_{\R}$, there is a tropical toric variety $N_{\R}(\Sigma)$. However, in this paper, we only use tropical toric varieties corresponding to cones.
\end{remark}

\section{The continuous spectrum $\Cont A$ of a topological semiring $A$}\label{sect: Cont}

In this section we define the continuous spectrum of a topological semiring $A$ in terms of continuity of a morphism $A\to K$ with $K$ a totally ordered semifield. We start by defining a topology on any totally ordered semifield. The results in this section rely heavily on the assumption that the semifields we consider are infinite, i.e., not equal to $\BB$.

On any totally ordered semifield $K$ 
we define the \emph{Huber topology} 
as follows.\footnote{This is a standard construction for totally ordered abelian groups in adic geometry first defined in \cite{Hub93}; one recent reference is Remark~1.17 in \cite{Wed12}. This is the finest topology that makes a surjective valuation on a field continuous when we give the field the valuation/metric topology. {That is, if $F$ is a field with a surjective valuation to $\base$ (i.e., $F$ has value group $\base^\times$) and valuation ring $R$, then the Huber topology is the quotient topology on $\base\cong F/R^\times$.}\label{footnote:HuberTopology}} A subset $U$ of $K$ is open if and only if either $0_K$ is not in $U$ or if $U$ contains $[0_K,a]$ for some $a\in K^\times$. A base for this topology is given by the sets $\{a\}$ and $[0_K,a]$ for all $a\in K^\times$. It is routine to check that this topology makes $K$ into a topological semiring, and that it is Hausdorff. 
Whenever we consider a totally ordered semifield as a topological semiring, it will be using 
the Huber topology 
unless specifically stated otherwise. We start by considering when a homomorphism of totally ordered semifields is continuous.



\begin{lemma}\label{lemma:ContsHomOfTotOrdSemifields}
Let $\ph:\base\to K$ be a homomorphism of totally ordered semifields. The following are equivalent.
\begin{enumerate}[label=(\arabic*)]
    \item\label{LemmaContsHomSemifieldsItem:HomIsCont} The homomorphism $\ph$ is continuous.
    \item\label{LemmaContsHomSemifieldsItem:LessOrEqual} For every $a\in K^\times$ there is some $b\in\base^\times$ such that $\ph(b)\leq a$.
    \item\label{LemmaContsHomSemifieldsItem:StrictlyLess} For every $a\in K^\times$ there is some $b\in\base^\times$ such that $\ph(b)< a$.
\end{enumerate}
\end{lemma}\begin{proof}
We first show that \ref{LemmaContsHomSemifieldsItem:HomIsCont} and \ref{LemmaContsHomSemifieldsItem:LessOrEqual} are equivalent. 
Because any homomorphism of semifields has trivial ideal-kernel, for any $a\in K^\times$ we automatically have that $\ph^{-1}(a)$ is open. So $\ph$ is continuous if and only if $\ph^{-1}([0_K,a])$ is open for all $a\in K^\times$. Since $\ph^{-1}([0_K,a])$ contains $0_{\base}$, this happens exactly if $\ph^{-1}([0_K,a])$ contains $[0_{\base},b]$ for some $b\in \base^\times$. This, in turn, is equivalent to having $\ph([0_{\base},b])\subset [0_K,a]$. Because $\ph$, as a homomorphism of semirings, sends $0_{\base}$ to $0_K$ and respects the ordering, this happens if and only if $\ph(b)\leq a$.

Clearly \ref{LemmaContsHomSemifieldsItem:StrictlyLess} implies \ref{LemmaContsHomSemifieldsItem:LessOrEqual}. For the other direction, note that, for any $a\in K^\times$, there is some $a'\in K^\times$ with $a'<a$. If $a>1_K$ we can take $a'=a^{-1}$, if $a<1_K$ we can take $a'=a^2$, and if $a=1_K$ such an $a'$ exists because $K\neq\B$. Then by \ref{LemmaContsHomSemifieldsItem:LessOrEqual} there is some $b\in \base^\times$ such that $\ph(b)\leq a'<a$.
\end{proof}

\begin{coro}\label{coro:SurjSemifieldHomIsConts}
Any surjective homomorphism of totally ordered semifields is continuous.
\end{coro}\begin{proof}
This follows because any surjective homomorphism satisfies \ref{LemmaContsHomSemifieldsItem:LessOrEqual} of Lemma \ref{lemma:ContsHomOfTotOrdSemifields}.
\end{proof}

\begin{coro}\label{coro:ContsHomFromSubsemifieldOfT}
Let $\base$ be a sub-semifield of $\T$. 
Then any continuous homomorphism $\ph:\base\to K$ of totally ordered semifields is injective.
\end{coro}\begin{proof}
It suffices to show that, for $\base$ a sub-semifield of $\T$, \ref{LemmaContsHomSemifieldsItem:StrictlyLess} of Lemma \ref{lemma:ContsHomOfTotOrdSemifields} implies that $\ph$ is injective. Since $\ph$ is a homomorphism of semifields it has trivial ideal-kernel, so $\ph$ is injective if and only if the corresponding map $\ph^\times:\base^\times\to K^\times$ is injective. Since $\base\subseteq\T$, the only proper convex subgroup of $\base^\times$ is $\{1_{\base}\}$. So if $\ph^\times$ were  not injective, it would have to map all of $\base^\times$ to $1_K$. But this contradicts \ref{LemmaContsHomSemifieldsItem:StrictlyLess} of Lemma \ref{lemma:ContsHomOfTotOrdSemifields} applied with $a=1_K$.
\end{proof}


The following proposition concerns the continuity of an arbitrary function $\ph:X\to K$ with $K$ a totally ordered semifield and $X$ a topological space. In practice we will apply it when $\ph$ is a homomorphism of totally ordered semifields. While we only need this result for these cases, the greater generality makes it easier to state and prove.

\begin{prop}\label{prop:ContinuityDependsOnlyOnImage}
Let $X$ be a topological space, and let $K$ and $K'$ be totally ordered semifields. Suppose $\ph:X\to K$ is a function, $\psi:K\to K'$ is an injective homomorphism of totally ordered semifields, and $\psi\circ\ph:X\to K'$ is continuous. Then $\ph$ is continuous.
\end{prop}\begin{proof}
Consider $a\in K^\times$. 
Since $\psi^{-1}(\psi(a))=\{a\}$, $\ph^{-1}(a)=(\psi\circ\ph)^{-1}(\psi(a))$, which is open in $X$ because $\psi(a)\neq0_{K'}$ and $\psi\circ\ph$ is continuous.

So it suffices to show that $\ph^{-1}([0_K,a])$ is open in $X$ for each $a\in K^\times$. Since $\psi$ is injective and order-preserving, $[0_K,a]=\psi^{-1}([0_{K'},\psi(a)])$. So $\ph^{-1}([0_K,a])=(\psi\circ\ph)^{-1}([0_{K'},\psi(a)])$, which is open in $X$ because $\psi\circ\ph$ is continuous.
\end{proof}

In what follows, we make use of classical results on totally ordered abelian groups. For the basics of, and our choice of notation for, totally ordered abelian groups we refer the reader to Appendix~\ref{app:TotOrdAbGps1}.
Let $G$ be a totally ordered abelian group, written multiplicatively. A \emph{lexicographic product decomposition of $G$} is a way of writing $G=G_1\times G_2$ as a product of abelian groups with $G_1$ and $G_2$ totally ordered abelian groups such that the order on $G$ is the lexicographic order on $G_1\times G_2$.

\begin{prop}\label{prop:ContinuityUsingLexico}
Suppose we have a commutative diagram 
\comd{
K\ar[rr]^{\ph}&&K'\\
&\base\ar[ul]^{\psi}\ar[ur]_{\psi'}
}
of homomorphisms of totally ordered semifields. Suppose there are lexicographic product decompositions of $K^\times$ and $(K')^\times$ such that $\psi$ and $\psi'$ identify $\base^\times$ with the first factor in each decomposition. Then $\ph$ is continuous.
\end{prop}\begin{proof}
Identifying $\base^\times$ with its image in each of $K^\times$ and $(K')^\times$, we can write the lexicographic product decompositions of $K^\times$ and $(K')^\times$ as $K^\times=\base^\times\times G$ and $(K')^\times=\base^\times\times G'$ for some totally ordered abelian groups $G$ and $G'$. Since $\ph|_{K^\times}:K^\times\to (K')^\times$ is a homomorphism of totally ordered abelian groups, there are homomorphisms $\phi_{11}:\base^\times\to\base^\times$, $\phi_{22}:G\to G'$, $\phi_{21}:\base^\times\to G',$ and $\phi_{12}:G\to\base^\times$ such that, for all $b_1\in\base^\times$ and $b_2\in G$, $\ph(b_1,b_2)=(\phi_{11}(b_1)\phi_{12}(b_2),\phi_{21}(b_1)\phi_{22}(b_2))$. Because the above diagram commutes and we have used $\psi$ and $\psi'$ to identify $\base^\times$ with subgroups of $K^\times$ and $(K')^\times$, we have $\phi_{11}=\id_{\base^\times}$ and $\phi_{21}$ is the constant function $1_{G'}$. Thus,
$\ph(b_1,b_2)=(b_1\cdot\phi_{12}(b_2),\phi_{22}(b_2))$.

We claim that $\phi_{12}=1_{\base^\times}$. To see this, fix $g\in G$.  For any $s\in\base^\times$ which is greater than $1_\base$, we have $(s,1_G)>(1_{\base},g)$, so $(s,1_{G'})=\ph(s,1_G)\geq\ph(1_{\base},g)=(\phi_{12}(g),\phi_{22}(g))$. By the definition of the lexicographic order this tells us that $s\geq\phi_{12}(g)$; since this is true for all $s>1_\base$ in $\base^\times$, we must have $1_{\base}\geq\phi_{12}(g)$. The same statement is also true if we replace $g$ with $g^{-1}$, so $1_\base\geq\phi_{12}(g)^{-1}$. Together, these show that $\phi_{12}(g)=1_{\base}$.

Say $a=(a_1,a_2)\in (K')^\times=\base^\times\times G'$; we want to show that there is some $b=(b_1,b_2)\in K^\times=\base^\times\times G$ such that $\ph(b)\leq a$, i.e., $(b_1,\phi_2(b_2))\leq_{lex}(a_1,a_2)$. Pick any $b_1\in\base^\times$ which is less than $a_1$ and set $b=(b_1,1_G)$. Then $\ph(b)=(b_1,\phi_{22}(1_G))=(b_1,1_{G'})\leq_{lex}(a_1,a_2)$, as desired.
\end{proof}

\begin{coro}\label{coro:ContinuityUsingLexico}
Suppose that $\ph:\base\to K$ is a homomorphism of totally ordered semifields and there is a lexicographic product decomposition of $K^\times$ such that $\ph$ identifies $\base^\times$ with the first factor. Then $\ph$ is continuous. 
\end{coro}\begin{proof}
This is the particular case of Proposition~\ref{prop:ContinuityUsingLexico} in which $\psi$ is the identity map.
\end{proof}

\begin{defi}\label{def:ResidueSemifield}
Let $P$ be a prime congruence on a semiring $A$. The {\it residue semifield of $A$ at $P$}, denoted $\kappa(P)$, is the total semiring of fractions of $A/P$. 
We denote the canonical homomorphism $A\to A/P\to\kappa(P)$ by $a\mapsto|a|_P$.
\end{defi}

\begin{example}\label{example:MatricesAreHahnEmbeddings}
Let $P$ be a prime congruence on $R = \base[\Mon]$, where $\base \subseteq \TT$ and $\Mon = \N^n$ or $\ZZ^n$. Let $C$ be a defining matrix for $P$ with $k$ rows. 
There is a corresponding embedding $\ph:\kappa(P)\to\RnLexSemifield{k}$ with the following property: if $m = t^a x_1^{v_1}\cdots x_n^{v_n}$, then $\ph(|m|_P)$ is  
$C\begin{pmatrix}a\\ v_1\\ \vdots\\ v_n\end{pmatrix}$. 
For a polynomial $f = \sum_i m_i$ in $R$ we have $|f|_P=\max_i \{ |m_i|_P \}$. For the existence of such embeddings in greater generality, see Section \ref{subsec:HahnEmbeddings}. \exEnd\end{example}

\begin{remark}
Because $|\cdot|_P$ is a semiring homomorphism, for any $f,g\in A$ we have that $|f+g|_P=|f|_P+|g|_P$, rather than just $|f+g|_P\leq|f|_P+|g|_P$.
\end{remark}

\begin{defi}\label{def:ContinuousSpectrumOfSemiring}
Let $A$ be a topological semiring. The \emph{continuous spectrum of $A$}, denoted $\Cont A$, is the set of prime congruences on $A$ for which $A/P\neq\B$ and the map $A\to\kappa(P)$ is continuous.
\end{defi}

\begin{remark}
Let $A$ be a topological semiring. For technical reasons, we assume that $A$ has a topologically nilpotent unit. That is, there is a unit $a\in A$ such that $\dlim_{n\to\infty}a^n=0_A$.
Let $\ph:A\to K$ be a continuous semiring homomorphism to a totally ordered semifield, and let $P$ be the congruence-kernel of $\ph$. The hypothesis that $A$ is has a topologically nilpotent unit implies that $A/P\neq\B$. Since the induced map $\kappa(P)\to K$ is injective, Proposition~\ref{prop:ContinuityDependsOnlyOnImage} then gives us that $P\in\Cont A$. Thus, $\Cont A$ is also the space of all continuous homomorphisms $\ph:A\to K$ to a totally ordered semifield, up to equivalence. Here we say that $\ph:A\to K$ and $\ph':A\to K'$ are equivalent if there is some $A\to K''$ such that $\ph$ and $\ph'$ factor as $A\to K''\to K$ and $A\to K''\to K'$, respectively.
\end{remark}

\begin{example}\label{ex: congSemif}
Let $K$ be a totally ordered semifield. By Corollary~\ref{coro:SurjSemifieldHomIsConts}, $\Cont K$ consists of all prime congruences on $K$ except for the unique maximal congruence $\mathfrak{M}$ on $K$ for which $K/\mathfrak{M}=\B$. Using Theorem~\ref{thm:TOSAndTOAGEquiv}, this can be seen to be in natural bijection with the set of proper convex subgroups of $K^\times$.
\end{example}

Let $A$ be a topological semiring. A \emph{topological $A$-algebra} is a topological semiring $B$ together with a continuous homomorphism $A\to B$.

\begin{lemma}\label{lemma:BasicPropertiesOfPInSpaA}
Let $\base$ be a totally ordered semifield and let $A$ be a topological $\base$-algebra. For any $P\in\Cont A$, we have the following.
\begin{enumerate}[label=(\arabic*)]
    \item As a subset of $A^2$, $P$ is closed.
    \item\label{itemBasicPropertiesOfPInSpaA:OrderProperty} For every $a\in \kappa(P)^\times$ there is some $b\in\base^\times$ such that $|b|_P<a$.
    \item If $\base$ is a sub-semifield of $\T$, then the composition $S\to A\to \kappa(P)$ is injective.
\end{enumerate}
\end{lemma}\begin{proof}
Since $\kappa(P)$ is Hausdorff, the diagonal in $\kappa(P)^2$ is closed. 
Consider $A^2\to\kappa(P)^2$ which is the canonical map $A\to\kappa(P)$ on each component. This map is continuous, and the inverse image of the diagonal under this map is $P$. Thus, $P$ is closed in $A^2$.

Note that the composition $S\to A\to\kappa(P)$ is continuous. The remaining claims now follow from Lemma \ref{lemma:ContsHomOfTotOrdSemifields} and Corollary \ref{coro:ContsHomFromSubsemifieldOfT}.
\end{proof}

\begin{defi}
Let $\base$ be a topological semiring and let $A$ be an $\base$-algebra. Let $\tau$ be the universal topology on $A$ from Theorem \ref{thm:UniversalTopOnSemiring} applied to the function $\base\to A$. Then we define $\ContBase{\base} A:=\Cont (A,\tau)$.
\end{defi}

\begin{lemma}\label{lemma:ContinuityGivesGeorge}
Let $\base$ be a sub-semifield of $\T$. For any $\base$-algebra $A$, $\ContBase{\base} A$ is the set of those prime congruences $P$ on $A$ which satisfy the following property: for all $a\in\kappa(P)^\times$ there is some $b\in \base^\times$ such that $|b|_P<a$. For any such prime, the composition $\base\to A\toup{|\cdot|_P}\kappa(P)$ is injective.
\end{lemma}\begin{proof}
For any prime congruence $P$ on $A$, $\base\to A\toup{|\cdot|_P}\kappa(P)$ is a homomorphism of semifields, so it has trivial ideal-kernel. Thus, the above property implies that there is some $b\in \base^\times$ with $0_{\kappa(P)}<|b|_P<1_{\kappa(P)}$, so $\kappa(P)\neq\B$.

For any prime congruence $P$ on $A$ with $\kappa(P)\neq\B$, $(A,\tau)\toup{|\cdot|_P}\kappa(P)$ is continuous if and only if the composition $\base\to A\toup{|\cdot|_P}\kappa(P)$ is continuous. The result now follows from Lemma \ref{lemma:ContsHomOfTotOrdSemifields} and Corollary \ref{coro:ContsHomFromSubsemifieldOfT}.
\end{proof}

\begin{coro}\label{coro: matrix_form}
Let $\base$ be a subsemifield of $\T$ and let $\Mon$ be $\N^n$ or $\Z^n$. Let $P$ be a prime on $\base[\Mon]$ and let $C$ be any defining matrix for $P$. Then $P\in\ContBase{\base}\base[\Mon]$ if and only if the $(1,1)$ entry of $C$ is positive. In particular, if $P\in\ContBase{\base}\base[\Mon]$, then we can choose a defining matrix for $P$ whose first column is the first standard basis vector $e_1$.
\end{coro}

\begin{proof}
We may assume without loss of generality that $C=(c_{i,j})$ has no row of zeros. First note that $c_{1,1}$ cannot be negative, because the first column of $C$ must be lexicographically greater than or equal to $0$.
%
Suppose that $c_{1,1}$ is 0, and let $c_{1,i}$ be the first nonzero entry in the first row of $C$. 
Note that the first row has to have a nonzero entry by assumption. If $c_{1,i}>0$, then take $a = |x_{i-1}|_P^{-1}$ and if $c_{1,i}<0$, then take $a = |x_{i-1}|_P$. In either case $a < |b|_P$ for any $b\in \base^\times$, so Lemma~\ref{lemma:ContinuityGivesGeorge} tells us that $P\notin\ContBase{\base}\base[\Mon]$.

Now suppose $c_{1,1}>0$. By Lemma~\ref{lemma:ContinuityGivesGeorge}, it suffices to show that, for any column vector $v=\begin{pmatrix}v_1\\\vdots\\v_k\end{pmatrix}$, there is some $a\in\base^\times$ such that $C\begin{pmatrix}\log(a)\\0\\\vdots\\0\end{pmatrix}<_{lex} v$. Since $c_{1,1}$ is positive, this is accomplished by picking any $a\in\base^\times\subseteq\T^\times$ less than $t^{v_1/c_{1,1}}$, which is possible because $\base\neq\B$.

The final claim follows because, by Lemma~\ref{lemma:rref}, we can scale $c_{1,1}$ to $1$ and use it to eliminate all entries below. 
\end{proof}

\begin{remark}\label{rem:NotMatrices,HahnEmbeddings}
Unfortunately, when $\Mon$ is a toric monoid other than $\N^n$ and $\Z^n$, there is no canonical notion of matrices for the primes on $\base[\Mon]$ even in the case $\base=\T$; any notion of matrices in this context would depend on a presentation $\N^n\onto\Mon$ and there would be algebraic conditions on the columns of the matrices corresponding to the kernel of this presentation.

As such, instead of trying to use matrices to get a handle on the primes of $\ContBase{\base}\base[\Mon]$, we generalize the corresponding notion from Example~\ref{example:MatricesAreHahnEmbeddings} - embeddings $\kappa(P)\into\RnLexSemifield{k}$. We do this in Section~\ref{subsec:HahnEmbeddings} using Hahn's embedding theorem.
\end{remark}


\begin{example}
The above Corollary gives us a way to produce many examples of points of $\Cont(A)$. 
We delay looking at these examples in detail until the follow up paper \cite{FM23}, where we interpret these points in terms of polyhedral geometry and classify when two matrices give the same point of $\ContBase\base[\Mon]$.
\exEnd\end{example}

\subsection{Semiring consequences of Hahn's embedding theorem}
\label{subsec:HahnEmbeddings}
In order to get embeddings of totally ordered semifields, we use the equivalence of categories from Theorem~\ref{thm:TOSAndTOAGEquiv} to take advantage of known results from the theory of totally ordered abelian groups. To be able to take full advantage of those results to understand $P\in\Cont A$, we start with a lemma about the totally ordered abelian group $\kappa(P)^\times$.

\begin{lemma}\label{lemma:LeastArchClassOfResidueSemifield}
Let $\base$ be a sub-semifield of $\T$, and let $A$ be a topological $\base$-algebra. For any prime congruence $P\in\Cont A$, the totally ordered abelian group $\kappa(P)^\times$ has a smallest archimedean class, which is the archimedean class of $|a|_P$ for any $a\in\base$ greater than $1_{\base}$.
\end{lemma}\begin{proof}
Since $\base^\times$ is archimedean, all $|a|_P$ for $a>1_{\base}$ are in the same archimedean class.

Fix $a>1_{\base}$ and suppose there were some $b\in\kappa(P)^\times$ with $|a|_P\ll b$. That is, $|a|_P ^n<b$ for all natural numbers $n$. Then $b^{-1}<|a^{-n}|_P$ for all natural numbers $n$. Since $\base^\times$ is archimedean, this means that $b^{-1}<c$ for all $c\in\base^\times$. But this is impossible by Lemma \ref{lemma:BasicPropertiesOfPInSpaA} part \ref{itemBasicPropertiesOfPInSpaA:OrderProperty}.
\end{proof}

Let $I$ be a totally ordered set. 
As in Appendix \ref{app:TotOrdAbGps2}, we let $\R^{(I)}$ denote the additive subgroup of $\R^I$ consisting of those $f\in\R^{I}$ such that $\supp(f):=\{i\in I\;:\; f(i)\neq0\}$ is well-ordered. We endow $\R^{(I)}$ with the lexicographic order, making it a totally ordered abelian group.
We let $\RnLexSemifield{I}$ denote the totally ordered semifield corresponding to $\R^{(I)}$. Note that, if $I$ has a smallest element $i_0$, then projection to the $i_0$ factor gives a homomorphism of totally ordered abelian groups $\R^{(I)}\to\R$, so we get a homomorphism of totally ordered semifields $\bigpi:\RnLexSemifield{I}\to\T$. By Corollary \ref{coro:SurjSemifieldHomIsConts}, $\bigpi$ is continuous. Similarly, the inclusion of the $i_0$ factor is a homomorphism of totally ordered abelian groups $\R\to\R^{(I)}$, and so gives a homomorphism of totally ordered semifields $j:\T\to\RnLexSemifield{I}$. Since the map $\R\to\R^{(I)}$ is the inclusion of the first factor in a lexicographic product, Corollary~\ref{coro:ContinuityUsingLexico} tells us that $j$ is continuous.

\begin{remark}
The goal of this section is to develop tools for the study of $\Cont A$, where $A$ is a semiring of convergent power series. At first glace, the above set-up may appear overly general, especially since our results in Section~\ref{section:TheoremCrown} imply that one can take $I$ to be finite {in all of our cases of interest}. However, those results rely heavily on the existence of the embeddings into $\T^{(I)}$, without being able to assume the finiteness of $I$.
\end{remark}

\begin{prop}\label{prop:HahnForSpaSemifields}
Let $\base$ be a sub-semifield of $\T$, let $A$ be a topological $\base$-algebra, and let $P\in\Cont A$. Let $I$ be the set of archimedean classes of $\kappa(P)^\times$, and let $i_0\in I$ be the common archimedean class of $|a|_P$ for all $a\in\base^\times$ which is greater than $1_\base$, which, by Lemma \ref{lemma:LeastArchClassOfResidueSemifield}, is the smallest element of $I$. Let $\bigpi:\RnLexSemifield{I}\to\T$ and $j:\T\to\RnLexSemifield{I}$ be the continuous homomorphisms of totally ordered semifields defined above. Let $\iota$ be the composition $S\to A\toup{|\cdot|_P}\kappa(P)$ and let $\psi:S\to\T$ be the inclusion map. Then there is an embedding $\ph:\kappa(P)\to\RnLexSemifield{I}$ such that the diagrams
$$\vcenter{\vbox{\xymatrix{
\kappa(P)\ar[r]^{\ph}&\RnLexSemifield{I}\ar[d]^{\diagbigpi}\\
\base\ar[u]^{\iota}\ar[r]^{\psi}&\T
}}}
\text{\hspace*{1cm} and \hspace*{1cm}}
\vcenter{\vbox{\xymatrix{
\kappa(P)\ar[r]^{\ph}&\RnLexSemifield{I}\\
\base\ar[u]^{\iota}\ar[r]^{\psi}&\T\ar[u]_{j}
}}}
$$
commute.
\end{prop}\begin{proof}

This is an immediate consequence of Theorem~\ref{thm:Hahns-embedding}, Theorem~\ref{thm:TOSAndTOAGEquiv}, and Lemma~\ref{lemma:LeastArchClassOfResidueSemifield}.

\end{proof}

In particular, there is a continuous homomorphism $A\to\kappa(P)\to\RnLexSemifield{I}$ with congruence-kernel $P$ such that the composition $\base\to A\to\RnLexSemifield{I}\toup{\diagbigpi}\T$ is the inclusion of $\base$ into $\T$.

We call embeddings as in Proposition~\ref{prop:HahnForSpaSemifields} \emph{Hahn embeddings}. We will use Hahn embeddings to prove various results about $\Cont A$. One of the ways in which it will be most helpful is in understanding containments between $P,P'\in\Cont A$. An example of an application that will be used repeatedly is Corollary~\ref{coro:PrimesInSpaChainInducedMapOnResidueSemifields}, which says that, if $P\subseteq P'$, then there is an induced morphism $\kappa(P)\to\kappa(P')$. In order to use Hahn embeddings to consider containments, we first consider the case where $P\in\Cont A$ is maximal with respect to inclusion.

\begin{lemma}\label{lemma:MaximalGeorgePrimes}
Let $\base$ be a sub-semifield of $\T$ and let $A$ be a topological $\base$-algebra. Let $P\in\Cont A$ be maximal with respect to inclusion. Then there is a unique embedding  $\kappa(P)\into\T$ such that the diagram
\comd{
\base\ \inj[rr] \inj@<-1ex>[dr]&&\T\\
&\kappa(P)\inj[ur]
}
commutes.


\end{lemma}\begin{proof}
By Theorem~\ref{thm:TOSAndTOAGEquiv}, it suffices to prove the corresponding statement about $\base^\times$, $\kappa(P)^\times$, and $\T^\times$ as totally ordered abelian groups.

Because $P\in\Cont A$, Lemma \ref{lemma:LeastArchClassOfResidueSemifield} tells us that for any $s\in \base^\times$ such that $s>1_\base$ there is no $a\in\kappa(P)^\times$ such that $a\gg s$. 

Suppose there were some $a\in\kappa(P)^\times$ such that $s\gg a>1_{\kappa(P)}$ for some $s\in \base$. Then the convex subgroup $G$ of $\kappa(P)^\times$ generated by $a$ intersects $\base^\times$ only in $1_{\base}$. In particular, $\kappa(P)^\times/G\neq\B$. So the totally ordered semifield $\base'$ corresponding to $\kappa(P)^\times/G$ is a quotient of $\kappa(P)$, and by Corollary \ref{coro:SurjSemifieldHomIsConts}, the map $\kappa(P)\to\base'$ is continuous. Thus, the congruence-kernel $P'$ of the map $A\to\kappa(P)\to\base'$ is a prime congruence which strictly contains $P$, and satisfies $P'\in\Cont A$, contradicting the maximality of $P$.

Thus, for any $s\in \base^\times$ and $a\in\kappa(P)^\times$ with $s,a>1$ there are natural numbers $n,m$ such that $a^n>s$ and $s^m>a$. Since $\base^\times$ is an archimedean totally ordered abelian group, this implies that $\kappa(P)^\times$ has only one archimedean class, i.e., $\kappa(P)^\times$ is archimedean. Thus, by Theorem~\ref{thm: rank1-emb},  $\kappa(P)^\times$ embeds into $\T^\times$, and once we fix the image of a single element of $\kappa(P)^\times\setminus\{1_{\kappa(P)}\}$, the embedding is uniquely determined. By picking any element of $\base^\times\setminus\{1_{\kappa(P)}\}\subseteq\kappa(P)^\times\setminus\{1_{\kappa(P)}\}$ and sending it to itself in $\base^\times\subseteq\T^\times$, we get an embedding $\kappa(P)^\times\into\T^\times$. The corresponding embedding $\kappa(P)\into\T$ of semifields is the desired embedding.
\end{proof}

\begin{remark}
    We will show later, in Proposition~\ref{lemma:MaximalGeorgePrimesBack}, that the converse of Lemma~\ref{lemma:MaximalGeorgePrimes} also holds. 
\end{remark}

The next proposition will allow us to use Hahn embeddings to prove statements about containments between primes in $\Cont{A}$. Before that, we make a simple observation. 

\begin{lemma}
\label{lemma:SemifieldHomAndInequalities}
Let $\ph:K\to K'$ be a homomorphism of totally ordered semifields and let $a,b\in K$.
\begin{enumerate}
    \item If $a\leq b$ then $\ph(a)\leq\ph(b)$. 
    \item\label{item:HomStrictInequality} If $\ph(a)<\ph(b)$ then  $a<b$.
\end{enumerate}
\end{lemma}
\begin{proof}
(1) Since $\ph$, as a homomorphism of semirings, preserves addition, and the relation $\leq$ is defined in terms of addition, $\ph$ preserves $\leq$.

(2) This is the contrapositive of (1), with the roles of $a$ and $b$ flipped.
\end{proof}

\begin{proposition}\label{lemma:MatrixDeterminesMaximalPrimeAbove}
Let $\base$ be a sub-semifield of $\T$, and let $A$ be a topological $\base$-algebra. Let $P,P'\in\Cont A$ with $P'$ maximal with respect to inclusion. Let $I$ be the set of archimedean classes of $\kappa(P)$, which by Lemma \ref{lemma:LeastArchClassOfResidueSemifield} has a least element. Let $\bigpi:\RnLexSemifield{I}\to\T$ be the canonical projection. By Proposition \ref{prop:HahnForSpaSemifields} and Lemma \ref{lemma:MaximalGeorgePrimes} there are semiring homomorphisms $\ph_P:A\to\RnLexSemifield{I}$ and $\ph_{P'}:A\to\T$ such that $P$ is the congruence-kernel of $\ph_{P}$, $P'$ is the congruence-kernel of $\ph_{P'}$, and the compositions 
$$\base\to A\toup{\ph_{P}}\RnLexSemifield{I}\toup{\diagbigpi}\T \text{ and } \base\to A\toup{\ph_{P'}}\T$$ 
are both the inclusion of $\base$ into $\T$. With this setup, $P'\supseteq P$ if and only if $\ph_{P'}=\bigpi\circ\ph_{P}$.
\end{proposition}\begin{proof}
If $\ph_{P'}=\bigpi\circ\ph_{P}$ then $P$, being the congruence-kernel of $\ph_{P}$, is contained in $P'$, which is the congruence-kernel of $\ph_{P'}$.
So suppose $\ph_{P'}\neq\bigpi\circ\ph_{P}$. We wish to find a pair $\alpha\in P\sdrop P'$.

By hypothesis, there is some $f\in A$ such that $\ph_{P'}(f)\neq\bigpi\circ\ph_{P}(f)$. On the other hand, fixing some $a\in\base^\times\sdrop\{1_{\base}\}$, we have $\ph_{P'}(a)=\bigpi\circ\ph_{P}(a)\in\T^\times\sdrop\{1_{\T}\}$. Because $\ph_{P'}(f)$ and $\bigpi\circ\ph_{P}(f)$ are distinct elements of $\T$, there are integers $m$ and $n$ with $n$ positive such that $\ph_{P'}(a)^{m/n}$ is in between $\ph_{P'}(f)$ and $\bigpi\circ\ph_{P}(f)$. Here $\ph_{P'}(a)^{m/n}$ is being taken in the algebraically closed semifield $\T$, and the existence of $m$ and $n$ follows from the density of $\Q$ in $\R$. We now consider two cases.

\underline{Case 1}: $\ph_{P'}(f)<\ph_{P'}(a)^{m/n}<\bigpi\circ\ph_{P}(f)$. In this case  $\ph_{P'}(f^n)<\ph_{P'}(a^m)=\bigpi\circ\ph_{P}(a^m)<\bigpi\circ\ph_{P}(f^n)$. 
    Thus, part~(\ref{item:HomStrictInequality}) of  Lemma~\ref{lemma:SemifieldHomAndInequalities} gives us that $\ph_{P}(a^m)<\ph_{P}(f^n)$. 
    So the element $\alpha=(a^m+f^n,f^n)$ is in $P$ but not $P'$.
    
\underline{Case 2}: $\ph_{P'}(f)>\ph_{P'}(a)^{m/n}>\bigpi\circ\ph_{P}(f)$. Analogously to the previous case, we find that $\alpha=(a^m+f^n,a^m)$ is in $P$ but not $P'$.
\end{proof}

\begin{example}
Let $\base$ be a sub-semifield of $\T$ and let $M=\N^n$ or $\Z^n$. Let $P\in\ContBase{\base}\base[\Mon]$ have defining matrix $C$. In this case, Proposition~\ref{lemma:MatrixDeterminesMaximalPrimeAbove} says  that the unique maximal prime in $\ContBase{\base}\base[\Mon]$ which contains $P$ has a defining matrix given by the first row of $C$.
\end{example}

\begin{coro}\label{cor:ExistMaxPrimeinCont}
Let $\base$ be a sub-semifield of $\T$, and let $A$ be a topological $\base$-algebra. For any prime congruence $P\in\Cont A$ there is a unique $P'\in\Cont A$ which is maximal with respect to inclusion and contains $P$.
\end{coro}\begin{proof}
By Proposition~\ref{lemma:MatrixDeterminesMaximalPrimeAbove}, $P'$ is uniquely determined as the congruence-kernel of $\bigpi\circ\ph_{P}$, for any choice of $\ph_{P}$.
\end{proof}


\begin{coro}\label{coro:PrimesInChainHaveSameIdealKernel}
Let $\base$ be a sub-semifield of $\T$ and let $A$ be a topological $\base$-algebra. If $P,P'\in\Cont A$ and $P\subseteq P'$ then $P$ and $P'$ have the same ideal-kernel.
\end{coro}\begin{proof}
Let $\check{P}$ be the unique maximal element of $\Cont A$ containing $P'$. Since $\check{P}$ is also the unique maximal element of $\Cont A$ containing $P$, Proposition~\ref{lemma:MatrixDeterminesMaximalPrimeAbove} gives us that, once we choose $\ph_{P}:A\to\RnLexSemifield{I}$, $\check{P}$ is the congruence-kernel of $A\toup{\ph_{P}}\RnLexSemifield{I}\toup{\diagbigpi}\T$. Since $\bigpi$ has trivial ideal-kernel, this gives us that the ideal-kernels of $P$ and $\check{P}$ are the same. Since $P'$ is in between $P$ and $\check{P}$, we get that $P$ and $P'$ have the same ideal-kernel.
\end{proof}

\begin{defi}
If $A$ is a topological semiring, we let $\ContInt A$  denote the set of those $P\in\Cont A$ which have trivial ideal-kernel. Similarly, if $\base$ is a topological semiring and $A$ is a $\base$-algebra we let $\ContIntBase{\base} A$ denote the set of those $P\in\ContBase{\base} A$ with trivial ideal-kernel. 
\end{defi}

\begin{remark}
We use the notation $\ContInt$ because, in the case where $\base$ is a subsemifield of $\T$ and $M$ is a toric monoid, we think of $\ContIntBase{\base}\base[\Mon]$ as corresponding to the complement of the toric boundary.
\end{remark}

\begin{coro}\label{coro:PrimesInChainInInterior}
Let $\base$ be a sub-semifield of $\T$ and let $A$ be a topological $\base$-algebra. If $P,P'\in\Cont A$ and $P\subseteq P'$ then $P\in\ContInt A$ if and only if $P'\in\ContInt A$.
\end{coro}\begin{proof}
This is immediate from Corollary~\ref{coro:PrimesInChainHaveSameIdealKernel}.
\end{proof}

\begin{coro}\label{coro:PrimesInSpaChainInducedMapOnResidueSemifields}
Let $\base$ be a sub-semifield of $\T$ and let $A$ be a topological $\base$-algebra. If $P,P'\in\Cont A$ and $P\subseteq P'$ then the induced map $A/P\to A/P'$, which makes the diagram
\comd{
&A\ar[ld]_{|\cdot|_{P}}\ar[rd]^{|\cdot|_{P'}}\\
A/P\ar[rr]&&A/P'
}
commute, extends to a map $\pi=\pi_{P',P}:\kappa(P)\to\kappa(P')$ making the diagram
\comd{
&A\ar[ld]_{|\cdot|_{P}}\ar[rd]^{|\cdot|_{P'}}\\
\kappa(P)\ar[rr]^{\pi}&&\kappa(P')
}
commute. Moreover, the map $\pi$ is surjective and has trivial ideal-kernel.
\end{coro}\begin{proof}
By Corollary~\ref{coro:PrimesInChainHaveSameIdealKernel}, $P$ and $P'$ have the same ideal-kernel, so the map $A/P\to A/P'$ has trivial ideal-kernel. Hence the map $A/P\to A/P'\into\kappa(P')$ sends all cancellative elements of $A/P$ to nonzero elements of $\kappa(P')$, which are therefore invertible. So by the universal property of the total semiring of fractions (see Definition~\ref{def: semiring-of-fractions}) this map $A/P\to\kappa(P')$ extends to a morphism $\kappa(P)\toup{\pi}\kappa(P')$. The corresponding diagram commutes because the image of $A\toup{|\cdot|_{P}}\kappa(P)$ is contained in $A/P$.

Surjectivity of $\pi$ follows from surjectivity of the map $A/P\to A/P'$. 
Because $\pi$ is a homomorphism of semifields, it has trivial ideal-kernel.
\end{proof}

\begin{remark}\label{rem:DianeMadeMeDoThis1}
We note that the proof above works to show the same conclusion any time $A$ is a semiring and $P\subseteq P'$ are prime congruences on $A$ with the same ideal-kernel.
\end{remark}

The following uses the above 
corollary 
to give a proof of the converse of Lemma~\ref{lemma:MaximalGeorgePrimes}.

\begin{proposition}\label{lemma:MaximalGeorgePrimesBack}
Let $\base$ be a sub-semifield of $\T$, $A$ be a topological $\base$-algebra, and  $P\in\Cont A$.
If there is an embedding $\kappa(P)\into\T$, then $P$ is maximal with respect to inclusion.
\end{proposition}

\begin{proof}
    Suppose $P' \supseteq P$ is in $\Cont A$.
    By Corollary~\ref{coro:PrimesInSpaChainInducedMapOnResidueSemifields} there is a surjection $\varphi: \kappa(P) \to \kappa(P')$. The congruence-kernel of $\varphi$ is a prime congruence on $\kappa(P)$. Since $\kappa(P)^\times$ has archimedean rank 1 (as it embeds into $\R$), the only convex subgroups are $\{1\}$ and $\kappa(P)^\times$. The later corresponds to a prime congruence with quotient $\B$ and the former corresponds to the diagonal $\Delta_{\kappa(P)}$. Since $P'\in\Cont A$,  $\ker\ph=\Delta_{\kappa(P)}$, so $P'=P$.
\end{proof}

We can use 
Remark~\ref{rem:DianeMadeMeDoThis1}
to prove a partial converse to Corollary~\ref{coro:PrimesInChainHaveSameIdealKernel} in a particular case. 
We first make a brief observation.



\begin{coro}
\label{coro:MapOnResidueSemifieldsAndInequalities}
Let $\base$ be a sub-semifield of $\T$, let $A$ be a topological $\base$-algebra, and let $P,P'\in\Cont A$ satisfying $P\subseteq P'$. Consider $a,b\in\kappa(P)$.
\begin{enumerate}
    \item If $a\leq b$ then $\pi_{P',P}(a)\leq\pi_{P',P}(b)$. 
    \item If $\pi_{P',P}(a)<\pi_{P',P}(b)$ then  $a<b$.
\end{enumerate}
\end{coro}\GeorgeStory{In which George goes to the preliminaries and back again.}
\begin{proof}
%
This is Lemma~\ref{lemma:SemifieldHomAndInequalities} applied to $\ph=\pi_{P',P}$.
\end{proof}

\begin{prop}\label{prop:GoingDownStaysInGeorge}
Let $\base$ be a sub-semifield of $\T$, and let $A$ be an $\base$-algebra. If $P'\in\ContBase{\base}A$ and $P\subseteq P'$ is a prime congruence on $A$ with the same ideal-kernel as $P'$, then $P\in\ContBase{\base}A$. In particular, if $P'\in\ContIntBase{\base}A$ then every prime congruence on $A$ which is contained in $P'$ is also in $\ContIntBase{\base}A$.
\end{prop}\begin{proof}
By Lemma \ref{lemma:ContinuityGivesGeorge}, we want to show that, given $a\in\kappa(P)^\times$, there is some $b\in\base^\times$ such that $|b|_{P}<a$. Since $P'\in\ContBase{\base} A$ and $\pi_{P',P}(a)\in\kappa(P')^\times$, there is some $b\in\base^\times$ such that $\pi_{P',P}(a)>|b|_{P'}=\pi_{P',P}(|b|_{P})$. So by part (2) of the previous lemma, $a>|b|_P$.
\end{proof}

In Section \ref{section:TheoremCrown} we will need to compare evaluations of a given semiring element at different primes. In order to do this, we prove a version of Proposition \ref{prop:HahnForSpaSemifields} for multiple prime congruences. First, we establish some notations.

Let $\base$ be a sub-semifield of $\T$ and let $A$ be a topological $\base$-algebra.
For any $P\in\Cont A$, we consider the set $I_P$ of archimedean classes of $\kappa(P)^\times$ and the semifield $\RnLexSemifield{I_P}$. Since $I_P$ has a least element $i_0(P)$, we have the canonical continuous homomorphisms $\bigpi:\RnLexSemifield{I_P}\to\T$ and $j:\T\to\RnLexSemifield{I_P}$. Recall that, if $P\in\Cont A$ is maximal, Lemma \ref{lemma:MaximalGeorgePrimes} gives us a canonical embedding $\kappa(P)\into\T$.

\begin{coro}\label{coro:BigDiagramForOnePrimeHahn}
Let $\base$ be a sub-semifield of $\T$ and let $A$ be a topological $\base$-algebra. Let $P\in\Cont A$ and consider the maximal $\check{P}\in\Cont A$ containing $P$. For any embedding $\ph:\kappa(P)\to\RnLexSemifield{I_P}$ as in Proposition \ref{prop:HahnForSpaSemifields}, the diagram
\comd{
\base\ar[r]\inj[d]&\kappa(P)\ar[r]\inj[d]_{\ph}&\kappa\left(\check{P}\right)\inj[d]\\
\T\ar[r]^{j}&\RnLexSemifield{I_P}\ar[r]^{\diagbigpi}&\T
}
commutes.
\end{coro}\begin{proof}
The commutativity of the square on the left follows from Proposition \ref{prop:HahnForSpaSemifields} and the commutativity of the square on the right follows from Proposition \ref{lemma:MatrixDeterminesMaximalPrimeAbove}.
\end{proof}

\begin{prop}\label{prop:HahnFor2SpaPrimes} 
Let $\base$ be a sub-semifield of $\T$ and let $A$ be a topological $\base$-algebra. Let $P,P'\in\Cont A$ and consider the maximal $\check{P}, \check{P}'\in\Cont A$ containing $P, P'$, respectively. Then there is a totally ordered set $I_{P,P'}$ with least element $i_0(P,P')$ and embeddings $\ph_P:\kappa(P)\into\RnLexSemifield{I_{P,P'}}$ and $\ph_{P'}:\kappa(P')\into\RnLexSemifield{I_{P,P'}}$ such that the diagram
\comd{
&&\kappa(P)\ar[r]\inj[d]_{\ph_P}&\kappa\left(\check{P}\right)\inj[d]\\
\base\inj[r]\ar[urr]\ar[drr]&\T\ar[r]^{j}&\RnLexSemifield{I_{P,P'}}\ar[r]^{\diagbigpi}&\T\\
&&\kappa(P')\ar[r]\linj[u]^{\ph_{P'}}&\kappa\left(\check{P}'\right)\linj[u]\\
}
commutes.
\end{prop}
\GeorgeStory{(In which George fights the whale and flies some kites.)}\begin{proof}
Write $I_P=\{i_0(P)\}\sqcup \breve{I}_P$ and $I_{P'}=\{i_0(P')\}\sqcup \breve{I}_{P'}$. Set $I_{P,P'}:=\{i_0(P,P')\}\sqcup\breve{I}_P\sqcup\breve{I}_{P'}$, and endow it with the total order indicated. The maps
\begin{multicols}{2}
\begin{align*}
I_P&\longrightarrow I_{P,P'}\\
i_0(P)&\longmapsto i_0(P,P')\\
\breve{I}_{P}&\toup{\id}\breve{I}_{P}\\
\end{align*}
\columnbreak

\begin{align*}
I_{P'}&\longrightarrow I_{P,P'}\\
i_0(P')&\longmapsto i_0(P,P')\\
\breve{I}_{P'}&\toup{\id} \breve{I}_{P'}
\end{align*}
\end{multicols}
\noindent induce injective homomorphisms $\RnLexSemifield{I_{P}}\to\RnLexSemifield{I_{P,P'}}$ and $\RnLexSemifield{I_{P'}}\to\RnLexSemifield{I_{P,P'}}$, which are continuous by Proposition~\ref{prop:ContinuityUsingLexico}. Moreover, these homomorphisms make the diagram
\comd{
&\RnLexSemifield{I_{P}}\ar[rd]^{\bigpi}\ar[d]\\
\T\ar[r]^{j}\ar[ru]^{j}\ar[rd]^{j}&\RnLexSemifield{I_{P,P'}}\ar[r]^{\bigpi}&\T\\
&\RnLexSemifield{I_{P'}}\ar[ru]_{\bigpi}\ar[u]
}
commute. Pick embeddings $\tilde{\ph}_P:\kappa(P)\into\RnLexSemifield{I_{P}}$ and $\tilde{\ph}_{P'}:\kappa(P')\into\RnLexSemifield{I_{P'}}$ as in Proposition \ref{prop:HahnForSpaSemifields}. Applying Corollary \ref{coro:BigDiagramForOnePrimeHahn} to each of these and using the previous commutative diagram, we see that the diagram
\comd{
&&\kappa(P)\inj[d]^{\tilde{\ph}_P} \ar[r]&\kappa\left(\check{P}\right)\inj[dd]\\
&&\RnLexSemifield{I_{P}}\ar[rd]^{\diagbigpi}\ar[d]\\
\base\ar[rruu]\ar[rrdd]\inj[r]&\T\ar[r]^{j}\ar[ru]^{j}\ar[rd]^{j}&\RnLexSemifield {I_{P,P'}}\ar[r]^{\diagbigpi}&\T\\
&&\RnLexSemifield{I_{P'}}\ar[ru]_{\diagbigpi}\ar[u]\\
&&\kappa(P')\linj[u]_{\tilde{\ph}_{P'}} \ar[r]&\kappa\left(\check{P}'\right)\linj[uu]\\
}
commutes. Thus, $\ph_P:\kappa(P)\toup{\tilde{\ph}_P}\RnLexSemifield{I_{P}}\to\RnLexSemifield{I_{P,P'}}$ and $\ph_{P'}:\kappa(P')\toup{\tilde{\ph}_{P'}}\RnLexSemifield{I_{P'}}\to\RnLexSemifield{I_{P,P'}}$ are as desired.
\end{proof}

\subsection{$\Cont$ as a topological space}\label{ch:ContTopSpace}

Let $A$ be a topological semiring. We now endow $\Cont A$ with the topology generated by the sets
$$\BasicOpen{f}{g}:=\{P\in\Cont A\ :\ |f|_P\leq|g|_P\neq0_{\kappa(P)}\}$$
for $f,g\in A$. This is analogous to the topology defined on the continuous spectrum of a ring; see \cite[Section 0]{Hub93} or \cite[Sections 4.1 and 7.2]{Wed12}. We give $\ContInt A\subseteq\Cont A$ the subspace topology.

\begin{lemma}
The sets $\BasicOpen{f}{g}$ form a basis for the topology on $\Cont A$.
\end{lemma}\begin{proof}
Note that $\BasicOpen{1}{1}=\Cont A$. So it suffices to show that, for any $f,f',g,g'\in A$, $\basicopen{f}{g}\cap\basicopen{f'}{g'}=\basicopen{\tilde{f}}{\tilde{g}}$ for some $\tilde{f},\tilde{g}\in A$.


Straightforward computation shows that $\basicopen{f}{g}\cap\basicopen{f'}{g'}=\basicopen{fg'}{gg'}\cap\basicopen{gf'}{gg'}$. Since $|fg'+gf'|_P=|fg'|_P+|gf'|_P$ is the maximum of $|fg'|_P$ and $|gf'|_P$ in $\kappa(P)$, we have that $\basicopen{fg'}{gg'}\cap\basicopen{gf'}{gg'}=\basicopen{fg'+gf'}{gg'}$.
\end{proof}

We now consider specializations in $\Cont A$.

\begin{prop}\label{prop:SpecializationsInCont}
Let $\base$ be a sub-semifield of $\T$ and let $A$ be a topological $\base$-algebra. For any $P,P'\in\Cont A$, $P'$ is a specialization of $P$ if and only if $P'\subseteq P$.
\end{prop}\begin{proof}
Note that $P'$ is a specialization of $P$ if and only if whenever $P'\in\basicopen{f}{g}$, we also have $P\in\basicopen{f}{g}$. Equivalently, $P'$ is a specialization of $P$ if and only if, for all $f,g\in A$, $|f|_{P'}\leq|g|_{P'}\neq0$ implies $|f|_{P}\leq|g|_{P}\neq0$.

If $P'\subseteq P$, then, by 
Corollaries~\ref{coro:PrimesInChainHaveSameIdealKernel} and \ref{coro:MapOnResidueSemifieldsAndInequalities}, 
whenever $|f|_{P'}\leq|g|_{P'}\neq0$ we have $|f|_{P}\leq|g|_{P}\neq0$, so $P'$ is a specialization of $P$.

Now suppose $P'$ is a specialization of $P$. Fix $(f,g)\in P'$; we wish to show that $(f,g)\in P$. If $|g|_{P'}\neq0$, then $|f|_{P'}\leq|g|_{P'}\neq0$ and $|g|_{P'}\leq|f|_{P'}\neq0$, so $|f|_{P}\leq|g|_{P}$ and $|g|_{P}\leq|f|_{P}$. Thus, $(f,g)\in P$. If, on the other hand, $|g|_{P'}=0$, we still have that, for all $\eps\in\base^\times$, $(f+\eps,g+\eps)\in P'$ and $|g+\eps|_{P'}\neq0$. So, by the previous case, $(f+\eps,g+\eps)\in P$. Since $\dlim_{\eps\to0}(f+\eps,g+\eps)=(f,g)$ and $P$ is closed, we get $(f,g)\in P$.
\end{proof}


\begin{coro}
Let $\base$ be a sub-semifield of $\T$ and let $A$ be a topological $\base$-algebra. Then $\Cont A$ is Kolmogorov, i.e., $\text{T}_0$.
\end{coro}\begin{proof}
If $P,P'\in\Cont A$ are topologically indistinguishable, then they are specializations of each other. By Proposition~\ref{prop:SpecializationsInCont}, 
this implies that $P=P'$.
\end{proof}

\subsection{$\Cont$ as a functor}

We now want to consider $\Cont A$ and its variations as functors. To do this, we need to define several categories. 
We denote the category of topological spaces as $\Top$.

\begin{defi}
Let $\base$ be a sub-semifield of $\T$. We define four categories as follows.

%
%
%

\begin{itemize}
\item The objects of $\Alg_\base$ are $\base$-algebras. The morphisms are all $\base$-algebra homomorphisms, i.e., semiring homomorphisms $A\to B$ which make the diagram
\comd{
A\ar[rr]&&B\\
&\base\ar[ur]\ar[ul]
}
commute.

\item The objects of $\TopAlg_\base$ are topological $\base$-algebras. The morphisms are all continuous $\base$-algebra homomorphisms.

\item The objects of $\AlgInt_\base$ are $\base$-algebras. The morphisms are those $\base$-algebra homomorphisms which have trivial ideal-kernel.

\item The objects of $\TopAlgInt_\base$ are topological $\base$-algebras. The morphisms are those continuous $\base$-algebra homomorphisms with trivial ideal-kernel.
\end{itemize}
\end{defi}

\begin{prop}\label{prop:Cont_is_functor}
Let $\base$ be a sub-semifield of $\T$. 
Then $\Cont:\TopAlg_\base\to\Top$ is a contravariant functor. 
\end{prop}\begin{proof}

Let $\varphi:A\to B$ be a morphism in $\TopAlg_\base$ and let $P \in \Cont B$. Then $\ph^*(P)$ is a prime congruence on $A$ and that $\ph$ induces an injective homomorphism $A/\ph^*(P)\into B/P$. So to get $\ph^*(P)\in\Cont A$ we just need to show that $A/\ph^*(P)\neq\B$ and the homomorphism $A\to\kappa(\ph^*(P))$ given by $f\mapsto|f|_{\ph^*(P)}$ is continuous.

The injective homomorphism $A/\ph^*(P)\to B/P$ induces an injective homomorphism on the residue semifields $\iota=\iota_P:\kappa(\ph^*(P))\into\kappa(P)$, which makes the diagram
\comd{
&\base\ar[dl]\ar[dr]\\
A\ar[rr]^{\ph}\ar[d]_{\pi_A}&&B\ar[d]^{\pi_B}\\
\kappa(\ph^*(P))\ar[rr]^<>(0.5){\iota}&&\kappa(P)
}
commute. Since $P\in\Cont B$, Lemma~\ref{lemma:BasicPropertiesOfPInSpaA} gives us that the map $\base\to B\to\kappa(P)$ is injective. By the commutativity of the diagram, this map is $\base\to A\to \kappa(\varphi^*(P))\into\kappa(P)$, so we get that $\base\to A\to \kappa(\varphi^*(P))$ is injective. In particular, $\kappa(\varphi^*(P))\neq\B$, so $A/\varphi^*(P)\neq\B$. 

Since $\ph$ and $\pi_B$ are continuous, $\iota\circ\pi_A=\pi_B\circ\ph$ is continuous. Since $\iota$ is an injective homomorphism of totally ordered semifields, Proposition~\ref{prop:ContinuityDependsOnlyOnImage} gives us that $\pi_A$ is continuous. Thus, $\ph^*(P)\in\Cont A$.

We have thus shown that $\ph:A\to B$ induces a function $\ph^*:\Cont B\to\Cont A$. To get that $\ph^*$ is continuous, it suffices to show that, for any $f,g\in A$, $(\ph^*)^{-1}(\basicopen{f}{g})$ is open in $\Cont B$. In fact, we claim that $(\ph^*)^{-1}(\basicopen{f}{g})=\basicopen{\ph(f)}{\ph(g)}$. To see this, consider $P\in\Cont B$. Note that $P\in(\ph^*)^{-1}(\basicopen{f}{g})$ if and only if $|f|_{\ph^*(P)}\leq|g|_{\ph^*(P)}\neq0_{\kappa(\ph^*(P))}$. Since $\iota_P:\kappa(\ph^*(P))\to\kappa(P)$ is an injective homomorphism of totally ordered semifields, $|f|_{\ph^*(P)}\leq|g|_{\ph^*(P)}\neq0_{\kappa(\ph^*(P))}$ if and only if $\iota_P\left(|f|_{\ph^*(P)}\right)\leq\iota_P\left(|g|_{\ph^*(P)}\right)\neq0_{\kappa(P)}$. By the commutativity of the diagram above, this happens if and only if $|\ph(f)|_P\leq|\ph(g)|_P\neq0_{\kappa(P)}$, i.e., if and only if $P\in\basicopen{\ph(f)}{\ph(g)}$. Thus, $(\ph^*)^{-1}(\basicopen{f}{g})=\basicopen{\ph(f)}{\ph(g)}$ and so $\ph^*$ is continuous. 

The verification that the map $\ph\mapsto\ph^*$ makes $\Cont$ into a contravariant functor is routine, and is left to the reader.
\end{proof}

%

\begin{prop}\label{ContInt_is_functor}
Let $\base$ be a sub-semifield of $\T$. Then $\ContInt:\TopAlgInt_\base\to\Top$ is a functor.
\end{prop}\begin{proof}
By restricting $\Cont$ to $\TopAlgInt_\base$ we get a functor $\Cont:\TopAlgInt_\base\to\Top$. It suffices to show that, if $\ph:A\to B$ is a morphism in $\TopAlgInt_\base$ and $P\in\ContInt B$ then $\ph^*(P)\in\ContInt A$. Since we already have $\ph^*(P)\in\Cont A$, we just need to show that $\ph^*(P)$ has trivial ideal-kernel.

Consider the commutative diagram
\comd{
A\ar[r]^{\ph}\ar[d]_{\pi_A}&B\ar[d]^{\pi_B}\\
\kappa(\ph^*(P))\ar[r]^<>(0.5){\iota}&\kappa(P).
}
Since $\ph^*(P)$ is the congruence-kernel of $\pi_A$, we want to show that the ideal-kernel of $\pi_A$ is trivial. So suppose $\pi_A(a)=0$ for some $a\in A$. Then $\pi_B(\ph(a))=\iota(\pi_A(a))=\iota(0)=0$. Since $P\in\ContInt B$ the ideal-kernel of $\pi_B$ is trivial, so $\ph(a)=0$. But, as $\ph$ is a morphism of $\TopAlgInt_\base$, the ideal-kernel of $\ph$ is trivial. So $a=0$.
\end{proof}

\begin{prop}
Let $\base$ be a sub-semifield of $\T$. Then $\ContBase{\base}:\Alg_{\base}\to\Top$ is a functor and $\ContIntBase{\base}:\AlgInt_\base\to\Top$ is a sub-functor of $\ContBase{\base}:\Alg_{\base}\to\Top$.
\end{prop}\begin{proof}
For any $\base$-algebra $A$, let $\tau_A$ be the universal topology on $A$ defined in Theorem \ref{thm:UniversalTopOnSemiring} and applied to the function $\base\to A$. It suffices to show that $A\mapsto(A,\tau_A)$ is a functor $\Alg_\base\to\TopAlg_\base$ with the map on morphisms being the identity map, for then it is also a functor $\AlgInt_\base\to\TopAlgInt_\base$, $\ContBase{\base}$ is the composition $\Alg_\base\toup{(\bullet,\tau_{\bullet})}\TopAlg_\base\toup{\Cont}\Top$, and $\ContIntBase{\base}$ is the composition
$\AlgInt_\base\toup{(\bullet,\tau_{\bullet})}\TopAlgInt_\base\toup{\ContInt}\Top$.

Say $\ph:A\to B$ is a $\base$-algebra homomorphism. Since $\base\to(A,\tau_A)\toup{\ph}(B,\tau_B)$ is the structure map of $(B,\tau_B)$ as a topological $\base$-algebra, it is continuous. So by the defining property of $\tau_A$, the map $(A,\tau_A)\to(B,\tau_B)$ is continuous.
\end{proof}

\begin{prop}\label{prop:PullbackFromContToContBase}
Let $\base$ be a sub-semifield of $\T$ and let $\ph:A\to B$ be a $\base$-algebra homomorphism with $B$ a topological $\base$-algebra. Then $\ph$ induces a continuous map $\ph^*:\Cont B\to\ContBase{\base} A$. Moreover, if $\ph$ has trivial ideal-kernel, then it also induces a continuous map $\ph^*:\ContInt B\to\ContIntBase{\base} A$.
\end{prop}\begin{proof}
Note that $\Cont B\subseteq\ContBase{\base} B$. Indeed, if $B\to\kappa(P)$ is continuous then the composition $\base\to B\to\kappa(P)$ is continuous. So the map $\ph^*:\Cont B\to\ContBase{\base} A$ is just the restriction of $\ph^*:\ContBase{\base} B\to \ContBase{\base} A$. The proof for $\ph^*:\ContInt B\to\ContIntBase{\base} A$ is the same.
\end{proof}


\section{The semiring $\Cnvg{P}$ of power series convergent at a prime $P$}\label{sect: Cnvg}

Let $\base\neq\B$ be a totally ordered semifield and let $\Mon$ be a monoid. 
We 
define the set of series with coefficients in $\base$ and exponents in $\Mon$ to be 
$$\GeneralSeries{\base}{\Mon}:= \left\{ f = \sum_{u\in \Mon} f_{u} \chi^{u}\ :\  f_{u} \in \base \right\}.$$
If $\base=\T$ and $\Mon=\Z^n$, we view this as an infinite analogue of the tropical Laurent polynomial semiring. However, unlike the set of tropical Laurent polynomials, this set of series does not form a semiring. It does, however, form an $\T$-semimodule with term-by-term addition.
On the other hand, if $\Mon=\N^n$ then $\GeneralSeries{\base}{\Mon}=\GeneralSeries{\base}{x_1,\ldots,x_n}$ is an $\base$-algebra. If $\base=\T$ then this semiring is called the semiring of \textit{tropical power series}.

\begin{definition}\label{def:convergent}
%
%
Let $\base(\neq\B)$ be a totally ordered semifield and $\Mon$ be a monoid. Let $P$ be a prime congruence on $\base[\Mon]$. A series $f=\sum_{u\in \Mon} f_{u} \chi^{u} \in \GeneralSeries{\base}{\Mon}$ {\it converges} at $P$ if for any $b \in \base[\Mon]/P$, $b \neq 0_{\base}$, the set 
$$\{ u\in \Mon \;:\; |f_{u}\chi^{u}|_P \geq b \}$$
is finite.
We denote the 
set of series that converge at $P$ 
by $\Cnvg{P}$.
\end{definition}

By definition, every polynomial converges at $P$. In general, the set $\Cnvg{P}$ contains infinite sums, as in the following examples. 

\begin{example}
Let $\base = \TT$, $\Mon = \ZZ$ and let $P$ be the prime congruence with defining matrix $\begin{pmatrix} 1&0 \end{pmatrix}$. Then 
$$\Cnvg{P} = \left\{ f = \sum_{u\in \Mon} f_{u} \chi^{u} \;:\; f_u \to 
0_{\T}
\text{ as } u \to \pm \infty \right\}.$$\exEnd\end{example}

\begin{example}
Let $\base = \TT$, $\Mon = \N$ and let $P$ be the prime congruence with defining matrix $\begin{pmatrix} 1& -\infty \end{pmatrix}$. Then $\Cnvg{P} =  \GeneralSeries{\T}{x}$, the semiring of tropical power series. \exEnd\end{example}

\begin{example}
Let $\base = \TT$, $\Mon = \N^n$ and let $P$ be the prime congruence with defining matrix $\begin{pmatrix} 1& 0&\cdots&0 \end{pmatrix}$. Then 
$$\Cnvg{P} =\left\{\dsum_{u\in\N^n}f_{u}x^{u} \in\GeneralSeries{\T}{x_1,\ldots,x_n}\;:\;f_u\to0_{\T}\text{ as } |u|\to\infty \right\}$$ 
is the \emph{tropical Tate algebra}. \exEnd\end{example}

Note that, in general, $\Cnvg{P}$ is an $\base$-sub-semimodule of $\GeneralSeries{\base}{\Mon}$. 
\begin{lemma}\label{lemma:CnvgInChain}
Let $\base$ be a totally ordered semifield and let $\Mon$ be a monoid. Let $A=\base[\Mon]$ and suppose that $P,P'\in\ContBase{\base} A$ satisfy $P\subseteq P'$. Then $\cnvg{P}=\cnvg{P'}$.
\end{lemma}\begin{proof}
Since $P\subseteq P'$, Corollary~\ref{coro:MapOnResidueSemifieldsAndInequalities} gives us that, for $g,h\in A$, if $|g|_{P}\geq|h|_{P}$ then $|g|_{P'}\geq|h|_{P'}$, and if $|g|_{P'}>|h|_{P'}$ then $|g|_{P}>|h|_{P}$. 

Let $f\in\GeneralSeries{\base}{\Mon}$ be a series $f=\dsum_{u\in \Mon}f_u \chi^u$ that converges at $P$. For any $b\in(A/P')\sdrop\{0_{A}\}\subseteq\kappa(P')^\times$, by Lemma~\ref{lemma:ContinuityGivesGeorge} we know that there is some $a\in \base^\times$ such that $|a|_{P'}<b$. So
\begin{align*}
\{u\in M\;:\;|f_u\chi^u|_{P'}\geq b\}&\subseteq\{u\in M\;:\; |f_u\chi^u|_{P'}>|a|_{P'}\}\\
&\subseteq\{u\in M\ :\ |f_u\chi^u|_{P}>|a|_{P}\},
\end{align*}
which is finite by the convergence of $f$ at $P$.

Suppose that $f$ converges at $P'$. If $b\in(A/P)\sdrop\{0_A\}\subseteq\kappa(P)^\times$, by Lemma \ref{lemma:BasicPropertiesOfPInSpaA} there is some $a\in \base^\times$ with $|a|_{P}<b$. Then 
\begin{align*}
\{u\in \Mon\;:\; |f_u \chi^u|_{P}\geq b\} &\subseteq\{u\in \Mon\;:\; |f_u \chi^u|_{P}\geq |a|_{P}\}\\
&\subseteq\{u\in M\;:\;|f_u \chi^u|_{P'}\geq |a|_{P'}\},
\end{align*} and because $|a|_{P'}\neq 0_{A/P'}$, the convergence of $f$ at $P'$ gives us that this last set is finite.
\end{proof}

\begin{lemma}\label{lemma:CnvgDefWithResidueSemifield}
Let $\base$ be a totally ordered semifield, let $\Mon$ be a monoid, and let $P\in\ContBase{\base}\base[\Mon]$. Then $f=\sum_{u\in\Mon}f_u\chi^u\in\GeneralSeries{\base}{\Mon}$ converges at $P$ if and only if for every element $b\in\kappa(P)^\times$ the set $\{u\in\Mon\,:\,|f_u\chi^u|_{P}\geq b\}$ is finite.
\end{lemma}\begin{proof}
Since the map $\base[\Mon]/P\to\kappa(P)$ is injective, the ``only if'' direction is clear. Let $f$ converge at $P$. Given $b\in\kappa(P)^\times$, there is some $s\in \base^\times$ such that $|s|_{P}\leq b$. This implies that $\{ u\in \Mon \,:\, |f_{u}\chi^{u}|_P \geq b \}\subseteq \{ u\in \Mon \,:\, |f_{u}\chi^{u}|_P \geq |s|_{P} \}$, but the bigger set is finite because $|s|_P\in(\base[\Mon]/P)\sdrop\{0_{\base}\}$.
\end{proof}

\begin{lemma}\label{lemma:AbsoluteValueOfConvergentSeriesMakesSense}
Let $\base$ be a totally ordered semifield, let $\Mon$ be a monoid, and let $P$ be a prime on $\base[\Mon]$. For any $f=\dsum_{u\in \Mon}f_u\chi^u\in\Cnvg{P}$ and any nonempty subset $X\subseteq\Mon$, 
there is a maximal element $\dsum_{u\in X}|f_u\chi^u|_{P}$ of the set $\{|f_u\chi^u|_{P}\,:\,u\in X\}$. 
Furthermore, the set $\{v\in X\,:\,|f_v\chi^v|_{P}=\dsum_{u\in X}|f_u\chi^u|_{P}\}$ is finite.
\end{lemma}\begin{proof}
Fixing $u_0\in X$, let $b=|f_{u_0}\chi^{u_0}|_{P}$. Then the set $X':=\{u\in X\;:\; |f_u\chi^u|_{P}\geq b\}$ is finite and nonempty. 
Thus, the largest of the finitely many values $|f_u\chi^u|_{P}$ for $u\in X'$ is also the largest value $\dsum_{u\in X}|f_u\chi^u|_{P}$ of $|f_u\chi^u|_{P}$ for $u\in X$. 
Finally, $\{v\in X\,:\,|f_v\chi^v|_{P}=\dsum_{u\in X}|f_u\chi^u|_{P}\}$ is finite because it is contained in $X'$.
\end{proof}

Given two series $f=\dsum_{u\in \Mon}f_u\chi^u$ and $g=\dsum_{u\in \Mon}g_u\chi^u$ in $\GeneralSeries{\base}{\Mon}$, we say that \emph{$f\cdot g$ exists} if, for each $u\in \Mon$, the set $\{f_vg_w\,:\, v,w\in M$ and $v+w=u\}$ has a greatest element. In this case, we will denote by $\dsum_{\substack{v,w\in \Mon\\v+w=u}}f_vg_w$ this maximum, and define $$f\cdot g:=\dsum_{u\in \Mon}\left(\dsum_{\substack{v,w\in \Mon\\v+w=u}}f_vg_w\right)\chi^u\in\GeneralSeries{\base}{\Mon}.$$


\begin{proposition}
Let $\base$ be a sub-semifield of $\T$ and let $\Mon$ be a 
toric monoid.
If $P \in \ContIntBase{\base} \base[\Mon]$ then $\CLS{P}$ is a semiring with the multiplication defined above.
\end{proposition}
\GeorgeStory{In which George rides camels.}
\begin{proof}

Given $f,g\in\Cnvg{P}$, we first show that $f\cdot g$ exists. Fix $u\in \Mon$. Let $$W:=\{v\in\Mon\,:\,(\exists w\in\Mon)\ u=v+w\}.$$ Note that $W$ is non-empty because $0_{\Mon}\in W$.
Because $f,g\in\Cnvg{P}$, there are largest values of $|f_v\chi^v|_{P}$ and $|g_v\chi^v|_{P}$ for $v\in W$. Furthermore, the sets 
\begin{align*}
    V_{f}&:=\{v\in W\,:\,|f_v\chi^v|_{P}\text{ is maximized}\} \text{ and } \\ 
    V_{g}&:=\{v\in W\,:\,|g_v\chi^v|_{P}\text{ is maximized}\}
\end{align*}
are finite. By the hypothesis on $\Mon$, the sets $\{v\in\Mon\,:\,u\in v+V_{f}\}$ and $\{v\in\Mon\,:\,u\in v+V_{g}\}$ are finite. 

Note that the existence of a greatest element of $\{f_vg_w\,:\, v,w\in \Mon$ and $v+w=u\}$ is unchanged if we replace $f$ and $g$ by $af$ and $bg$, respectively, for any fixed $a,b\in\base^\times$. 
Because $P\in\ContIntBase{\base} \base[\Mon]$, we can find $s_f,s_g\in\base^\times$ with $|s_f|_{P}\leq\min\{|f_v\chi^v|_{P}\,:\,u\in v+V_{g}\}$ and $|s_g|_{P}\leq\min\{|g_v\chi^v|_{P}\,:\,u\in v+V_{f}\}$. By replacing $f$ and $g$ with $s_f^{-1}f$ and $s_g^{-1}g$, respectively, we may assume without loss of generality that $\min\{|f_v\chi^v|_{P}\,:\,u\in v+V_{g}\}\geq 1_{\kappa(P)}$ and $\min\{|g_v\chi^v|_{P}\,:\,u\in v+V_{f}\}\geq 1_{\kappa(P)}$. In particular, $\mu_f:=\sum_{v\in W}|f_v\chi^v|_{P}\geq 1_{\kappa(P)}$ and $\mu_g:=\sum_{v\in W}|g_v\chi^v|_P\geq 1_{\kappa(P)}$.

Consider the set 
$$F:=\{(v,w)\in\Mon^2\,:\, v+w=u, |f_v\chi^v g_{w}\chi^w|_{P}\geq\mu_f+\mu_g\}.$$
Recall that $\mu_f+\mu_g$ is the larger of $\mu_f$ and $\mu_g$. 
We claim that $F$ is a finite set. To see this, note that if $|f_v\chi^v|_{P}<1_{\kappa(P)}$ then $|f_v\chi^v g_{w}\chi^w|_{P}<\mu_g$ for any $w\in\Mon$ such that $u=v+w$. Similarly, if $|g_w\chi^w|_{P}<1_{\kappa(P)}$ then $|f_v\chi^v g_{w}\chi^w|_{P}<\mu_f$ for any $v\in\Mon$ such that $u=v+w$. Taking the contrapositives of these statements, we find that 
$$F\subseteq\{v\in\Mon\,:\,|f_v\chi^v|_{P}\geq 1_{\kappa(P)}\}\times \{w\in\Mon\,:\,|g_w\chi^w|_{P}\geq 1_{\kappa(P)}\},$$
which is finite because $f,g\in\Cnvg{P}$.

We now claim that $F$ is nonempty. To see this, pick $v_0\in V_{f}$ and $w_0\in V_g$. In particular, $|f_{v_0}\chi^{v_0}|_{P}=\mu_f$ and  $|g_{w_0}\chi^{w_0}|_{P}=\mu_g$. Since $v_0,w_0\in W$, there are $v_1,w_1\in\Mon$ such that $v_0+w_1=u$ and $v_1+w_0=u$. 
Since 
\begin{align*}
    &\min\{|f_v\chi^v|_{P}\,:\,u\in v+V_{g}\}\geq 1_{\kappa(P)} \text{ and} \\
    &\min\{|g_w\chi^w|_{P}\,:\,u\in w+V_{f}\}\geq 1_{\kappa(P)},
\end{align*}
%
we have $|f_{v_1}\chi^{v_1}|_{P}\geq 1_{\kappa(P)}$ and $|g_{w_1}\chi^{w_1}|_{P}\geq 1_{\kappa(P)}$. Thus,
\begin{align*}
    &|f_{v_0}\chi^{v_0} g_{w_1}\chi^{w_1}|_{P}\geq\mu_f \text{ and} \\ 
    &|f_{v_1}\chi^{v_1} g_{w_0}\chi^{w_0}|_{P}\geq\mu_g,
\end{align*}
so at least one of $(v_0,w_1)$ and $(v_1,w_0)$ is in $F$.

Thus, we conclude that $\dsum_{\substack{v,w\in \Mon\\v+w=u}}|f_v\chi^v g_w\chi^w|_{P}$ exists in $\kappa(P)$. For $v,w\in \Mon$ with $v+w=u$ we have $|f_v\chi^v g_w\chi^w|_{P}=|f_v g_w|_{P}\cdot |\chi^u|_{P}$. Since $\kappa(P)$ is a totally ordered semifield and since $|\chi^u|_{P}\neq 0_{\kappa(P)}$, this gives us that $\dsum_{\substack{v,w\in \Mon\\v+w=u}}|f_v g_w|_{P}$ exists. Because the map $\base\to\base[\Mon]\toup{|\cdot|_{P}}\kappa(P)$ is injective and each $f_v g_{w}$ is in $\base$, this implies that $\dsum_{\substack{v,w\in \Mon\\v+w=u}}f_vg_w$ exists in $\base$. This completes the proof that $f\cdot g$ exists in $\GeneralSeries{\base}{\Mon}$.

We now need to show that $f\cdot g\in\Cnvg{P}$. For each $u\in\Mon$ let $c_u=\dsum_{\substack{v,w\in \Mon\\v+w=u}}f_vg_w$, so $f\cdot g=\dsum_{u\in\Mon}c_u\chi^u$. Let $m_f=\dsum_{u\in \Mon}|f_u\chi^u|_{P}$ and $m_g=\dsum_{u\in\Mon}|g_u\chi^u|_{P}$. If $m_f$ or $m_g$ is $0_{\kappa(P)}$ then $f$ or $g$, respectively, is the zero series, so $f\cdot g=0\in\Cnvg{P}$. So from now on we may assume that $m_f,m_g\neq 0_{\kappa(P)}$. Fix a nonzero $b\in \base[\Mon]/P$. 
Set $a=\frac{b}{m_f+m_g}$, where each of these operations takes place in $\kappa(P)$. 
For any $u\in\Mon$ with $$b\leq|c_u\chi^u|_{P}=\dsum_{\substack{v,w\in\Mon\\v+w=u}}|f_v\chi^v g_w\chi^w|_{P}$$ there are $v_0,w_0\in\Mon$ such that $v_0+w_0=u$ and $|f_{v_0}\chi^{v_0} g_{w_0}\chi^{w_0}|_{P}\geq b$. Then $$b\leq|f_{v_0}\chi^{v_0} g_{w_0}\chi^{w_0}|_{P}\leq |f_{v_0}\chi^{v_0}|_{P} m_g\leq |f_{v_0}\chi^{v_0}|_{P}(m_f+m_g),$$ so $|f_{v_0}\chi^{v_0}|_{P}\geq a$. Similarly, we get that $|g_{w_0}\chi^{w_0}|_{P}\geq a$. Thus, 
$$\{u\in\Mon\,:\,|c_u\chi^u|_{P}\geq b\}\subseteq\{v\in\Mon\,:\,|f_v\chi^v|_{P}\geq a\}+\{w\in\Mon\,:\,|g_w\chi^w|_{P}\geq a\},$$
and Lemma \ref{lemma:CnvgDefWithResidueSemifield} applied to both $f$ and $g$ gives us that the set on the right hand side is finite.

The proof that the operations $+$ and $\cdot$ on $\Cnvg{P}$ satisfy the axioms for a semiring are routine and are directly analogous to the proof that the usual operations on power series over a ring make the set of power series into a ring.
\end{proof}

\begin{remark}
The only property of toric monoids (as opposed to arbitrary monoids) that was used in this proof is that for all $u,w\in\Mon$, the set $\{v\in\Mon\ :\ u=v+w\}$ is finite. In fact, Section~\ref{sect: geom-crown} is the first point any other property of toric monoids used - all of the results before that point that use toric monoids are also true under this weak cancellativity hypothesis.
\end{remark}

Let $\base$ be a totally ordered semifield, let $\Mon$ be a monoid, and let $P$ be a prime congruence on $\base[\Mon]$. For any $f=\dsum_{u\in\Mon}f_u\chi^u\in\Cnvg{P}$, we write $$|f|_{P}:=\dsum_{u\in\Mon}|f_u\chi^u|_{P}\in \base[\Mon]/P\subseteq\kappa(P).$$ It is routine to check that whenever the multiplication defined above makes $\Cnvg{P}$ into a semiring, the resulting map $\Cnvg{P}\toup{|\cdot|_{P}}\kappa(P)$ is a semiring homomorphism.


\section{The metric topology on $\Cnvg{P}$}\label{sect: topology-Cnvg}

Let $\base$ be a sub-semifield of $\T$ and, as usual, we will assume that $\base$ is infinite, i.e., not $\BB$. Let $\Mon$ be a monoid, and let $P$ be a prime congruence on $\base[\Mon]$. We define a (generalized) ultrapseudometric on $\Cnvg{P}$ as follows. Given $f=\dsum_{u\in\Mon}f_u\chi^u$ and $g=\dsum_{u\in\Mon}g_u\chi^u$, we set
$$d(f,g)=d_{P}(f,g):=\dsum_{\substack{u\in\Mon\\f_u>g_u}}\left|f_u\chi^u\right|_{P}+\dsum_{\substack{u\in\Mon\\g_u>f_u}}\left|g_u\chi^u\right|_{P}\in\base[\Mon]/P\subseteq\kappa(P),$$
which makes sense by Lemma \ref{lemma:AbsoluteValueOfConvergentSeriesMakesSense}. 
Note that $d_{P}(f,g)$ is the maximum of the values $|f_u\chi^u|_{P}+|g_u\chi^u|_{P}$ for $u\in\Mon$ such that $f_u\neq g_u$, so we can also write
$$d_{P}(f,g)=\dsum_{\substack{u\in\Mon\\f_u\neq g_u}}|f_u\chi^u|_{P}+|g_u\chi^u|_{P}=\dsum_{\substack{u\in\Mon\\f_u\neq g_u}}|(f_u+g_u)\chi^u|_{P}
.$$ 
Here we use the fact that $|\cdot|_P$ is a semiring homomorphism and so respects addition.
We say that $d_{P}:\Cnvg{P}\times\Cnvg{P}\to\kappa(P)$ is an ultrapseudometric because, for all $f,g,h\in\Cnvg{P}$,
\begin{align*}
    d_{P}(f,g)&\geq0_{\kappa(P)},\\
    d_{P}(f,f)&=0_{\kappa(P)},\\
    d_{P}(f,g)&=d_{P}(g,f),\\
    \intertext{and}
    d_{P}(f,h)&\leq d_{P}(f,g)+d_{P}(g,h)=\max(d_{P}(f,g),d_{P}(g,h)).
\end{align*}
If $P$ has trivial ideal-kernel, then $d_P$ is an ultrametric, in the sense that, in addition to the above properties, for any $f,g\in\Cnvg{P}$, $d_{P}(f,g)=0_{\kappa(P)}$ if and only if $f=g$.

\begin{remark}
The above definition is motivated as follows. 
The distance $d_{P}(f,g)$ is the smallest value of $|f_1|_{P}+|g_1|_{P}$ over all $f_1,g_1,h\in\Cnvg{P}$ such that $f=h+f_1$ and $g=h+g_1$. This is realized by choosing $$h=\dsum_{u\in\Mon}\min(f_u,g_u)\chi^u, \ f_1=\dsum_{\substack{u\in\Mon\\f_u>g_u}}f_u\chi^u,\text{ and } g_1=\dsum_{\substack{u\in\Mon\\g_u>f_u}}g_u\chi^u.$$
Thus, the definition of $d_P(f,g)$ gives a way of replicating what ``$|f-g|_P$'' would be if it made sense.

Alternatively, we can motivate this definition in the following way. Suppose that $K$ is a field with a nontrivial (non-archimedean) valuation $v:K\onto\base$ and $\Mon$ is finitely generated (e.g.\ toric). We extend $v$ to a map $v_1:K[\Mon]\onto\base[\Mon]$ defined by taking the valuation of coefficients, and composing with $|\cdot|_{P}:\base[\Mon]\to\kappa(P)$ gives a norm $|\cdot|$ on $K[\Mon]$. The completion $K\langle\Mon\rangle_{P}$ of $K[\Mon]$ with respect to this norm admits a natural map $v_2:K\langle\Mon\rangle_{P}\onto\Cnvg{P}$ extending $v_1:K[\Mon]\onto\base[\Mon]$. (When $\Mon=\N^n$ and $P$ is given by the matrix $\begin{pmatrix}1&0&\cdots&0\end{pmatrix}$, this is coefficient-wise valuation from a Tate algebra to a tropical Tate algebra.) Given $f,g\in\Cnvg{P}$, 
$d_P(f,g)=\displaystyle\inf_{\substack{F\in v_2^{-1}(f)\\G\in v_2^{-1}(g)}}|F-G|$, as is realized by picking $H\in v_2^{-1}(h)$, $F_1\in v_2^{-1}(f_1)$, and $G_1\in v_2^{-1}(g_1)$ and letting $F=H+F_1$ and $G=H+G_1$. Because of this, we can say that $d_P$ is the quotient metric on $\Cnvg{P}$ as a quotient of $K\langle\Mon\rangle_P$; the above infimum is a direct analog of the definition of a quotient norm.

\end{remark}

In the same way that a usual metric induces a topology, $d_P$ induces a topology on $\Cnvg{P}$. Similarly, the definition of a Cauchy sequence transfers over to our setup in the obvious way.

\begin{thm}\label{thm:CnvgIsComplete}
Let $\base$ be a sub-semifield of $\T$, let $\Mon$ be a monoid, and let $P$ be a prime on $\base[\Mon]$. With respect to the pseudometric $d=d_P$, $\Cnvg{P}$ is complete. That is, every Cauchy sequence in $\Cnvg{P}$ converges.
\end{thm}\begin{proof}
Let $(f_i)_{i\in\N}$ be a Cauchy sequence in $\Cnvg{P}$ and write $f_i=\dsum_{u\in\Mon}f_{i,u}\chi^u$ for each $i\in\N$. So for each $\eps\in\kappa(P)^\times$ we can fix some $N_{\eps}\in\N$ such that for all $i,j\geq N_\eps$, $d(f_i,f_j)<\eps$.

Fix $\eps\in\kappa(P)^\times$ and $u\in\Mon$. If there were some $i,j\geq N_\eps$ such that $|f_{i,u}\chi^u|_{P}\geq\eps$ and $f_{i,u}\neq f_{j,u}$, then $$d(f_i,f_j)=\dsum_{\substack{w\in\Mon\\f_{i,w}\neq f_{j,w}}}|f_{i,w}\chi^w|_{P}+|f_{j,w}\chi^w|_{P}\geq|f_{i,u}\chi^u|_{P}+|f_{j,u}\chi^u|_{P}\geq\eps,$$
which is impossible. So for every $i\geq N_\eps$, either $|f_{i,u}\chi^u|_{P}<\eps$ or $f_{j,u}=f_{i,u}$ for every $j\geq N_{\eps}$. In the second case, we get that $f_{j,u}=f_{N_\eps,u}$. Thus, either $|f_{i,u}\chi^u|_{P}<\eps$ for all $i\geq N_\eps$ or $f_{j,u}=f_{N_\eps,u}$ for every $j\geq N_\eps$.

In particular, if $u\in\Mon$ is such that the sequence $(f_{i,u})_{i\in\N}$ does not stabilize, then, for every $\eps\in\kappa(P)^\times$ and $i\geq N_{\eps}$, $|f_{i,u}\chi^u|_P<\eps$. For any $u\in\Mon$, we define $F_u\in\base$ to be the eventual value of $(f_{i,u})_{i\in\N}$ if the sequence stabilizes, and $0_S$ otherwise. Set $F:=\dsum_{u\in\Mon}F_u\chi^u\in\GeneralSeries{\base}{\Mon}$.
We want to show that $F$ converges at $P$. Fix $b\in(\base[\Mon]/P)\sdrop\{0_S\}$ and let $L:=\{u\in\Mon:|F_u\chi^u|_P\geq b\}$. Consider $u\in L$. Then $F_u\neq0_S$, 
so for all sufficiently large $i$, $f_{i,u}=F_u$, which gives us that $|f_{i,u}\chi^u|_P\geq b$. Thus, for all $j\geq N_b$, $f_{j,u}=f_{N_b,u}$, so $F_u=f_{N_b,u}$. In particular, this shows that $L\subset\{u\in\Mon:|f_{N_b,u}\chi^u|_P\geq b\}$, which is finite because $f_{N_b}\in\Cnvg{P}$. So $F$ converges at $P$.

Now we show that $(f_i)_{i\in\N}$ converges to $F$. Fix $\eps\in\kappa(P)^\times$ and consider $j\geq N_\eps$. For $u\in\Mon$, if $f_{j,u}\neq F_u$ then $|f_{i,u}\chi^u|_P<\eps$ for all $i\geq N_\eps$. In particular, $|f_{j,u}\chi^u|_P<\eps$ and $|F_u\chi^u|_P<\eps$. Thus, $$d(f_j,F)=\dsum_{\substack{u\in\Mon\\f_{j,u}\neq F_u}}|f_{j,u}\chi^u|_{P}+|F_u\chi^u|_{P}<\eps,$$ for all $j\geq N_\eps$, so $(f_i)_{i\in\N}$ converges to $F$.
\end{proof}

\begin{remark}
Our notion of completeness of $\Cnvg{P}$ agrees with a more general notion of completeness. Namely, in the same way that a usual metric induces a \emph{uniform structure}, $d_P$ induces a uniform structure on $\Cnvg{P}$. Then, upon replacing sequences with nets, the above proof goes through to show that the uniform structure on $\Cnvg{P}$ is complete, in the sense that every Cauchy net converges.
\end{remark}

Even when two distinct primes have $\Cnvg{P}=\Cnvg{P'}$, the metrics $d_P$ and $d_{P'}$ are distinct. However, the following theorem shows that the two metrics induce the same topology.

\begin{thm}\label{thm:TopologiesOnCnvgCoincide}
Let $\base$ be a sub-semifield of $\T$, let $\Mon$ be a monoid, and let $P,P'\in\ContBase{\base} \base[\Mon]$ be such that $P\supseteq P'$. Then the topologies on $\Cnvg{P}=\Cnvg{P'}$ induced by $d_P$ and $d_{P'}$ are the same.
\end{thm}
\GeorgeStory{In which George meets Borpf and Borqg.}
\begin{proof}

We want to show that the identity maps $(\Cnvg{P'},d_{P'})\to(\Cnvg{P},d_{P})$ and \newline $(\Cnvg{P},d_{P})\to(\Cnvg{P'},d_{P'})$ are continuous. In fact, it is just as easy to show that they are uniformly continuous. 

Let $\pi=\pi_{P,P'}:\kappa(P')\to\kappa(P)$ be the induced surjective morphism on residue semifields as in Corollary \ref{coro:PrimesInSpaChainInducedMapOnResidueSemifields}. Note that, for any $f,g\in\Cnvg{P}$, $$d_P(f,g)=\dsum_{\substack{u\in\Mon\\f_u\neq g_u}}|(f_u+g_u)\chi^u|_P=\pi\left(\dsum_{\substack{u\in\Mon\\f_u\neq g_u}}|(f_u+g_u)\chi^u|_{P'}\right)=\pi(d_{P'}(f,g)).$$

Fix $r\in\kappa(P)^\times$; we want to show that there is some $r'\in\kappa(P')^\times$ such that, for any $f,g\in\Cnvg{P'}$ with $d_{P'}(f,g)<r'$, we have $d_{P}(f,g)<r$. Pick $\eps\in\kappa(P)^\times$ \footnote{It is essential that $\kappa(P)\neq\B$ to ensure that such $\eps$ exists.} with $\eps<r$  and $r'\in\kappa(P')^\times$ with $\pi(r')=\eps$. If $d_{P'}(f,g)<r'$ then $d_P(f,g)=\pi(d_{P'}(f,g))\leq\pi(r')=\eps<r$.

Now fix $r'\in\kappa(P')^\times$; we want to show that there is a $r\in\kappa(P)^\times$ such that, if $f,g\in\Cnvg{P}$ have $d_P(f,g)<r$ then $d_{P'}(f,g)<r'$. Let $r=\pi(r')$. Then $d_P(f,g)<r$ says $\pi(d_{P'}(f,g))<\pi(r')$ so, by part (2) of Corollary~\ref{coro:MapOnResidueSemifieldsAndInequalities}, $d_{P'}(f,g)<r'$.
\end{proof}

In light of Theorem \ref{thm:TopologiesOnCnvgCoincide}, from now on, when we refer to $\Cnvg{P}$, we think of it as a topological space with the topology induced by $d_P$, which we call the \emph{metric topology}.

For each $f\in\Cnvg{P}$ and $r\in\kappa(P)^\times$, we set
$$B(f,r)=B_P(f,r):=\{g\in\Cnvg{P}:d_P(g,f)<r\},$$
the open ball around $f$ of radius $r$.

\begin{prop}\label{prop:SequencesSuffice}
Let $\base$ sub-semifield of $\T$ and let $\Mon$ be a monoid. For any $P\in\ContBase{\base} \base[\Mon]$, the {metric} topology on $\Cnvg{P}$ is first-countable. In particular, for any positive integer $\ell$ and any subset $\Xi\subseteq\cnvg{P}^\ell$, the closure of $\Xi$ is exactly the set of points in $\Cnvg{P}^\ell$ which are the limits of sequences $(x_n)_{n\in\N}$ with each $x_n\in\Xi$.
\end{prop}\begin{proof}
Fix any $\eps\in\base^\times$ which is less than $1_\base$. For any $a\in\kappa(P)^\times$, there is some $s\in\base^\times$ such that $|s|_P<a$. Since $\base\subseteq\T$, there is some natural number $n$ such that $\eps^n<s$, so $|\eps^n|_P<a$. Thus, for any $f\in\Cnvg{P}$ the set of open balls $\{B(f,|\eps^n|_P):n\in\N\}$ is a countable neighborhood basis at $f$.
\end{proof}

\begin{prop}\label{prop:ConvergentSeriesAreLimitsOfPolynomials}
Let $\base$ sub-semifield of $\T$ and let $\Mon$ be a monoid. For any $P \in \ContBase{\base}\base[\Mon]$, every $f\in\Cnvg{P}$ is a limit of its partial sums. Thus, the polynomial semiring $\base[\Mon]$ is dense in $\Cnvg{P}$ and so $\Cnvg{P}$ is the completion of $\base[\Mon]$ with respect to the pseudometric $d_p$.

\end{prop}

\begin{remark}
This proposition is not tautological, as convergent series are defined as abstract series satisfying a certain property, rather than as limits of polynomials.
\end{remark}

\begin{proof}[Proof of Proposition~\ref{prop:ConvergentSeriesAreLimitsOfPolynomials}]
Fix any $f\in\Cnvg{P}$ and $\eps\in\base^\times$ which is less than $1_\base$. As in the proof of Proposition~\ref{prop:SequencesSuffice}, the open balls $B(f,|\eps^n|_P)$ form a neighborhood basis at $f$. For any $n\in\N$, let $I_n$ denote the set of $u\in\Mon$ such that $|f_u\chi^u|_P\geq |\eps^n|_P$. Since $f \in \Cnvg{P}$, the set $I_n$ is finite. Now consider the polynomial $g_n = \sum\limits_{u \in I_n} f_u\chi^u$. By the definition of $d_{P}$, we have that $d_{P}(f,g_n) = \sum\limits_{u \not\in I_n}|f_u\chi^u|_P <|\eps^n|_P$. Thus $\dlim_{n\to\infty}g_n=f$.
\end{proof}

We will now show that $\Cnvg{P}$ is a topological semiring whenever $P\in\ContIntBase{\base} \base[\Mon]$, where $\base$ is a sub-semifield of $\T$ and $\Mon$ is a troic monoid.

\begin{theorem}
Let $\base$ be a sub-semifield of $\T$ and $\Mon$ a toric monoid. For $P\in\ContIntBase{\base}{\base[\Mon]}$, the semiring $\Cnvg{P}$ is a topological $\base$-algebra.
\end{theorem}

\begin{proof}
The proof follows directly from Proposition~\ref{prop:+isContinuous} and Proposition~\ref{prop:xisContinuous} showing that the the two operations on $\Cnvg{P}$ are continuous and Proposition~\ref{prop:SToCnvgIsContinuous} showing that the map $\base\to\base[\Mon]\to\Cnvg{P}$ is continuous.
\end{proof}


\GeorgeStory{Where George meets Ulr and Ugr... }
\begin{proposition}\label{prop:+isContinuous}
Let $\base$ be a sub-semifield of $\T$, $\Mon$ a monoid, and $P\in\ContIntBase{\base}{\base[\Mon]}$. Then the addition on $\Cnvg{P}$ is continuous.
\end{proposition}

\begin{proof}
\newcommand{\addmap}{\mathfrak{Z}}
The addition on $\Cnvg{P}$ will be the map
$$\addmap: \Cnvg{P} \times \Cnvg{P} \rightarrow \Cnvg{P}.$$
We will show that the inverse of every open ball in $\Cnvg{P}$ is open in $\Cnvg{P} \times \Cnvg{P}$, where $\Cnvg{P} \times \Cnvg{P}$ has the product topology.

Let $f \in \Cnvg{P}$ and $r \in \kappa(P)^\times$. The preimage of $B(f,r)$ is the set 
\begin{align*}
    \addmap^{-1}(B(f,r)) &= \left\{ (g,h) \in \Cnvg{P} \times \Cnvg{P}\ : \ d_{P}(g+h,f)<r  \right\}\\
                     &= \left\{ (g,h) \in \Cnvg{P} \times \Cnvg{P}\ : \ \dsum_{\substack{u\in \Mon\\g_u+h_u \neq f_u}} |(g_u+h_u+f_u)\chi^u|_P <r  \right\}.
\end{align*}

We want to show that, around each point $(g,h)\in\addmap^{-1}(B(f,r))$, 
there is a product $\mathcal{B}$ of open balls such that, for all $(g',h') \in \mathcal{B}$, we have $g'+h' = \addmap (g', h') \in B(f,r)$.  We first investigate when $g+h \in B(f,r)$. Consider the sets
$$U_{gr} = \left\{ u\in \Mon \ : \  |f_u\chi^u|_P \geq r \right\}
\text{\; and\; }
U_{lr} = \left\{ u\in \Mon \ : \  |f_u\chi^u|_P < r\right\}.$$
With this notation, $g+h \in B(f,r)$ if and only if the following condition holds: for all $u \in U_{lr}$, we have $|g_u\chi^u|_P,|h_u\chi^u|_P < r$ and, for all $u \in U_{gr}$, we have $g_u + h_u = f_u$.

Now we show that, if $g+h \in B(f,r)$ and $(g',h') \in B(g,r) \times B(h,r)$, then $g'+h' \in B(f,r)$. 

For $u \in U_{lr}$ we either have $g_u = g'_{u}$ or $g_u \neq g'_{u}$. In the first case we get that $|g_u\chi^u|_P = |g'_{u}\chi^u|_P <r$. In the second case, where $g_u \neq g'_{u}$, both $|g_u\chi^u|_P$ and $|g'_{u}\chi^u|_P$ contribute to the distance $d_P(g,g')$. Since $g' \in B(g, r)$, we conclude that  
$$|g'_{u}\chi^u|_P\leq |g_u\chi^u|_P + |g'_{u}\chi^u|_P\leq d_P(g,g') <r.$$ 
In either case, we obtain that $|g'_{u}\chi^u|_P < r$. Similarly, we observe that $|h'_{u}\chi^u|_P < r$.

For $u \in U_{gr}$ we can assume without loss of generality that $g_u = f_u$ and that $h_u \leq f_u$. With this assumption we have that $|g_u\chi^u|_P = |f_{u}\chi^u|_P \geq r$. If $g'_u \neq g_{u}$, then $$d_P(g', g) \geq |g_u\chi^u|_P + |g'_{u}\chi^u|_P \geq |g_u\chi^u|_P  \geq r,$$
but this cannot happen. So we conclude that $g'_u = g_{u} = f_u$.
Looking at $h'$, we either have $|h'_u\chi^u|_P < r$ or $|h'_{u}\chi^u|_P \geq r$. In the first case, since $|f_u\chi^u|_P \geq r$, we have 
$$|h'_u|_P |\chi^u|_P= |h'_u\chi^u|_P  < r \leq |f_u\chi^u|_P = |f_u |_P |\chi^u|_P.$$
Since $|\chi^u|_P\in\kappa(P)^\times$, the above inequality implies that 
$|h'_u|_P < |f_u |_P $ and so $h'_u < f_u$.
In the second case, $|h_{u}\chi^u|_P \geq r$ and, since $h' \in B(h, r)$, we conclude that $h'_u = h_u \leq f_u$. Either way, $h'_u\leq f_u$ and, since we already know that $g'_{u} = f_u$, we obtain that $g'_{u} + h'_{u} = f_u$. Thus, $g'+h' \in B(f,r)$.
\end{proof}


We use the following lemma to prove that the multiplication on $\Cnvg{P}$ is continuous.
\begin{lemma}\label{lemma:MultAndDist}
Let $\base$ be a sub-semifield of $\T$, let $\Mon$ be a toric monoid, and let $P\in\ContIntBase{\base} \base[\Mon]$. For any $g,g',h\in\Cnvg{P}$, $d_P(gh,g'h)\leq d_P(g,g')|h|_P$.
\end{lemma}\begin{proof}
For $u\in\Mon$, let 
$$c_u=\dsum_{\substack{v,w\in \Mon\\v+w=u}}g_vh_w 
\quad \text{and} \quad c_u'=\dsum_{\substack{v,w\in \Mon\\v+w=u}}g'_vh_w,$$
so 
$$d_P(gh,g'h)=\dsum_{\substack{u\in \Mon\\c_u\neq c_u'}}|(c_u+c_u')\chi^u|_P
\quad \text{and} \quad
d_P(g,g')=\dsum_{\substack{u\in \Mon\\g_u\neq g_u'}}|(g_u+g_u')\chi^u|_P.
$$

Fix $u\in\Mon$ with $c_u\neq c_u'$; we want to show that $|(c_u+c_u')\chi^u|_P\leq d_P(g,g')|h|_P$. We assume without loss of generality that $c_u>c_u'$, so $c_u+c_u'=c_u$. Pick $v_0,w_0\in\Mon$ attaining the maximum in $c_u=\dsum_{v+w=u}g_vh_w$, i.e., $v_0+w_0=u$ and $c_u=g_{v_0}h_{w_0}$. Then 
$|(c_u+c_u')\chi^u|_P=|g_{v_0}h_{w_0}\chi^u|_P$, 
so we want to show that $|g_{v_0}h_{w_0}\chi^u|_P\leq d(g,g')|h|_P$. If $h_{w_0}=0_\base$ this is clear, so assume $h_{w_0}\in\base^\times$. Hence, we can divide the inequality $g_{v_0}h_{w_0}=c_u>c_u'\geq g'_{v_0}h_{w_0}$ by $h_{w_0}$ to get $g_{v_0}>g'_{v_0}$. In particular, $|(g_{v_0}+g_{v_0}')\chi^{v_0}|_P\leq d_P(g,g')$. Thus, 
$$|g_{v_0}h_{w_0}\chi^u|_P=|g_{v_0}\chi^{v_0}|_P\ |h_{w_0}\chi^{w_0}|_P=|(g_{v_0}+g'_{v_0})\chi^{v_0}|_P\ |h_{w_0}\chi^{w_0}|_P\leq d_P(g,g')|h|_P.$$
\end{proof}

\begin{proposition}\label{prop:xisContinuous}
Let $\base$ be a sub-semifield of $\T$ and $\Mon$ a toric monoid. For any prime $P\in\ContIntBase{\base} \base[\Mon]$, the multiplication on $\Cnvg{P}$ is continuous.
\end{proposition}

\begin{proof}
The multiplication on $\Cnvg{P}$ will be the map
$$m: \Cnvg{P} \times \Cnvg{P} \rightarrow \Cnvg{P}.$$
We will show that the inverse image of every open ball in $\Cnvg{P}$ is open in $\Cnvg{P}^2$ with the product topology.

Let $f \in \Cnvg{P}$ and $r \in \kappa(P)^\times$. The preimage of the ball of radius $r$ centered at $f$ is the set 
$$m^{-1}(B(f,r)) = \left\{ (g,h) \in \Cnvg{P} \times \Cnvg{P}\ : \ d_{P}(gh,f)<r  \right\}.$$
Analogously to proving that addition is continuous, it suffices to show that if $gh \in B(f,r)$ and $(g',h') \in B(g,r_1) \times B(h,r_2)$ then $g'h' \in B(f,r)$, for some choices $r_1,r_2\in\kappa(P)^\times$. We will consider 3 cases:

\noindent \underline{Case 1}: $g=h=0$. In this case we obtain that 
$$\Qnorm{f} = d_P(0,f) = d_P(gh,f) < r.$$
The last inequality holds since $gh \in B(f,r)$. If $r \leq 1$ then $r^2 \leq r$ so we set $r_1 = r$ and ensure that $r_1^2 \leq r$. Otherwise, if $r \geq 1$ then $1^2 \leq r$ so we set $r_1 = 1$ and again ensure that $r_1^2 \leq r$. 
Either way, set $r_2=r_1$. If $(g',h') \in B(g,r_1) \times B(h,r_2) = B(0,r_1) \times B(0,r_1)$ then
$$ \Qnorm{g'h'} = \Qnorm{g'} \Qnorm{h'} < r_1^2 \leq r,$$
and so
$$ d_P(g'h',f) \leq \Qnorm{g'h'} + \Qnorm{f} < r + r = r .$$
\noindent \underline{Case 2}: $g=0$ but $h \neq 0$. In this case we still have that 
$$\Qnorm{f} = d_P(0,f) = d_P(gh,f) < r.$$
Set $r_1 = r / \Qnorm{h}$ and $r_2 =  \Qnorm{h}$. For $(g',h') \in B(g,r_1) \times B(h,r_2) = B(0,r_1) \times B(h,r_2)$ we have that 
$$ \Qnorm{g'h'} = \Qnorm{g'} \Qnorm{h'} < r_1 \Qnorm{h'} \leq  r_1(r_2 + \Qnorm{h}),$$
where the last inequality holds since $h' \in B(h, r_2)$. Thus, we obtain that
$$d_P(g'h',f) \leq \Qnorm{g'h'} + \Qnorm{f} < r_1(r_2 + \Qnorm{h}) + r = r+r = r,$$
where the second-to-last equality is a consequence of our choice of $r_1$ and $r_2$.

\noindent \underline{Case 3}: $g, h \neq 0$. Set $r_1 = r / \Qnorm{h}$ and $r_2 =  \min(\Qnorm{h}, r / \Qnorm{g} )$. Let $(g',h') \in B(g,r_1) \times B(h,r_2)$. Then, by applying the triangle inequality twice, we get that 
$$d_P(g'h',f) \leq d_P(g'h',gh) + d_P(gh,f) \leq d_P(g'h',g'h) + d_P(g'h,gh) + d_P(gh,f).$$ 
By Lemma~\ref{lemma:MultAndDist}, this gives us
\begin{align*}
    d_P(g'h',f)
    &\leq d_P(h',h)\Qnorm{g'} + d_P(g',g)\Qnorm{h} + d_P(gh,f) \\
    &< r_2\Qnorm{g'} + r_1\Qnorm{h} + r \\
    &\leq r_2(\Qnorm{g}+r_1) + r_1\Qnorm{h} +  r\\ 
    &=r_2\Qnorm{g}+r_1r_2 + r_1\Qnorm{h} + r 
    =r.
\end{align*}
\end{proof}


\begin{prop}\label{prop:SToCnvgIsContinuous}
Let $\base$ be a sub-semifield of $\T$, let $\Mon$ be a monoid, and let $P\in\ContBase{\base}\base[\Mon]$. Then the map $\base\to\base[\Mon]\to\Cnvg{P}$ is continuous.
\end{prop}\begin{proof}
For any $f\in\Cnvg{P}$ and $r\in\kappa(P)^\times$, we want to show that $B(f,r)\cap\base$ is open in $\base$. If $0_\base$ is not in $B(f,r)\cap\base$ then $B(f,r)\cap\base$ is open. 

So suppose $0_\base\in B(f,r)$, and fix some $a\in\base^\times$ with $|a|_P<r$. Then
\begin{align*}
B(f,r)\cap\base&=B(0_\base,r)\cap\base\\
&=\{c\in\base\ :\ |c|_P<r\}\\
&\supseteq\{c\in\base\ :\ |c|_P\leq|a|_P\}\\
&=[0_\base,a],
\end{align*}
so $B(f,r)\cap\base$ is open in $\base$.
\end{proof}

\begin{remark} Let $\base$ be a subsemifield of $\T$, $\Mon$ a monoid, and $P\in\ContBaseInt{\base}\base[\Mon]$. In the induced topology by $d_P$ on the space of monomials of $\base[\Mon]$, every point except for $0_{\base}$ is isolated. Indeed,  $f,g$ be distinct monomials (terms) in $\base[\Mon]$. Then by definition $$d_P(f,g) = \sum\limits_{\substack{u\in \Mon \\ f_u \neq g_u}} |(f_u + g_u)\chi^u|_P = |f|_P + |g|_P \geq |f|_P.$$ This shows that the only monomial in the open ball around $f$ of radius $|f|_P$ is $f$ itself. 

This statement is analogous to the fact that, in any totally ordered semifield, every nonzero point is isolated.
\end{remark}


\section{The canonical extension of $P$ to a prime on $\Cnvg{P}$}\label{sect: canExt}
\GeorgeStory{George and the dragon eggs.}

\noindent Let $\base\neq\B$ be a subsemifield of $\T$, let $\Mon$ be a toric monoid, and let  $P$ be a prime in $\ContIntBase{\base}\base[\Mon]$.

\begin{definition} 
The {\it canonical extension} $Q_P$ of $P$ to $\Cnvg{P}$ is the congruence-kernel of the map $\Cnvg{P} \toup{|\cdot|_P} \base[\Mon]/P$. 
Since $\Cnvg{P} \toup{|\cdot|_P} \base[\Mon]/P$ is surjective, there is a natural identification of $\kappa(P)$ with $\kappa(Q_P)$, under which we have that $|\sum\limits_u f_u\chi^u|_{Q_P} = \max\limits_u |f_u\chi^u|_P$.
\end{definition}

Our first result about $Q_P$ is more easily first seen as a result about ultrapseudometrics.

\begin{lemma}\label{lemma:DistIsContinuous}
Let $K$ be a totally ordered semifield and let $d:X^2\to K$ be an ultrapseudometric. For any $a\in X$, the map $d(a,\cdot):X\to K$ is continuous in the topology on $X$ induced by $d$.
\end{lemma}\begin{proof}
Let $\ph(x)=d(a,x)$. 
It suffices to show that  $\ph^{-1}(r)$ and $\ph^{-1}\left([0_{K},r]\right)$ are open in $X$ for $r\in K^\times$.

Say $x\in\ph^{-1}(r)$. If $y\in B(x,r)$ then $d(x,y)<r=d(a,x)$, so $\ph(y)=d(a,y)=d(a,x)+d(x,y)=d(a,x)=r$.

Now consider $x\in\ph^{-1}([0_K,r])$. If $y\in B(x,r)$ then $\ph(y)=d(a,y)\leq d(a,x)+d(x,y)\leq r+r=r$, i.e., $\ph(y)\in[0_K,r]$.
\end{proof}

\begin{prop}\label{prop:CanonicalExtensionIsClosed}
For any $P\in\ContIntBase{\base}\base[\Mon]$, we have $Q_P\in\ContInt\Cnvg{P}$.
\end{prop}\begin{proof}
%
%
Since the map $|\cdot|_P:\Cnvg{P}\to\kappa(P)$ has trivial ideal-kernel, it suffices to show that this homomorphism is continuous. But $|\cdot|_P=d_P(0_{\base},\cdot)$, so Lemma~\ref{lemma:DistIsContinuous} tells us that this is continuous.
\end{proof}

For any $Q\in\Cont(\Cnvg{P})$, Proposition \ref{prop:PullbackFromContToContBase} tells us that the restriction $Q|_{\base[\Mon]}$ of $Q$ to $\base[\Mon]$ is in $\ContBase{\base} \base[\Mon]$.
We now work towards showing that $Q_P$ is the unique element of $\Cont(\Cnvg{P})$ whose restriction to $\base[\Mon]$ is $P$. 
First, we require an auxiliary topological result.

\begin{prop}\label{prop:CannonicalExtensionIsClosure}
Let $P$ be a prime congruence in 
$\ContIntBase{\base}\base[\Mon]$, 
and let $Q_P$ denote the canonical extension of $P$ to $\Cnvg{P}$. If we view $P\subseteq\base[\Mon]^2\subseteq\cnvg{P}^2$, then $Q_P$ is the closure of $P$ in $\Cnvg{P}^2$.
\end{prop}\begin{proof}
From Proposition \ref{prop:CanonicalExtensionIsClosed} we know that $Q_P$ is closed. So it remains to show that, for any closed set $\fred$ such that $P\subseteq\fred\subseteq\cnvg{P}^2$, $Q_P\subseteq\fred$. Say $(f,g)\in Q_P$. By Proposition~\ref{prop:ConvergentSeriesAreLimitsOfPolynomials}, there are sequences $(f_n)_{n\in\N}$ and $(g_n)_{n\in\N}$ in $\base[\Mon]$ with $f=\dlim_{n\to\infty}f_n$ and $g=\dlim_{n\to\infty}g_n$. Note that, for $n$ sufficiently large, $f_n$ and $g_n$ must contain the $P$-leading terms of $f$ and $g$, respectively. For such $n$ we have $(f_n,g_n)\in P\subseteq\fred$. Since $(f_n,g_n)$ is eventually in the closed set $\fred$, $(f,g)=\dlim_{n\to\infty}(f_n,g_n)$ is also in $\fred$.
\end{proof}

The uniqueness of $Q_P$ as an element of $\Cont(\Cnvg{P})$ whose restriction to $\base[\Mon]$ is $P$ will now follow from a more general uniqueness result which will also be used in the proof of the main theorem in Section~\ref{section:TheoremCrown}. 
\begin{prop}\label{prop:GeneralUniquenessOfExtension}
Let $\base$ be a subsemifield of $\T$ and let $\ph:A\to B$ be a homomorphism of $\base$-algebras with $B$ a topological $\base$-algebra. Suppose that the closure $\bar{P}$ of $(\ph\times\ph)(P)$ in $B^2$ is a prime congruence on $B$, $\ph^*(\bar{P})=P$, and that the map $A/P\to B/\bar{P}$ induced by $\ph$ is an isomorphism. Then $\bar{P}$ is the unique closed prime on $B$ whose pullback to $A$ is $P$.
\end{prop}\begin{proof}
%
%
Note that, since $\ph^*(\bar{P})=P$, we must have $\ker\ph\subseteq P$. Hence, by quotienting $A$ by $\ker\ph$, we may assume without loss of generality that $A$ is an $\base$-subalgebra of $B$. Since $A/P\to B/\bar{P}$ is surjective, for all $f\in B$ there is some $f_0\in A$ such that $(f_0,f)\in \bar{P}$.

Suppose that $Q$ is a closed prime congruence on $B$ such that $Q\cap A^2=P$. Since $Q$ is closed and contains $P$, we have $Q\supseteq\bar{P}$.

Consider $(f,g)\in Q$. So there exist $f_0,g_0\in A$ such that $(f_0,f)$ and $(g_0,g)$ are both in $\bar{P}\subseteq Q$. By symmetry and transitivity of $Q$, we get that $f_0\sim_{Q}f\sim_{Q}g\sim_{Q}g_0$ gives us $(f_0,g_0)\in Q\cap A^2=P$. So $f\sim_{\bar{P}}f_0\sim_{\bar{P}}g_0\sim_{\bar{P}}g$, giving $(f,g)\in\bar{P}$. Hence $Q\subseteq \bar{P}$.
\end{proof}

\begin{coro}\label{coro:CannonicalExt}
For any $P\in\ContIntBase{\base}\base[\Mon]$, $Q_P$ is the unique element of $\Cont(\Cnvg{P})$ whose restriction to $\base[\Mon]$ is $P$.
\end{coro}\begin{proof}
Note that, by definition of the canonical extension $Q_P$, $Q_P|_{\base[\Mon]}=P$ and the induced map $\base[M]/P\to\Cnvg{P}/Q_P$ is an isomorphism. Proposition~\ref{prop:CannonicalExtensionIsClosure} tells us that $Q_P$ is the closure of $P$ in $\Cnvg{P}^2$. Since any element of $\Cont(\Cnvg{P})$ is closed in $\Cnvg{P}^2$, Proposition~\ref{prop:GeneralUniquenessOfExtension} now gives us that $Q_P$ is the unique element of $\Cont(\cnvg{P})$ whose restriction to $\base[\Mon]$ is $P$. 
\end{proof}

\begin{remark} 
%
To get Corollary~\ref{coro:CannonicalExt} we applied Proposition~\ref{prop:GeneralUniquenessOfExtension} with $A=\base[\Mon]$ and $B=\Cnvg{P}$ where $P$ is also the prime we are considering extending. In Section~\ref{section:TheoremCrown} we will apply Proposition~\ref{prop:GeneralUniquenessOfExtension} to the case where these primes are distinct.
\end{remark}


\begin{example}\label{ex:CanonicalExtUnderP}
Given $P\in\ContIntBase{\base}\base[\Mon]$, consider any prime congruence $P' \subseteq P$. By Proposition~\ref{prop:GoingDownStaysInGeorge}, $P'\in\ContIntBase{\base}\base[\Mon]$ so we can apply Lemma~\ref{lemma:CnvgInChain} to get that $\Cnvg{P'} = \Cnvg{P}$. 
Hence, $P'$ also has a unique extension in $\Cont(\Cnvg{P})$, namely $Q_{P'}$.\exEnd
\end{example}




\section{Description of $\Cont{(\Cnvg{P})}$}\label{section:TheoremCrown}
\GeorgeStory{In which George's hat gets fancy.\newline}

In this section we let $\base\neq\B$ be a sub-semifield of $\T$, let $\Mon$ be a toric monoid, and let $P\in\ContIntBase{\base} \base[\Mon]$.
Our goal is to study the rest of $\Cont(\Cnvg{P})$. To do this, we will consider other $P'\in\ContBase{\base} \base[\Mon]$ and compare them to $P$. In order to compare evaluations of elements of $\base[\Mon]$ at both $P$ and $P'$, we embed $\kappa(P)$ and $\kappa(P')$ into a single totally ordered semifield.

For any such $P$ and $P'$, Proposition \ref{prop:HahnFor2SpaPrimes} tells us that we can fix a totally ordered set $I_{P,P'}$ with least element $i_0(P,P')$ and embeddings $\ph_P:\kappa(P)\to\RnLexSemifield{I_{P,P'}}$ and $\ph_{P'}:\kappa(P')\to\RnLexSemifield{I_{P,P'}}$ such that the diagram
\comd{
&&\kappa(P)\ar[r]^{\pi_{\check{P},P}}\inj[d]_{\ph_P}&\kappa\left(\check{P}\right)\inj[d]\\
\base\inj[r]\ar[urr]\ar[drr]&\T\ar[r]^{j}&\RnLexSemifield{I_{P,P'}}\ar[r]^{\diagbigpi}&\T\\
&&\kappa(P')\ar[r]_{\pi_{\check{P}',P'}}\linj[u]^{\ph_{P'}}&\kappa\left(\check{P}'\right)\linj[u]\\
}
commutes. Here $\check{P}$ and $\check{P}'$ are the maximal elements of $\ContBase{\base} \base[\Mon]$ containing $P$ and $P'$, respectively, and $j$ and $\bigpi$ are the continuous homomorphisms induced by inclusion of, and projection onto, the $i_0(P,P')$ factor. Recall that, by Lemma \ref{lemma:MaximalGeorgePrimes}, the embeddings $\kappa(\check{P})\into\T$ and $\kappa(\check{P}')\into\T$ are uniquely determined by their compatibility with $\base$.

Using these embeddings and the section $j:\T\to\RnLexSemifield{I_{P,P'}}$ of $\bigpi$, for any $f\in\base[\Mon]$ we can view $|f|_{P}, |f|_{P'}, |f|_{\check{P}}$, and $|f|_{\check{P}'}$ as elements of $\RnLexSemifield{I_{P,P'}}$, and we do so for the rest of this section. 
Viewing $j$ as an embedding also makes $\bigpi$ an endomorphism of $\RnLexSemifield{I_{P,P'}}$, which satisfies the following property:
 for any $\gamma\in\RnLexSemifield{I_{P,P'}}$ and $d\in\T$ such that $d>1_{\T}$, 
 \begin{equation}\label{eq:bp1}
      \bigpi(\gamma)<d \gamma \text{ and } d\bigpi(\gamma)>\gamma.
 \end{equation}


We start by considering the closures in $\Cnvg{P}^2$ of primes on $\base[\Mon]$. These are not necessarily congruences, since they do not need to be transitively closed.

\begin{notation}
For primes $P \in \ContBaseInt{\base}{\base[\Mon]}$ and $P' \in \ContBase{\base}{\base[\Mon]}$ we will denote by $\bar{P'}$ the topological closure of $P'$ in $\Cnvg{P}^2$. 
\end{notation}


\begin{example}[Motivating example]\label{ex:motivating} Let $P$ and $P'$ be prime congruences on $\T[x^{\pm 1}]$ with defining matrices $\begin{pmatrix} 1 & 0 \end{pmatrix}$ and $\begin{pmatrix} 1 & 1 \\ 0 & 1\end{pmatrix}$ respectively. Consider two sequences of polynomials for $n \geq 0$,
$$f_n = 1_\T + t^{-n}x^n \text{, and } g_n = t^{-n}x^n.$$
Both of these sequences converge in $\Cnvg{P}$. More precisely,
$$d_P(1_\T, f_n) = d_P(1_\T,1_\T + t^{-n}x^n) = |t^{-n}x^n|_P = t^{-n}.$$
Thus, as $\displaystyle \lim_{n\rightarrow \infty} d_P(1_\T, f_n) = 0_\T$, and so $f_n$ converges to $1_\T$ in the $P$ norm. Similarly, we check that $\displaystyle \lim_{n\rightarrow \infty} d_P(0_\T, g_n) = 0_\T$, so the sequence $g_n$ converges to $0_\T$. Since the pairs of polynomials $(f_n, g_n)$ are elements of $P'$, we get that $\displaystyle \lim_{n\rightarrow \infty}(f_n, g_n) = (1_\T, 0_\T)$ belongs to $\bar{P'}$. 
Thus, any congruence that contains the closure of $P'$ contains $(1_\T, 0_\T)$ and is therefore all of $\Cnvg{P}^2$. 
In particular, there is no closed prime on $\Cnvg{P}$ which contains $P'$. 
\exEnd\end{example}

We abstract the idea of this example into the following lemma.
\begin{lemma}\label{lem:converge-to-(0,1)}
Consider a maximal prime $P \in \ContBaseInt{\base}{\base[\Mon]}$ and any prime $P'\in \ContBase{\base}{\base[\Mon]}$. If there is a monomial $m$ such that $|m|_P < 1_\base \leq |m|_{P'}$, then 
there is no closed prime of $\Cnvg{P}$ which contains ${P'}$.
\end{lemma}

\begin{proof}
Assume there is such an $m$, and suppose there were a closed prime congruence $Q\supseteq P'$ on $\Cnvg{P}$. We can consider two sequences, $\{ f_n\}_{n\in \N}$ and $\{ g_n\}_{n\in \N}$ given by  
$$f_n = 1_\base + m^n \text{ and } g_n = m^n.$$ 
Notice that $(f_n, g_n) \in P'$. To see what these sequences converge to, compute
$$d_P(1_\base, f_n) = d_P(1_\base, 1_\base + m^n) = |m^n|_P = |m|_P^n.$$
Since $|m|_P < 1_\base$ and is in $\T$, $\displaystyle \lim_{n\rightarrow \infty} d_P(1_\base, f_n) = 0_\base$, and so $f_n$ converges to $1_\base$ in $\Cnvg{P}$. Similarly, we check that $d_P(0_\base, g_n) = 0_\base$, so the sequence $g_n$ converges to $0_\base$. 
Thus, $(f_n, g_n)$ converges to $(1_\base, 0_\base)$, which tells us that $(1_\base,0_\base)\in\bar{P}'\subseteq Q$, contradicting the assumption that $Q$ is prime.
\end{proof}




In the following theorem, we classify which primes on $\base[\Mon]$ occur as the restriction of some $Q\in\Cont(\Cnvg{P})$. We use the above lemma to rule out those primes that do not arise this way. First, we need to introduce the following condition on primes.

\begin{defi} For prime congruences $P \in \ContIntBase{\base} \base[\Mon]$ and $P'\in\ContBase{\base} \base[\Mon]$, we say that $P'$ satisfies $\propStar$ with respect to $P$ if there exists a nonzero $c \in \base$ such that, for any term $m \in \base[\Mon]$, $c|m|_P \geq |m|_{P'}$.
\end{defi}

\begin{theorem}[$\crown{}$]\label{thm:crown}
Fix, $P \in \ContIntBase{\base} \base[\Mon]$. The map $\iota^*:\Cont (\Cnvg{P}) \rightarrow \ContBase{\base} \base[\Mon]$ induced by the inclusion $\iota: \base[\Mon] \rightarrow \Cnvg{P}$ is a bijection onto the set of those prime congruences in  $\ContBase{\base} \base[\Mon]$ that satisfy $\propStar$ with respect to $P$.
\end{theorem}

Before we prove the theorem we look at special cases of interest.

\begin{coro}\label{coro:crown_torus}
Let $\base$ be a sub-semifield of $\T$ and let $\Mon=\Z^n$, so $\base[\Mon]=\base[x_1^{\pm1},\ldots,x_n^{\pm1}]$ is the Laurent polynomial semiring over $\base$. If $P\in\ContBase{\base} \base[\Z^n]$ is a maximal prime, then the map $\iota: S[x_1^{\pm1},\ldots,x_n^{\pm1}] \rightarrow \Cnvg{P}$ induces a bijection onto the set of primes $P' \subseteq P$.
\end{coro}

\begin{proof}
In this case $\ContBase{\base} \base[\Z^n] = \ContIntBase{\base} \base[\Z^n]$ and every $P'\subseteq P$ satisfies $\propStar$ with respect to $P$. To see this, fix $c\in\base^\times$ greater than $1_\base$ and note that, for any term $m$,  $c|m|_P=c\bigpi(|m|_{P'})>|m|_{P'}$ by (\ref{eq:bp1}).

Now it suffices to show that, if $P'\not\subseteq P$, then $P'$ is not in the image of $\iota^*$.
Notice that $P' \subseteq P$ if and only if for all terms $m_1$ and $m_2$, for which $|m_1|_{P'} \leq |m_2|_{P'}$, we have that $|m_1|_P \leq |m_2|_P$. So, if $P' \not\subseteq P$, there are terms $m_1$ and $m_2$ violating the above conditions. Consider $m = m_1/m_2$ and observe that $|m|_{P'} > |1_S|_{P'}$ and $|m|_{P} < |1_S|_{P}$. 
Now apply Lemma~\ref{lem:converge-to-(0,1)} to this $m$ to conclude that there is no closed prime of $\Cnvg{P}$ containing ${P'}$.
\end{proof}


\begin{coro}\label{coro:crown_affine}
Let $\base$ be a sub-semifield of $\T$ and let $\Mon=\N^n$, so $\base[\Mon]=\base[x_1,\ldots,x_n]$ is the polynomial semiring over $\base$. Consider $P\in\ContIntBase{\base} \base[x_1, \ldots, x_n]$ and $P' \in\ContBase{\base} \base[x_1, \ldots, x_n]$. By Corollary~\ref{coro: matrix_form} we can pick defining matrices $$U = \begin{pmatrix} 1 & p_1 & \ldots & p_n \\
0& * & \cdots & * \\ \vdots & \vdots & \ddots & \vdots \\ 0& * & \cdots & * \end{pmatrix}
\text{\; and\; }
U' = \begin{pmatrix} 1 & p_1' & \cdots & p_n' \\
0& * & \cdots & * \\ \vdots & \vdots & \ddots & \vdots \\ 0& * & \cdots & * \end{pmatrix}
$$
for $P$ and $P'$, respectively, where $U$ has entries in $\R$ and $U'$ has entries in $\R\cup\{-\infty\}$. Then $P'$ is in the image of the injective map $\iota^*: \Cont (\Cnvg{P}) \rightarrow \ContBase{\base}\base[x_1, \ldots, x_n ]$ if and only if $p_i' \leq p_i$  for all $1 \leq i \leq n$.
\end{coro}

\begin{remark}
One technicality we need to take into account in the above corollary is that in order to compare the norms of terms taken with respect to different primes we need to use matrices of the same size. Thus, for the purpose of stating the corollaries we will assume that every prime corresponding to a matrix of size $k \times (n+1)$ is defined by a square matrix, whose remaining $n+1-k$ rows are all 0. Each matrix gives a Hahn embedding as in Example~\ref{example:MatricesAreHahnEmbeddings} and adding the extra rows creates a joint Hahn embedding as in  Proposition~\ref{prop:HahnFor2SpaPrimes}.
\end{remark}

\begin{proof}[Proof of Corollary~\ref{coro:crown_affine}]
Let $P$ and $P'$ have defining matrices $U$ and $U'$, as in the statement. So for $m=t^ax_1^{u_1}\cdots x_n^{u_n}$, $t^b|m|_P\geq|m|_{P'}$ exactly if 
$$\begin{pmatrix}b\\ 0\\ \vdots\\ 0\end{pmatrix}+_{\R}U\begin{pmatrix}a\\ u_1\\ \vdots\\ u_n\end{pmatrix}\geq_{lex}U'\begin{pmatrix}a\\ u_1\\ \vdots\\ u_n\end{pmatrix}.$$

Suppose $p_i'\leq p_i$ for $1\leq i\leq n$. Pick any $c\in\base^\times$ greater than $1_\base$, so $c=t^b$ for some $b>0_{\R}$. For any term $m=t^ax_1^{u_1}\cdots x_n^{u_n}$, we have
\begin{equation}\label{ineq:CrownAffineNorms}
    a+\dsum_{i=1}^n p_i'u_i\leq a+\dsum_{i=1}^n p_iu_i<b+a+\dsum_{i=1}^n p_iu_i
\end{equation}
because $u_i\geq0$. 
Since, in (\ref{ineq:CrownAffineNorms}), left-most sum is the first entry of 
$U'\begin{pmatrix}a& u_1& \cdots& u_n\end{pmatrix}^T$
and the right-most sum is the first entry of 
$\begin{pmatrix}b& 0& \cdots& 0\end{pmatrix}^T +_{\R} U\begin{pmatrix}a& u_1& \cdots& u_n\end{pmatrix}^T$, 
this gives us that $c|m|_P\geq|m|_{P'}$. Thus, $P'$ satisfies $\propStar$ with respect to $P$.

Now suppose $p_l'>p_l$ for some $l$ between $1$ and $n$; it suffices to show that $P'$ is not in the image of $\iota^*$. Since $(p_l,p_l')$ is a nonempty open interval of $\R$, there exists a positive integer $\beta$ such that $(\beta p_l, \beta p'_l)$ contains some $-\alpha$ with $t^\alpha \in \base$. Let $m = t^{\alpha} x^{\beta}_l$, which has exponent vector $\alpha e_1+\beta e_{l+1}$. Since $U'(\alpha e_1+\beta e_{l+1})$ has first entry $\alpha+_{\R}\beta p_l'$, $U(\alpha e_1+\beta e_{l+1})$ has first entry $\alpha+_{\R}\beta p_l$, and $\alpha +_{\R} \beta p_l < 0_\R < \alpha +_{\R} \beta p_l'$, we have $|m|_P<1_{\base}<|m|_{P'}$. 
By Lemma~\ref{lem:converge-to-(0,1)} applied to this $m$, there is no closed prime of $\Cnvg{P}$ containing ${P'}$.
\end{proof}

\begin{remark}
Recall from Section~\ref{sect:prelim} that, if the defining matrix of a prime on $\base[x_1,\ldots,x_n]$ has $-\infty$ as an entry then all other entries in that column are also $-\infty$. Equivalently, we can remove the columns with $-\infty$ entries and consider a prime on a semiring with less variables. 
\end{remark}

We now proceed to prove a sequence of statements that imply Theorem~\ref{thm:crown}. The following proposition will allow us to assume, in the proof of Theorem~\ref{thm:crown}, that $P$ is maximal in $\ContIntBase{\base}\base[\Mon]$.

\begin{prop}\label{prop:SalvageThmCrownByPMaxmial}
Let $P, \check{P} \in\ContIntBase{\base} \base[\Mon]$ be such that $\check{P}$ is maximal and $\check{P}\supseteq P$, and let $P'\in\ContBase{\base} \base[\Mon]$. Then $P'$ satisfies $\propStar$ with respect to $P$ if and only if it satisfies $\propStar$ with respect to $\check{P}$.
\end{prop}\begin{proof}
Suppose that there exists a $c\in\base^\times$ such that $c|m|_P\geq|m|_{P'}$ for all terms $m\in\base[\Mon]$. Fix any $d\in\base$ greater than $1_{\base}$ and let $\check{c}=cd$. For any term $m\in\base[\Mon]$ we have 
$$\check{c}|m|_{\check{P}}=cd\bigpi(|m|_P)>c|m|_P\geq|m|_{P'}.$$

For the other direction, suppose that there is some $\check{c}\in\base$ such that $\check{c}|m|_{\check{P}}\geq|m|_{P'}$ for all terms $m\in\base$. Fix any $d>1_{\base}$ in $\base$ and let $c=d\check{c}$. For any term $m\in\base[\Mon]$ we have
$$c|m|_P=d\check{c}|m|_P>\check{c}\bigpi(|m|_P)=\check{c}|m|_{\check{P}}\geq|m|_{P'}.$$
\end{proof}


\begin{proposition}\label{prop:convg-containment} 
If $P' \in\ContBase{\base}  \base[\Mon]$ satisfies $\propStar$ with respect to a prime $P \in \ContIntBase{\base} \base[\Mon]$, then $\Cnvg{P} \subseteq \Cnvg{P'}$. In this case, the inclusion map $\Cnvg{P}\to\Cnvg{P'}$ is uniformly continuous.
\end{proposition}

\begin{proof}
Fix $f \in \Cnvg{P}$. For any nonzero element $b\in \base[\Mon]/P'$ we want to show that $\{ u \in \Mon \ : \ |f_u\chi^u|_{P'} \geq b\}$ is finite. Since $b\in\kappa(P')^\times$, there exists a nonzero element $d \in \base$ such that $|d|_{P'} < b$, whence
$$\{ u\in \Mon \ : \ |f_u\chi^u|_{P'} \geq b \} \subseteq \{ u\in \Mon \ : \ |f_u\chi^u|_{P'} \geq |d|_{P'} \}.$$
Since $P'$ satisfies $\propStar$, there is a $c\neq 0_\base$ such that, for any $u\in \Mon$ with $|f_u\chi^u|_{P'} \geq |d|_{P'}$,
$$c |f_u\chi^u|_{P} \geq |f_u\chi^u|_{P'} \geq |d|_{P'} = |d|_{P}.$$
Dividing by $c$ we obtain that $|f_u\chi^u|_{P} \geq |d/c|_P$. This implies that 
$$\{ u\in \Mon \ : \ |f_u\chi^u|_{P'} \geq |d|_{P'} \}  \subseteq  \{ u\in \Mon \ : \ |f_u\chi^u|_{P} \geq |d/c|_P \},$$
where the second set is finite, since $f\in \Cnvg{P}$.
Thus, we conclude that $f\in\Cnvg{P'}$.

Given $\eps\in\kappa(P')^\times$, we want to show that there is a $\delta\in\kappa(P)^\times$ such that, if $f,g\in\Cnvg{P}$ with $d_P(f,g)<\delta$, then $d_{P'}(f,g)<\eps$. Note that
\begin{align*}
    d_{P'}(f,g)=\sum_{\substack{u\in\Mon\\f_u\neq g_u}}\left(|f_u|_{P'}+|g_u|_{P'}\right) \leq\sum_{\substack{u\in\Mon\\f_u\neq g_u}}c\left(|f_u|_{P}+|g_u|_{P}\right) =c d_P(f,g).
\end{align*}
Thus, picking any $\delta\in\base^\times$ which is less than $\eps/c$ will work.
\end{proof}

The above proposition is used two ways. First, it guarantees that the $P'$-leading terms of $f\in\Cnvg{P}$ form a polynomial as pointed out in Remark~\ref{rem:leading-term-wrt-P'}. Second, the continuity statement will be used to deduce that $\bar{P}'\in\Cont\Cnvg{P}$.

\begin{remark}\label{rem:leading-term-wrt-P'} 
By Proposition~\ref{prop:convg-containment} every $f \in \Cnvg{P}$ has a leading term with respect to the ordering induced by $P'$, and there are finitely many such leading terms. 
\end{remark}

The following result is a technical tool that we use in the subsequent lemma to give an algebraic characterization of $\bar{P}'$.

\begin{lemma}\label{lem:claim4} 
Let $P' \in\ContBase{\base} \base[\Mon]$ be a prime that satisfies $\propStar$ with respect to a prime $P \in \ContIntBase{\base} \base[\Mon]$. Fix $f \in \Cnvg{P}$ and let $\{f_n\}_{n \in \N}$ be any sequence of polynomials in $\base[\Mon]$ converging to $f$ in $\Cnvg{P}$. Denote by $\mathcal{L}_f'$ the set of leading terms of $f$ with respect to the preorder induced by $P'$, and let $f' = \sum\limits_{m \in \mathcal{L}_f'} m$. Then there exists $n_0 \in \N$ such that, for every $n \geq n_0$, we have that 
$(f',f_n)\in P'$.
\end{lemma}

\begin{proof}
Write $f = f' + f''$ so that the supports of $f'$ and $f''$ are disjoint. By Remark~\ref{rem:leading-term-wrt-P'}, $f'$ is a polynomial. Since $P'$ satisfies $\propStar$, we can fix a nonzero $c \in \base$ such that $c|m|_P \geq |m|_{P'}$ for all terms $m$. For some $n_0 \in \N$ and for all $n \geq n_0$, $$d_P(f_n, f) < c^{-1}|f'|_{P'}\leq|f'|_P.$$
In particular, for all $n \geq n_0$ we can write $f_n = f' + f_n''$ with the supports of $f'$ and $f_n''$ disjoint. Consider a term $\mu = f_{n,u}\chi^u$ of $f_n''$. 
If $f_{n,u} \neq f_u$, where $f_u$ is the coefficient of $\chi^u$ in $f$, then 
$$|\mu|_{P'}\leq c|\mu|_P \leq c\ d_P(f_n, f)<|f'|_{P'}.$$
If $f_{n,u} = f_u$, then $\mu = f_u\chi^u$ is a term of $f''$, so $|\mu|_{P'}<|f'|_{P'}$. In both cases, $|\mu|_{P'}<|f'|_{P'}$ and thus $|f_n''|_{P'}<|f'|_{P'}$ implying that $|f_n|_{P'} = |f' + f_n''|_{P'} = |f'|_{P'}$. Since $f',f_n\in\base[\Mon]$, this says that $(f',f_n)\in P'$.
\end{proof}

\begin{lemma}\label{lem:claim5}
Let $P' \in \ContBase{\base} \base[\Mon]$ be a prime that satisfies $\propStar$ with respect to $P \in \ContIntBase{\base} \base[\Mon]$. For a pair $(f,g) \in \Cnvg{P}^2$, denote by $m_f$ and $m_g$ any leading monomials of $f$ and $g$ with respect to $P'$. Then $(f,g)$ is in $\bar{P'}$ if and only if $(m_f, m_g) \in P'$.
\end{lemma}

\begin{proof}
Suppose $(f,g)$ is in $\bar{P'}$. By Proposition \ref{prop:SequencesSuffice}, we can pick a sequence $(f_n, g_n)$ in $P'$ of pairs of polynomials converging to $(f,g)$ in $\cnvg{P}$. Denote by $\mathcal{L}_f'$ and $\mathcal{L}_g'$ the sets of leading terms of $f$ and $g$ with respect to the ordering induced by $P'$. Let $f' = \sum\limits_{m_f \in \mathcal{L}_f'} m_f$ and $g' = \sum\limits_{m_g \in \mathcal{L}_g'} m_g$. By Lemma~\ref{lem:claim4}, there exists $n_0\in \N$ such that for all $n \geq n_0$, 
$(f', f_n) \in P'$ and $(g', g_n) \in P'$. Since $(f_n, g_n) \in P'$, transitivity of $P'$ gives us $(f',g') \in P'$.\par
For the other direction, let $f'$ and $g'$ be defined as before. By assumption, $(f', g') \in P'$. Let $\{ f_n\}_{n\in \N}$ and $\{ g_n\}_{n\in \N}$ be sequences of partial sums with $f_1 = f'$ and $g_1 = g'$, such that $\{ f_n\}$ converges to $f$ and $\{ g_n\}$ converges to $g$ in $\Cnvg{P}$. Then $(f_n, f')\in P'$ and $(g_n, g')\in P'$ and, by transitivity, $(f_n, g_n) \in P'$. Thus, $(f,g) = \lim\limits_{n\rightarrow\infty}(f_n, g_n) \in \bar{P'}$.
\end{proof}

The following Corollary together with the previous statement give us one direction of the proof of 8.5.

\begin{coro}\label{coro:propStarImpliesSpa}
If a prime $P'\in\ContBase{\base} \base[\Mon]$ satisfies $\propStar$ with respect to a prime congruence $P\in\ContBaseInt{\base}{\base[\Mon]}$, then $\bar{P}'\in\Cont (\Cnvg{P})$.

\end{coro}\begin{proof}
By Lemma \ref{lem:claim5}, $\bar{P}'$ is a prime congruence on $\Cnvg{P}$. 
The same lemma also shows that the inclusion $\base[\Mon]\subseteq\Cnvg{P}$ induces an isomorphism $\kappa(P')\cong\kappa(\bar{P}')$ which makes the diagram
\comd{
\cnvg{P}\inj[r]\ar[d]_{|\cdot|_{\bar{P}'}}&\cnvg{P'}\ar[d]^{|\cdot|_{P'}}\\
\kappa(\bar{P}')&\kappa(P')\ar[l]_{\cong}
}
commute.
So it suffices to show that the composition 
$$\Cnvg{P}\into\Cnvg{P'}\toup{|\cdot|_{P'}}\kappa(P')$$
is continuous. The first map is continuous by Proposition~\ref{prop:convg-containment}. 
The continuity of the second map, $f\mapsto |f|_{P'}=d_{P'}(0_{\base},f)$, follows from Lemma~\ref{lemma:DistIsContinuous}.
\end{proof}

\begin{proof}[Proof of Theorem~\ref{thm:crown}]
Fix $P \in \ContBaseInt{\base}{\base[\Mon]}$. Since changing $P$ to the maximal $\check{P}\in\ContBaseInt{\base}{\base[\Mon]}$ does not change $\Cnvg{P}$ as a topological $\base[\Mon]$-algebra, 
and by Proposition \ref{prop:SalvageThmCrownByPMaxmial} it also does not change which $P'\in\ContBase{\base}{\base[\Mon]}$ satisfy $\propStar$ with respect to $P$, 
we may assume without loss of generality that $P$ is maximal in $\ContBase{\base}{\base[\Mon]}$. 
We will first show that the image of $\iota^*:\Cont (\Cnvg{P}) \rightarrow \ContBase{\base}{\base[\Mon]}$ is the set of primes of $\base[\Mon]$ satisfying $\propStar$ with respect to $P$. 
If $P'$ satisfies $\propStar$ with respect to $P$ then Corollary~\ref{coro:propStarImpliesSpa} asserts that $\bar{P}'\in\Cont(\Cnvg{P})$ and Lemma~\ref{lem:claim5} shows that $\iota^*(\bar{P}')=P'$. 

For the other inclusion, assume $P'$ does not satisfy $\propStar$ with respect to $P$. We will show the existence of a monomial $m$ such that $|m|_P < 1_\base < |m|_{P'}$. 
Applying Lemma~\ref{lem:converge-to-(0,1)} to this $m$ shows that there is no closed prime of $\Cnvg{P}$ containing ${P'}$.\par

If $P'$ does not satisfy $\propStar$ with respect to $P$, then for all $c\in \base\setminus\{0_\base\}$ there exists a term $\mu \in\base[\Mon]$ such that $c|\mu|_P < |\mu|_{P'}$. In particular, we may assume that $c > 1_\base$ and so $$|\mu|_{P}<c|\mu|_{P}<|\mu|_{P'}.$$
Since $|\mu|_{P} \in \T^\times$, $(|\mu|_{P}, c|\mu|_{P})$ is an interval so there exists a $\beta\in \N$ such that for some $a\in \base$, $a^{-1}$ is in the interval $(|\mu|_{P}^\beta, (c|\mu|_{P})^\beta)$. 
Now let $m = a\mu^\beta$ and observe that
$$|m|_P = a|\mu|_P^\beta < 1_\base < a(c|\mu|_P)^\beta < a|\mu|_{P'}^\beta = |m|_{P'}.$$

It remains to show that the map $\iota^*$ is an injection. 
For any $P'$ in the image of $\iota^*$, Lemma~\ref{lem:claim5} gives us that the pullback of $\bar{P}'$ to $\base[\Mon]$ is $P'$ and the induced map $\base[\Mon]/P' \to \Cnvg{P}/\bar{P'}$ is an isomorphism. Thus, by Proposition~\ref{prop:GeneralUniquenessOfExtension}, the prime congruence $\bar{P}'$ is the unique element of $\Cont(\cnvg{P})$ whose pullback to $\base[\Mon]$ is $P'$.
\end{proof}

We now show that $\propStar$ is equivalent to another condition which is, \emph{a priori}, stronger. This version of the condition will be used in proving the main theorem of the next subsection.

%
%
\begin{prop}\label{prop:TwoVersionsOfProperty*AreSame}
Let $P \in \ContBaseInt{\base}{\base[\Mon]}$ and $P'\in \ContBase{\base}{\base[\Mon]}$. Then $P'$ satisfies $\propStar$ with respect to $P$ if and only if, for every $c\in\base$ with $c>1_{\base}$ and every term $m\in \base[\Mon]$, $c|m|_P\geq|m|_{P'}$.
\end{prop}
\begin{proof}
The ``if'' direction is clear. So assume that $P'$ satisfies $\propStar$ with respect to $P$, i.e., there is some nonzero $c\in \base$ such that $c|m|_{P}\geq|m|_{P'}$ for all terms $m\in \base[\Mon]$. 

If $c\leq 1_\base$ then, for any $\tilde{c}\in\base$ with $\tilde{c}>1_\base$, $\tilde{c}|m|_P\geq c|m|_{P}\geq|m|_{P'}$.

Now say $c>1_\base$. For any $\tilde{c}\in\base$ with $\tilde{c}>1_\base$, there is a positive integer $\ell$ such that $c^{1/\ell}\leq\tilde{c}$ in $\T$. Then, in the algebraically closed semifield $\RnLexSemifield{I_{P,P'}}$, we have 
$$\tilde{c}|m|_P\geq c^{1/\ell}|m|_P=(c|m^\ell|_P)^{1/\ell}\geq(|m^\ell|_{P'})^{1/\ell}=|m|_{P'}.$$
\end{proof}

\subsection{Geometric description of $\Cont{(\Cnvg{P})}$}\label{sect: geom-crown}


We round off this section by providing a more geometric flavor to the above results. 
We continue the assumption that $\base$ is a sub-semifield of $\T$, different from $\BB$. In this section, we write the toric monoid $\Mon $ as $\sigma^{\vee} \cap \Lambda$ for $\Lambda$ a 
finitely generated free abelian group and $\sigma$ a cone in $\Lambda^*$. 
We fix primes $P \in \ContBaseInt{\base}{\base[\Mon]}$ and $P' , \check{P}'\in \ContBase{\base}{\base[\Mon]}$, where $\check{P}'\supseteq P'$ and $\check{P}'$ is maximal in $\ContBase{\base}{\base[\Mon]}$ with respect to inclusion.\par\medskip


There is a map $\Phi:\ContBase{\base}{\base[\Mon]}\to N_{\R}(\sigma)$ defined as follows. Given $P'\in\ContBase{\base}{\base[\Mon]}$ and following the notations above, we have the semiring homomorphism
$$\ph=\ph_{P'}:\base[\Mon]\to\kappa(P')\toup{\pi_{\check{P}',P'}}\kappa(\check{P}')\to\T.$$
\noindent Precomposing with the canonical monoid homomorphism from $\Mon$ to the multiplicative monoid of $\base[\Mon]$ gives us a monoid homomorphism $\phi_{P'}:\Mon\to(\T,\cdot)$, i.e., an element of the tropical toric variety $N_{\R}(\sigma)$. We set $\Phi(P')=\phi_{P'}$. Moreover, if $P'$ has trivial ideal-kernel, $\phi_{P'}$ does not map anything to $0_{\T}$, and so induces a group homomorphism $\Lambda\to\R$, i.e., an element of $N_{\R}$. So the map $\Phi:\ContBase{\base}{\base[\Mon]}\to N_{\R}(\sigma)$ restricts to a map $\ContBaseInt{\base}{\base[\Mon]}\to N_{\R}$.

\begin{example}\label{ex:MapToClassicalTorus}
Suppose $M=\Lambda=\Z^n$, so $\sigma=\{0\}$. In this case we can give an even more explicit description of the map $\Phi:\ContBase{\base}\base[\Mon]\to N_{\R}(\sigma)=N_{\R}=\R^n$. Given $P'\in\ContBase{\base}\base[\Mon]$, pick a defining matrix for $P'$ of the form 
$$\begin{pmatrix}
 1&w\\
 *&*\\
 \vdots&\vdots\\
 *&*
\end{pmatrix}$$
where $w\in\R^n$. Then $\Phi(P')=w$.
\end{example}

\begin{example}\label{ex:MapToClassicalAffineSpace}
Suppose $M=\N^n$, so $\Lambda=\Z^n$ and $\sigma$ is the totally negative orthant in $\R^n$. We can also give an explicit description of the map $\Phi:\ContBase{\base}\base[\Mon]\to N_{\R}(\sigma)=\T^n$ in this case. As above, we can pick a defining matrix for $P'$, but now some entries may be $-\infty$. Thus, we still have $w\in\T^n$, and $\Phi(P')=w$.
\end{example}

\begin{remark}
In the contexts of Examples~\ref{ex:MapToClassicalTorus} and \ref{ex:MapToClassicalAffineSpace}, $\Phi(P')$ can also be described as the unique point of the congruence variety of $P'$. See \cite[Section 3.1]{BE13} for the definition of the congruence variety.
\end{remark}

The following lemma allows us to reduce our discussion to primes in $\ContBase{\base}{\base[\Mon]}$  which are maximal with respect to inclusion.

\begin{lemma}\label{lem:claim1*}
Let $P , P',$ and $ \check{P}'$ be as above. For every term $m\in \base[\Mon]$ we have that $c|m|_P \geq |m|_{P'}$ for all $c\in \base$ with $c > 1_\base$ if and only if $c|m|_P \geq |m|_{\check{P}'}$ for all $c\in \base$ with $c > 1_\base$.
\end{lemma}

\begin{proof}
First assume that $c|m|_P \geq |m|_{P'}$ for every $c\in \base$, $c > 1_\base$, and for every term $m\in \base[\Mon]$. We will show that $c |m|_P \geq |m|_{\check{P}'} =\bigpi(|m|_{P'})$. Since $\base \subseteq \T$, we can find a $d \in \T$ such that $1_\T < d < c$. Then
$$c |m|_P = \frac{c}{d}\ d|m|_P \geq d|m|_{P'} >\bigpi(|m|_{P'})=|m|_{\check{P}'},$$
where the last inequality follows from Inequalities~(\ref{eq:bp1}). 
The proof of the reverse direction is similar.
\end{proof}

\begin{lemma}\label{lem:claim2*}
Let $P , P'$ be as previously defined with $P,P'\in\ContBase{\base}{\base[\Mon]}$ maximal. For every $c\in \base$ with $c > 1_\base$ we have that $c|m|_P \geq |m|_{P'}$ for all terms $m\in \base[\Mon]$ if and only if $|m|_P \geq |m|_{P'}$ for all terms $m\in \base[\Mon]$.
\end{lemma}

\begin{proof}
If $|m|_P \geq |m|_{P'}$, then $c|m|_P \geq |m|_{P'}$ since $c > 1_\base$. For the other direction, assume that for all $c\in \base$ with $c > 1_\base$ and every term $m\in \base[\Mon]$ we have that $c|m|_P \geq |m|_{P'}$. In particular, for all $l \in \Z_{\geq 1}$ we have $c|m^l|_P \geq |m^l|_{P'}$. 
Thus, in the algebraically closed semifield $\T$, we have 
$$c^{1/l}|m|_P = (c|m^l|_P)^{1/l} \geq (|m^l|_{P'})^{1/l}=|m|_{P'}.$$
By letting $l$ get arbitrarily large, we obtain the desired inequality $|m|_P \geq |m|_{P'}$ because $\T^\times$ is archimedean.
\end{proof}

We are now ready to state the geometric version of our main theorem of the section. 

\begin{theorem}[$\crown{}$g]\label{thm:crown_geom}
Let $\base\neq\B$ be a sub-semifield of $\T$ and let $\Mon$ be the toric monoid corresponding to a cone $\sigma$ in $N_{\R}$. Also, let $P \in \ContBaseInt{\base}{\base[\Mon]}$ and $P' \in \ContBase{\base}{\base[\Mon]}$. Then the map $\iota^*:\Cont(\Cnvg{P}) \rightarrow \ContBase{\base}{\base[\Mon]}$ induced by the inclusion $\iota: \base[\Mon] \rightarrow \Cnvg{P}$ is injective and the following statements are equivalent:

\begin{enumerate}
    \item[i)] $P'$ is in the image of $\iota^*$.
    \item[ii)] For all $c\in \base$, $c>1_\base$ and all terms $m\in \base[\Mon]$, $c|m|_P \geq |m|_{P'}$.
    \item[iii)] Letting $\bar{\sigma}$ be the closure of $\sigma$ in $N_{\R}(\sigma)$, we have $\Phi(P')\in\Phi(P)+_{\R}\bar{\sigma}$.
\end{enumerate}
\end{theorem}

\begin{proof}
The injectivity of $\iota^*$ and the equivalence of $(i)$ and  $(ii)$ follow from the statements of Theorem~\ref{thm:crown} and Proposition~\ref{prop:TwoVersionsOfProperty*AreSame}.

Now it remains only to show that $(ii)$ and $(iii)$ are equivalent. 
Note that $\Phi(P)=\Phi(\check{P})$ where $\check{P}$ is the maximal element of $\ContBase{\base}\base[\Mon]$ containing $P$, so replacing $P$ and $P'$ by the maximal elements of $\ContBase{\base}\base[\Mon]$ above them does not change the statement of $(iii)$.
In view of Lemma~\ref{lem:claim1*}, without loss of generality we can assume that $P'$ is a maximal prime in $\ContBase{\base}{\base[\Mon]}$ with respect to inclusion. 
We can also assume without loss of generality that $P$ is maximal: by Theorem \ref{thm:crown}, $P'$ satisfies $\propStar$ with respect to $P$ if and only if $P'$ is in the image of $\iota^*$ , but $\iota$, and hence $\iota^*$, is unchanged if we replace $P$ by the maximal prime in $\ContBase{\base}{\base[\Mon]}$ above it.
Moreover, if $P,P'\in\ContBase{\base}{\base[\Mon]}$ are maximal Lemma~\ref{lemma:MaximalGeorgePrimes} guarantees the existence of canonical embeddings of $\kappa(P)$ and $\kappa(P')$ into $\T$. 

Thus, since the primes $P$ and $P'$ are maximal in $\ContBase{\base}\base[\Mon]$, we have that $w:=\Phi(P)$ is the map $\Mon\to\base[\Mon]\to\kappa(P)\to\T$ and $w'=\Phi(P')$ is $\Mon\to\base[\Mon]\to\kappa(P')\to\T$. 
By Lemma~\ref{lem:claim2*}, (ii) is equivalent to having $|m|_P\geq|m|_{P'}$ for all terms $m\in \base[\Mon]$.
Let $m=a\chi^u$ be a nonzero term of $\base[\Mon]$.
Note that $|m|_{P}\geq|m|_{P'}$ says that $\angbra{w,u}+_{\R}\log(a)\geq\angbra{w',u}+\log(a)$, which is true if and only if $\angbra{w,u}\geq\angbra{w',u}$. Note that we can form the subtraction $w'-_{\R}w$ because $w\in N_{\R}$. So (ii) holds if and only if, for all $u\in\sigma^\vee$, $0 \geq \angbra{w'-_{\R}w,u}$. This, in turn, is equivalent to $w'-_{\R}w\in\bar{\sigma}$ by \cite[Proposition 3.19]{Rab10}. Finally, $w'-_{\R}w\in\bar{\sigma}$ is the same as $w'\in w+_{\R}\bar{\sigma}$.
\end{proof}

\begin{remark}
In view of Lemma~\ref{lem:claim2*} we can even allow $c \in \T$ in (ii).
\end{remark}


\section{The dimension of $\Cnvg{P}$}\label{sect: dim-Cnvg}

\GeorgeStory{In which George realizes that $\Q$ is $\Z$-flat.}

Throughout this section we continue to assume that $\base\neq\B$ is a sub-semifield of $\T$ and $\Mon = \sigma^{\vee} \cap \Lambda$ for $\Lambda$ a lattice and $\sigma$ a strongly convex rational polyhedral cone in $N_{\R}=\Lambda^*\otimes\R$.

In this section we study the dimension of the semiring of convergent power series at a prime. We compute it exactly when the coefficients are in $\T$. When the coefficients are in a strict sub-semifield of $\T$, we provide bounds on the dimension. In that case we give examples showing that both equality and strict inequality are possible. 

In \cite{JM17} the \emph{dimension} of a semiring $A$, denoted $\dim A$, is defined to be the number of strict inclusions in a chain of prime congruences on $A$ of maximal length. We refine this notion in the next definition. 

\begin{defi}
When $\base \subseteq \T$ is a sub-semifield and $A$ is a $\base$-algebra, the \emph{relative dimension of $A$ over} $\base$ is the number $\dim_\base A$ of strict inclusions in a longest chain of prime congruences in $\ContBase{\base}{A}$.

Moreover, if $A$ is a topological $\base$-algebra, the \emph{topological dimension of $A$} is the number $\dim_{\mathrm{top}} A$ of strict inclusions in a longest chain of prime congruences in $\Cont {A}$.

If there is no finite upper bound on the length of chains of prime congruences, we say that the dimension (topological or relative) is infinite.
\end{defi}

\begin{remark}
    We will be most interested in the cases when $A$ is a polynomial or Laurent polynomial semiring with coefficients in a sub-semifield of $\T$, the semirings of convergent power series, and their quotients. In these cases, there is always a longest (finite) chain of prime congruences. For polynomial or Laurent polynomial semiring this existence was proven in \cite[Theorem 4.14]{JM17} and for semirings of convergent power series we will see this in Corollary~\ref{coro:DimIneqsAroundCont}.
\end{remark}

We illustrate the distinction between dimension and relative dimension in the following examples.

\begin{example}
%
%
Let $\base$ be a sub-semifield of $\T$. Then, as per Example~\ref{ex: congSemif}, $\dim_\base \base$ is the number of nontrivial proper convex subgroups of $\base^\times$, which is 0. On the other hand, $\dim S$ is equal to the number of (possibly trivial) proper convex subgroups of $\base^\times $ which is 1.
\exEnd\end{example}


\begin{example}
Let $A=\T[x]$. Then $\dim A = 2$ but $\dim_\T A = 1$. The same happens when $\T$ is replaced by any sub-semifield. \exEnd\end{example}

We now state the main theorem of this section. 

\begin{theorem}\label{thm:DimIneqsGeneral}
Let $\base$ be a sub-semifield of $\T$, $M$ a toric monoid, and $P \in \ContBaseInt{\base}{\base[\Mon]}$ then 
\begin{equation*}
    \begin{split}
        \rk (\kappa(P)^\times/\base^\times) + \dim_{\mathrm{top}} \Cnvg{P} &\geq \dim_\base \base[\Mon]\\ &\geq \max \{\rk (\kappa(P)^\times/\base^\times), \dim_{\mathrm{top}} \Cnvg{P}\}.
    \end{split}
\end{equation*}
\end{theorem}

This immediately implies the following corollary.

\begin{coro}\label{coro:DimIneqsAroundCont}
Let $\base$ be a sub-semifield of $\T$, $\Mon$ be a toric monoid, and $P \in \ContBaseInt{\base}{\base[\Mon]}$. Then 
$$
\dim_{\base}\base[\Mon]\geq\dim_{\mathrm{top}}\Cnvg{P}\geq\dim_{\base}\base[\Mon] - \rk(\kappa(P)^\times/\base^\times).
$$
\end{coro}

\begin{remark}
If $M\cong\Z^n$, then, given a $k\times (n+1)$ matrix $C$ for $P\in\ContBaseInt{\base}{\base[\Mon]}$, one can compute $\rk(\kappa(P)^\times/\base^\times)$ as follows. Let $V$ denote the rational subspace of $\R^k$ generated by all but the first column of $C$. Let $W$ be the rational subspace of $\R^k$ consisting of vectors of the form $\begin{pmatrix}q\alpha\\0\\\vdots\\0\end{pmatrix}$ with $q\in\Q$ and $t^\alpha\in\base$. Then $\rk(\kappa(P)^\times/\base^\times)=\dim_{\Q}(V+W)/W$.
\end{remark}

We also have the following corollary, computing the topological dimension of $\Cnvg{P}$ exactly in the case of $\T$ coefficients.

\begin{coro}\label{coro:DimWithTCoeffs}
If $P \in \ContBaseInt{\T}{\T[\Mon]}$, then
$$\dim_{\mathrm{top}} \Cnvg{P} = \dim_\T \T[\Mon].$$
\end{coro}\begin{proof}
Let $\check{P}$ be the unique maximal element of $\ContBaseInt{\T}{\T[\Mon]}$ over $P$. Then $\Cnvg{P}=\Cnvg{\check{P}}$ and $\kappa(\check{P})^\times/\T^\times$ is the trivial group. The result follows immediately from Theorem \ref{thm:DimIneqsGeneral}.
\end{proof}


We will first prove Theorem~\ref{thm:DimIneqsGeneral} in the case when $\Mon$ is a group. Since we have the running hypothesis that $\Mon$ is a toric monoid, this means $\Mon\cong\Z^n$. In this case, the statement of Theorem~\ref{thm:DimIneqsGeneral} reduces to the following proposition.

\begin{proposition}\label{prop:dim-group}

Let $\base$ be a sub-semifield of $\T$, let $\Mon$ be a group, and let $P \in \ContBaseInt{\base}{\base[\Mon]}$ be maximal. 
Then 
$$\rk (\kappa(P)^\times/\base^\times) + \dim_{\mathrm{top}} \Cnvg{P} = \dim_\base \base[\Mon].$$
\end{proposition} 

First we need the following easy lemma. We let $\Terms{\base[\Mon]}$ denote the multiplicative monoid of (nonzero) terms of $\base[\Mon]$; if $\Mon$ is a group, then $\Terms{\base[\Mon]}$ is a group. Note that, if $A$ is a semiring and $\mathcal{G}$ is a set of additive generators for $A$, then, for any prime congruence $P$ on $A$, every equivalence class in $A/P$ has a representative in $\mathcal{G}\cup\{0_A\}$.

\begin{lemma}\label{lem:QisZ-flat}
Let $\base$ be a sub-semifield of $\T$, let $\Mon$ be a group, and let $P \in \ContBaseInt{\base}{\base[\Mon]}$. 
Letting $\pi_P:\Terms{\base[\Mon]}\to\kappa(P)^\times$ be the natural map, we have 
$$\rk(\ker(\pi_P)) + \rk(\kappa(P)^\times/\base^\times) = \rk(\Mon).$$
\end{lemma}

\begin{proof}
Note that $\Terms{\base[{\Mon}]}$ is just $\base^\times \times \Mon$. 
Since every nonzero equivalence class in $\base[\Mon]/P$ has a representative which is a unit, $\base[\Mon]/P$ is a semifield, so $\base[\Mon]/P=\kappa(P)$. 
This results in the  short exact sequence
\begin{equation*}
  0\rightarrow \ker(\pi_P) \rightarrow \Terms{\base[{\Mon}]} \rightarrow \kappa(P)^\times \rightarrow 0.  
\end{equation*}
Since $\pi_P$ maps $\base^\times$ injectively into $\kappa(P)^\times$ we can quotient out by $\base^\times$ to get
\begin{equation}\label{eq:ab-groups}
    0\rightarrow \ker(\pi_P) \rightarrow \Mon \rightarrow \kappa(P)^\times/\base^\times \rightarrow 0.
\end{equation}
We can take the ranks of the abelian groups in sequence (\ref{eq:ab-groups}) to obtain 
\begin{equation*}
   \rk(\ker(\pi_P)) + \rk(\kappa(P)^\times/\base^\times) = \rk(\Mon).
\end{equation*}
\end{proof}

In order to construct the chains of primes necessary to show Proposition~\ref{prop:dim-group}, we use the following proposition.

\begin{proposition}\label{prop: side_claim}

Let $\Mon$ be a group. A relation $\preceq$ on the terms of $\base[\Mon]$ is the same as $\leq_P$ for a prime $P$ on $\base[\Mon]$ if and only if the following conditions hold:
\begin{enumerate}
    \item[(1)] $\preceq$ is a total preorder, 
    \item[(2)] $\preceq$ is multiplicative, and
    \item[(3)] $\preceq$ respects the order on $\base$.
\end{enumerate}
Moreover, $P\in\ContBase{\base}{\base[\Mon]}$ if and only if
\begin{enumerate}[resume]
    \item[(4)] for any term $a\chi^u$ of $\base[\Mon]$, $\exists b \in \base^\times$, such that $b \prec a\chi^u$.
\end{enumerate}
\end{proposition}
\begin{proof}
By definition the relation $\leq_P$ satisfies all these conditions, so it suffices to show the opposite direction. 
Conditions (1), (2) and (3) guarantee that the quotient of $\Terms{\base[\Mon]}$ by the equivalence relation $a\sim b$ if and only if $a \preceq b$ and $a \succeq b$ is a totally ordered abelian group. Thus, the congruence $P$ that $\sim$ generates on $\base[\Mon]$ has $\base[\Mon]/P$ a totally ordered semifield, giving us that $P$ is a prime congruence.
The last condition 
is equivalent to the congruence being in
$\ContBase{\base}{\base[\Mon]}$ by Lemma~\ref{lemma:ContinuityGivesGeorge}.
\end{proof}

\begin{proof}[Proof of Proposition~\ref{prop:dim-group}]
For $P \in \ContBaseInt{\base}{\base[\Mon]}$ let $\pi_P$ be the natural map from 
$\Terms{\base[{\Mon}]}$ 
to the multiplicative group $\kappa(P)^\times$ of the residue semifield. 
From Lemma~\ref{lem:QisZ-flat} we already know that $\rk(\ker(\pi_P)) + \rk(\kappa(P)^\times/\base^\times) = \rk(\Mon)$. 
By \cite[Theorem 3.16]{JM15}, $\dim_{\base}\base[\Mon]=\rk(M)$.
Thus, it remains only to show that $ \rk(\ker(\pi_P)) = \dim_{\mathrm{top}} \Cnvg{P}.$
To show that  $\rk(\ker(\pi_P))\geq\dim_{\mathrm{top}} \Cnvg{P}$, let $P' \subseteq P$ and consider the diagram
\begin{center}
\begin{tikzcd}0 \arrow[r, ""] &\ker(\pi_{P'})\arrow[r, "\iota"]\arrow[d, ""]& 
\Terms{\base[{\Mon}]}\arrow[r, "\pi_{P'}" ] 
& \kappa(P')^\times \arrow[r, ""]& 0\\0 \arrow[r,""] & \ker(\pi_P) \arrow[ur, ""] \end{tikzcd}
\end{center}
in which the first row is exact.

Since $\kappa(P')^\times$ is a totally ordered abelian group, it is torsion-free, so $\ker(\pi_{P'})$ is saturated in 
$\Terms{\base[{\Mon}]}$. 
Since $\ker(\pi_P)$ is a subgroup of 
$\Terms{\base[{\Mon}]}$ 
containing $\ker(\pi_{P'})$, we get that $\ker(\pi_{P'})$ is saturated in $\ker(\pi_P)$. Thus, if $P'\subsetneq P$, then $\rk(\ker(\pi_{P'}))<\rk(\ker(\pi_P))$. In particular, if we have a chain of primes
$$P \supsetneq P' \supsetneq P'' \supsetneq \ldots \supsetneq P^{(n)},$$
then 
$$\rk(\ker(\pi_P))>\rk(\ker(\pi_{P'}))>\cdots>\rk(\ker(\pi_{P^{(n)}}))\geq0,$$
so $\rk(\ker(\pi_P))\geq n$. 
By Corollary~\ref{coro:crown_torus}, pullback along the inclusion map $\base[\Mon]\into\Cnvg{P}$ gives a bijection from $\Cont (\Cnvg{P})$ to the primes contained in $P$. 
So $\rk(\ker(\pi_P))\geq\dim_{\mathrm{top}} \Cnvg{P}$.

\newcommand{\frakE}{\mathfrak{E}}

Towards proving the opposite inequality, let $\frakE:\Terms{\base[\Mon]}\to\Mon$ be the group homomorphism taking a term to its exponent, i.e., $\frakE(a\chi^u)=u$. In particular, the kernel of $\frakE$ is $\base^\times$. So, because $\ker(\pi_P)\cap\base^\times=\{1_\base\}$, $\rk\big(\frakE(\ker(\pi_P))\big)=\rk(\ker(\pi_P))$.

As usual, let $N = \Hom_\ZZ(M, \ZZ)$. For any $w \in N\setminus \frakE(\ker(\pi_P))^\perp$, we can specify a prime $P' \subsetneq P$ by the multiplicative preorder $\leq_{P'}$ on $\text{Terms}\ \base[{\Mon}]$ given as follows:
$$a\chi^u \leq_{P'} b\chi^v \Longleftrightarrow a\chi^u <_P b\chi^v \text{ or } (a\chi^u \equiv_P b\chi^v \text{ and } \left< w,u \right>\leq  \left< w,v \right>).$$
We claim that for such a pair $P$ and $P'$ we have that $\rk(\ker(\pi_{P'})) = \rk(\ker(\pi_{P}))-1.$ To see this, note that $a\chi^u \in \ker(\pi_{P'})$ if and only if $a\chi^u \in \ker(\pi_{P})$ and $\left< w,u \right>=0 $. Thus, $\frakE(\ker(\pi_{P'})) = \frakE(\ker(\pi_{P})) \cap w^\perp$.
Since $w$ is not 0 on $\frakE(\ker(\pi_{P}))$, we have 
$$
\rk(\ker(\pi_{P'}))=\rk\left(\frakE(\ker(\pi_{P})) \cap w^\perp\right)
= \rk\left(\frakE(\ker(\pi_{P})\right))-1
=\rk(\ker(\pi_P))-1,
$$
as claimed.  
Now denote by $k$ the rank of $\ker(\pi_P)$ and observe that $k = 0$ if and only if $N = \frakE(\ker(\pi_P))^\perp$. So we can apply the above construction recursively $k$ times to obtain a chain of primes 
$$P \supsetneq P' \supsetneq P'' \supsetneq \ldots \supsetneq P^{(k)}.$$
Thus, $ \dim_{\mathrm{top}} \Cnvg{P} \geq k = \rk(\ker(\pi_P))$ and so, combining the two inequalities, we get the desired statement.
\end{proof}

\begin{remark}\label{remark:InterpretKerPiP}
The above proof also gives a result in the case where $P\in\ContBaseInt{\base}\base[\Mon]$ is not necessarily maximal. Here we are still assuming that $\base$ is a subsemifield of $\T$ and $\Mon$ is a group. In this case, the proof shows that $\rk(\ker(\pi_P))$ is the maximum length of a chain under $P$, where $\pi_P:\Terms{\base[\Mon]}\to\kappa(P)^\times$ is the natural map.


\end{remark}

\begin{defi}
The \emph{height} of a prime congruence $P$ on a semiring $A$ is the maximum length $\HT(P)$ of a chain $P_0\subsetneq P_1\subsetneq\cdots\subsetneq P_{k}=P$ of primes under $P$.
\end{defi}

With this terminology, Remark~\ref{remark:InterpretKerPiP} says that, when $\base$ is a sub-semifield of $\T$ and $\Mon\cong\Z^n$, the height of any $P\in\ContIntBase{\base}\base[\Mon]$ is $\HT(P)=\rk(\ker \pi_P)=n-\rk(\kappa(P)^\times/\base^\times)$. We can also extend this to the case of a toric monoid.

\begin{lemma}\label{lemma:HeightInToricCase}
Let $\base$ be  sub-semifield of $\T$ and let $\Mon\subseteq\Lambda$ be a toric monoid corresponding to a cone $\sigma$ in $N_{\R}$. For any $P\in\ContIntBase{\base}\base[\Mon]$, $\HT(P)=\rk(\Lambda)-\rk(\kappa(P)^\times/\base^\times)$.
\end{lemma}\begin{proof}
Because $P$ has trivial ideal-kernel, it extends to a unique prime $\widetilde{P}\in\ContBaseInt{\base}\base[\Lambda]$ and the same applies to all of the primes contained in $P$. Since all the primes below $\widetilde{P}$ arise this way, we have $\HT(P)=\HT(\widetilde{P})=n-\rk(\kappa(\widetilde{P})^\times/\base^\times)=n-\rk(\kappa(P)^\times/\base^\times)$.
\end{proof}

We can also extract another result for the case of toric monoids. First we require a lemma.

\begin{lemma}\label{lemma:PreorderBeingAnOrderDeterminedOnTorus}
Let $\base$ be  sub-semifield of $\T$ and let $\Mon\subseteq\Lambda$ be a toric monoid corresponding to a cone $\sigma$ in $N_{\R}$. Fix $P\in\ContIntBase{\base}\base[\Mon]$ and let $\wt{P}$ be the corresponding prime congruence on $\base[\Lambda]$. Then the preorder $\leq_P$ on $\Terms{\base[\Mon]}$ is an order if and only if $\leq_{\wt{P}}$ is an order on $\Terms{\base[\Lambda]}$.
\end{lemma}\begin{proof}
Note that $\leq_{P}$ and $\leq_{\wt{P}}$ are orders if and only if the functions $\pi_P:\Terms{\base[\Mon]}\to\kappa(P)^\times$ and $\pi_{\wt{P}}:\Terms{\base[\Lambda]}\to\kappa(\wt{P})^\times$, respectively, are injective. 

Under the identification $\kappa(P)\cong\kappa(\wt{P})$, $\pi_P$ is the restriction of $\pi_{\wt{P}}$ to $\Terms{\base[\Mon]}\subseteq\Terms{\base[\Lambda]}$, so if $\pi_{\wt{P}}$ is injective then so is $\pi_{P}$.

Now suppose $\pi_{P}$ is injective and $a\chi^u, a'\chi^{u'}\in\Terms{\base[\Lambda]}$ are such that $\pi_{\wt{P}}(a\chi^u)=\pi_{\wt{P}}(a'\chi^{u'})$. Since $\sigma^\vee$ is a full-dimensional rational (not necessarily strongly convex) cone in $\Lambda_{\R}$, we can pick $\mu\in\Lambda$ which is in the interior of $\sigma^\vee$. So for any sufficiently large $r\in\N$, $u+r\mu,u'+r\mu\in\sigma^\vee\cap\Lambda=\Mon$. Now
\begin{align*}
\pi_{P}(a\chi^{u+r\mu})
&=\pi_{\wt{P}}(a\chi^{u})\pi_{\wt{P}}(\chi^{r\mu})\\
&=\pi_{\wt{P}}(a'\chi^{u'})\pi_{\wt{P}}(\chi^{r\mu})\\
&=\pi_{P}(a'\chi^{u'+r\mu}),
\end{align*}
so $a\chi^{u+r\mu}=a'\chi^{u'+r\mu}$. Thus $a=a'$ and $u=u'$, showing that $\pi_{\wt{P}}$ is injective.
\end{proof}

\begin{coro}
Let $\base$ be  sub-semifield of $\T$ and let $\Mon\subseteq\Lambda$ be a toric monoid corresponding to a cone $\sigma$ in $N_{\R}$. For any $P\in\ContBase{\base}\base[\Mon]$, the preorder $\leq_P$ on $\Terms{\base[\Mon]}$ is an order if and only if $P\in\ContIntBase{\base}\base[\Mon]$ and $\HT(P)=0$.
\end{coro}
\begin{proof}
Note that $\leq_P$ is an order if and only if the map $\Terms{\base[\Mon]}\to\kappa(P)$ is injective. Since the ideal-kernel of $\base[\Mon]\to\kappa(P)$ is generated by its intersection with $\Terms{\base[\Mon]}$, whenever $\leq_P$ is an order we have $P\in\ContBaseInt{\base}\base[\Mon]$. Therefore, it suffices to show that, if $P$ has trivial ideal-kernel, then $\leq_P$ is an order if and only if $\HT(P)=0$.

When $P\in\ContBaseInt{\base}\base[\Mon]$, it corresponds to a prime $\wt{P}$ on $\base[\Lambda]$. By Lemma~\ref{lemma:PreorderBeingAnOrderDeterminedOnTorus}, $\leq_P$ is an order if and only if $\leq_{\wt{P}}$ is an order. This happens exactly if the group homomorphism $\pi_{\wt{P}}:\Terms{\base[\Lambda]}\to\kappa(P)^\times$ is injective, i.e., 
$$0=\rk(\ker(\pi_{\wt{P}}))=\HT(\wt{P})=\HT(P).$$
\end{proof}



We now continue with the toric case and move towards proving Theorem~\ref{thm:DimIneqsGeneral}. In studying chains of primes, we will use the stratification given in the following proposition.

\begin{prop}\label{prop:AdicToricStratification}
Let $\base$ be a sub-semifield of $\T$, let $\Lambda$ be a finitely generated free abelian group, 
let $\sigma$ be a cone in $N_{\R}$, 
and let $\Mon = \sigma^\vee\cap\Lambda$. There is a partition of $\ContBase{\base}\base[\Mon]$ into subsets $V_{\tau}$ for $\tau\leq\sigma$ such that
\begin{enumerate}
\item if $P\subseteq P'$ are in $\ContBase{\base}\base[\Mon]$, then they are in the same $V_{\tau}$,
\item $V_\tau\cong\ContBase{\base}\base[\Lambda\cap\tau^\perp]$, and
\item under the map $\Phi:\ContBase{\base}\base[\Mon]\to N_{\R}(\sigma)$ defined in Section~\ref{sect: geom-crown}, $V_\tau$ is the inverse image of the stratum of $N_{\R}(\sigma)$ corresponding to $\tau$.
\end{enumerate}
Moreover, this partition is given by the equivalence relation on $\ContBase{\base}\base[\Mon]$ given by $P\sim P'$ if $P$ and $P'$ have the same ideal-kernel.
\end{prop}

\begin{remark}
Recall that the stratification of $N_{\R}(\sigma)$ is defined as follows. Given $\tau\leq\sigma$, the stratum of $N_{\R}(\sigma)=\Hom(\Mon,\T)$ corresponding to $\tau$ consists of those $\phi:\Mon\to\T$ such that $\phi^{-1}(0_\T)=\{u\in\Mon\;:\;u\notin\tau^\perp\}$.
\end{remark}

\begin{proof}[Proof of Proposition~\ref{prop:AdicToricStratification}]
For $\tau\leq\sigma$ we let $K_\tau$ be the set of those terms $a\chi^u$ in $\base[\Mon]$ such that $u\notin\tau^\perp$ and let $\calK_\tau=\left\langle m\sim 0_{\base}\;:\; m\in K_\tau\right\rangle$ be the congruence on $\base[\Mon]$ generated by setting each element of $K_\tau$ to be $0_{\base}$. For the zero cone $\tau=\{0_{N_{\R}}\}$, we have  $K_\tau=\emptyset$ and $\calk_\tau=\Delta$ is the trivial congruence on $\base[\Mon]$.

For any prime $P$ on $\base[\Mon]$ there is a face $\tau \leq \sigma$ such that the restriction of the ideal-kernel of $P$ to the set of terms $\Terms{\base[\Mon]}$ is $K_\tau$. We let $V_\tau$ denote the set of such $P\in\ContBase{\base}\base[\Mon]$. 

If $P\in V_\tau$ then $P\supseteq\calk_\tau$ and, by Remark~\ref{remark:IdealKernelAndGenerators}, the ideal-kernel of $\calK_\tau$ is the same as the ideal-kernel of $P$. Thus, for $P\in\ContBase{\base}\base[\Mon]$, $P\in V_\tau$ if and only if $P\supseteq\calk_\tau$ and the prime $\wt{P}$ on $\base[\Mon]/\calk_\tau$ corresponding to $P$ has trivial ideal-kernel. By Proposition~\ref{prop:primes-total-sFrac}, this exactly says that $P$ corresponds to a prime congruence on $\Frac(\base[\Mon]/\calk_\tau)$. So we have
\begin{align*}
V_\tau&\cong\ContBase{\base}\Frac(\base[\Mon]/\calk_\tau)\\
&\cong\ContBase{\base}\Frac(\base[\sigma^\vee\cap\Lambda\cap\tau^\perp])\\
&\cong\ContBase{\base}\base[\Lambda\cap\tau^\perp].
\end{align*}
This proves (2).

If $P\subseteq P'$ are in $\ContBase{\base}\base[\Mon]$ then Corolary~\ref{coro:PrimesInChainHaveSameIdealKernel} tells us that $P$ and $P'$ have the same ideal-kernel, so they are in the same $V_\tau$. So we have shown (1).

Towards (3), recall that $\Phi$ is defined as follows. Given $P$, let $\check{P}$ be the maximal element of $\ContBase{\base}\base[\Mon]$ containing $P$, so there is a canonical embedding $\kappa(\check{P})\into\T$. Then $\Phi(P)$ is the map $\Mon\into\base[\Mon]\toup{\pi_{\check{P}}}\kappa(\check{P})\into\T$. Note that $a\chi^u$ is in the ideal-kernel of $P$ if and only if it is in the ideal-kernel of $\check{P}$, which in turn is equivalent to 
$u$ being in $\big(\Phi(P)\big)^{-1}(0_{\T})$. 
Thus $P\in V_\tau$ if and only if $\Phi(P)$ is in the stratum of $N_{\R}(\sigma)$ corresponding to $\tau$.
\end{proof}


\begin{proposition}\label{prop:rk-of-lattice-dim-of-Ssigma}
Let $\base$ be a sub-semifield of $\T$, let $\Lambda$ be a finitely generated free abelian group, 
let $\sigma$ be a cone in $N_{\R}$, 
and let $\Mon = \sigma^\vee\cap\Lambda$. Then $$\dim_\base\base[\Mon]=\rk \Lambda.$$
\end{proposition}

\begin{proof}
Let $\{V_\tau\ :\ \tau\leq\sigma\}$ be the partition of $\ContBase{\base}\base[\Mon]$ from Proposition~\ref{prop:AdicToricStratification}. By part (1) of that proposition, $\dim_\base \base[\Mon]$ is the maximum over $\tau$ of the largest possible number of strict inclusions in a chain of primes in $V_\tau$. Part (2) then implies that this is $\dim_\base \base[\Mon]=\max\{\dim_\base\base[\Lambda\cap\tau^\perp]\ : \ \tau \leq \sigma \}$. By \cite[Theorem 3.16]{JM15} we know that $\dim_{\base}\base[\Lambda\cap\tau^\perp]=\rk(\Lambda\cap\tau^\perp)=\rk(\Lambda)-\dim\tau$, so we conclude that $\dim_\base \base[\Mon]=\rk(\Lambda)$.
\end{proof}

\begin{proof}[Proof of Theorem \ref{thm:DimIneqsGeneral}]
Write $\Mon =\sigma^\vee\cap\Lambda$ for some cone $\sigma$ in $N_{\R}$. Since $P$ has trivial ideal-kernel, it has a unique extension $\widetilde{P}$ to $\base[\Lambda]$. The induced homomorphism $\kappa(P)\to\kappa(\widetilde{P})$ is an isomorphism. With this notation, we have the following diagram:

\begin{center}
\begin{tikzcd}\text{Terms}\ \base[{\Mon}]\arrow[hookrightarrow, r, ""]\arrow[d, ""]& \text{Terms}\ \base[{\Lambda}] \arrow[d, "\pi_{\widetilde{P}}"]  \\ (\base[\Mon]/P)\setminus \{0_S\} \arrow[hookrightarrow, r, ""] &  \kappa(\widetilde{P})^\times  
\end{tikzcd}
\end{center}
%
From Lemma~\ref{lem:QisZ-flat}, we know that $\rk \Lambda = \rk (\ker\pi_{\widetilde{P}}) + \rk (\kappa(P)^\times/S^\times),$ but $\rk \Lambda = \dim_\base\base[\Mon]$ by Proposition~\ref{prop:rk-of-lattice-dim-of-Ssigma}, implying that $\dim_\base\base[\Mon] \geq \rk (\kappa(P)^\times/S^\times)$. Since, by Theorem~\ref{thm:crown}, the map $\iota^*:\Cont(\Cnvg{P})\to \ContBase{\base}{\base[\Mon]}$ induced by $\iota:\base[\Mon]\to\Cnvg{P}$ is an injection,
we also have that $\dim_\base\base[\Mon] \geq \dim_{\mathrm{top}} \Cnvg{P}$. Combining these statements we get the second half of the theorem, namely 
$$\dim_\base\base[\Mon] \geq \max \{\dim_{\mathrm{top}} \Cnvg{P}, \rk (\kappa(P)^\times/S^\times)\}.$$
It remains to show the inequality $\rk (\kappa(P)^\times/\base^\times) + \dim_{\mathrm{top}} \Cnvg{P} \geq \dim_\base \base[\Mon]$. We have $\HT(P)=\rk(\Lambda)-\rk(\kappa(P)^\times/\base^\times)$ by Lemma~\ref{lemma:HeightInToricCase}.
Considering the canonical extensions of primes under $P$ to $\Cnvg{P}$ as in Example~\ref{ex:CanonicalExtUnderP}, we get $\dim_{\mathrm{top}}\Cnvg{P}\geq\HT(Q_P)=\HT(P)=\rk(\Lambda)-\rk(\kappa(P)^\times/\base^\times)$, where $Q_P$ is the canonical extension of $P$.
\end{proof}

Corollary~\ref{coro:DimWithTCoeffs} provides us with examples where equality is attained in both inequalities of Theorem~\ref{thm:DimIneqsGeneral}. The following examples show that we can have a strict inequality on one side with equality on the other or strict inequality on both sides.

\begin{example}\label{ex:DimEx1}
Let $\base = \Q_{\max}$ and $A = \Q_{\max}[x]$, corresponding to the cone $\sigma=\R_{\leq0}\subseteq\R$. 
By Proposition~\ref{prop:rk-of-lattice-dim-of-Ssigma}, $\dim_\base A = 1$. 
Take $P$ to be the prime congruence on $A$ with defining matrix 
$\begin{pmatrix}1& \sqrt2\end{pmatrix}$. 
The residue semifield $\kappa(P) $ is $(\Q + \Z\sqrt{2})_{\max}$ and so $\rk(\kappa(P)^\times / \base^\times) =1$. We claim that $\dim_{\mathrm{top}} \Cnvg{P}=1$. Since we have $\dim_{\mathrm{top}}\Cnvg{P}\leq\dim_{\base}A=1$ it suffices to exhibit a chain of primes of length one on $\Cnvg{P}$. By Theorem~\ref{thm:crown_geom}, the primes on $A=\Q_{\max}[x]$ defined by the matrices $\begin{pmatrix}1&0\end{pmatrix}$ and $\begin{pmatrix}1&0\\0&1\end{pmatrix}$ extend to primes of $\Cnvg{P}$.\exEnd\end{example}

\begin{example}\label{ex:DimEx2}
Let $\base = \Q_{\max}$ and let $\sigma = \R_{\geq 0}(-1, -1) \subseteq \R^2$. In this case, we have 
$$A = \Q_{\max}[\sigma^\vee\cap\Z^2] = \Q_{\max}[xy, (xy^{-1})^{\pm1}].$$
By Proposition~\ref{prop:rk-of-lattice-dim-of-Ssigma}, $\dim_\base A =2$. 
Take $P$ to be the prime congruence determined by the matrix $\begin{pmatrix} 1 & \sqrt2 & \sqrt3 \end{pmatrix}$. For this $P$ we have that $\rk(\kappa(P)^\times / \base^\times) =2$. 
However, one can show that $\dim_{\mathrm{top}} \Cnvg{P} = 1$. To see the last statement, let $P'$ be the prime on $A$ with defining matrix $\begin{pmatrix} 1& 0&\sqrt3-\sqrt2 \end{pmatrix}$. Here, we get $P'$ by letting the matrix define a prime on $\base[\Z^2]$ and then pulling back along the inclusion $A\into\base[\Z^2]$. We know that $P'$ extends to a prime congruence on $\Cnvg{P}$ by Theorem~\ref{thm:crown_geom}. By
Lemma~\ref{lemma:HeightInToricCase},
$\HT(P')=2-\rk(\kappa(P')^\times/\base^\times)=2-1=1$, showing that $\dim_{\mathrm{top}}\Cnvg{P}\geq1$. 

To see that we cannot get a longer chain, consider any prime $P'$ on $A$ corresponding to a maximal prime $Q'\in\Cont(\Cnvg{P})$. If $Q'\in\ContInt(\Cnvg{P})$ then $P'$ comes from a point of $(\sqrt{2},\sqrt{3})+\R_{\geq0}(-1,-1)$. Any such point has at least one irrational coordinate, giving us $\rk(\kappa(P')^\times/\base^\times)\geq1$, and so $\HT(Q')=\HT(P')\leq 1$. On the other hand, if $Q'\notin\ContInt(\Cnvg{P})$, then $P'$ corresponds to a prime of a toric $\base$-algebra  of strictly smaller dimension than $A$, as in the proof of Proposition~\ref{prop:rk-of-lattice-dim-of-Ssigma}. So in this case we also get that $\HT(Q')=\HT(P')\leq1$. Thus, $\dim_{\mathrm{top}}\Cnvg{P}\leq 1$. \exEnd\end{example}

\begin{example}\label{ex:DimEx3}
Let $\base = \Q_{\max}$ and let $\sigma = \R_{\geq 0}(-1, -1, -1) \subseteq \R^3$. Take $A $ to be the semiring $\Q_{\max}[\sigma^\vee\cap \Z^3]$. In this case $\dim_\base A = 3$. Let $P$ be the prime defined by the matrix $\begin{pmatrix}1& \sqrt2 & \sqrt2 & \sqrt3 \end{pmatrix}$ so $\rk(\kappa(P)^\times / \base^\times) =2$. Similarly to the previous example, we get that the prime $P'$ defined by the matrix $\begin{pmatrix}1& 0& 0& \sqrt3-\sqrt2 \end{pmatrix}$ extends to a $Q'\in\Cont(\Cnvg{P})$. We can compute that $\HT(Q')=2$, and the same type of argument as in the previous example then shows that there is no longer chain. Thus, we see that $\dim_{\mathrm{top}}\Cnvg{P}=2$. \exEnd \end{example}

The last statement of this section is a refinement of Theorem~\ref{thm:DimIneqsGeneral} in the case when $\Mon = \sigma^\vee\cap\Lambda$ and $\sigma$ is a full-dimensional cone.

\begin{proposition}
Let $\base$ be a sub-semifield of $\T$, $\Mon = \sigma^\vee\cap\Lambda$, where $\sigma$ is a full-dimensional cone in $N_{\R}$. If $P \in \ContBaseInt{\base}{\base[\Mon]}$ then 
$$\dim_{\mathrm{top}} \Cnvg{P}=\dim_\base\base[\Mon].$$
\end{proposition}

\begin{proof} 

Given $P$, let $\omega$ be the image of $P$ in $N_{\R}$. Theorem~\ref{thm:crown_geom} tells us that the maximal $P'\in\ContBaseInt{\base}{\base[\Mon]}$ correspond to the points of $\omega+\bar{\sigma}$. 
Let $\Gamma:=\log(\base^\times)$ be the subgroup of $\R$ corresponding to $\base$.
Since $\sigma$ is full-dimensional, $\omega+\sigma$ contains a $\Gamma$-rational point $\omega'$, i.e., a point of $N_{\Gamma}=N\otimes\Gamma$. 
%
%
Let $P'$ be the maximal element of $\Cont(\base[\Mon])$ corresponding to $\omega'$, so  we have that $\kappa(P')=\base$.
By 
Lemma~\ref{lemma:HeightInToricCase}, 
we get $\HT(P')=\rk(\Lambda)-\rk(\kappa(P')^\times/\base^\times)=\rk(\Lambda)$.
Together with Proposition~\ref{prop:rk-of-lattice-dim-of-Ssigma}, this implies that $\dim_{\mathrm{top}}\Cnvg{P}\geq\rk(\Lambda)=\dim_{\base}\base[\Mon]$. Theorem~\ref{thm:DimIneqsGeneral} gives the opposite inequality, so we obtain that
$\dim_{\mathrm{top}} \Cnvg{P}=\dim_\base\base[\Mon]$.
\end{proof}

\appendix
\section{Totally ordered abelian groups and Hahn's embedding theorem}

\subsection{Totally ordered abelian groups}\label{app:TotOrdAbGps1}
Recall that a totally ordered abelian group is an abelian group with a total order such that addition preserves the order.
For any totally ordered abelian group $G$, written additively, and for any $a,b\in G$ with $a,b> 0_G$, we write $a\gg b$ whenever $a>n\cdot b$ for all natural numbers $n$. 

We say that $a,b> 0_G$ are \emph{archimedean equivalent}, and write $a\sim b$, if neither $a\gg b$ nor $b\gg a$. This gives an equivalence relation, whose equivalence classes are called the \emph{archimedean equivalence classes} of $G$. We order the set of archimedean classes $[a]$ of $G$ ``backwards'', i.e., $[a]<[b]$ if $a\gg b$. 

The following definitions can be made in greater generality of nonabelian groups, but we make the definitions only in the context in which we use them.

\begin{defi}
A \textit{convex subgroup} $H$ of a totally ordered abelian group $G$ is a subgroup of $G$ which is also a convex subset with respect to the order.
\end{defi}

See \cite[Remark and Definition 1.7]{Wed12} for equivalent definitions of a convex subgroup.

\par\medskip
Let $G$ and $G'$ be totally ordered abelian groups. A homomorphism of totally ordered abelian groups from $G$ to $G'$ is a group homomorphism $\ph:G\to G'$ which preserves the order. The congruence-kernel of such a homomorphism is a convex subgroup of $G$.

\begin{defi}
A totally ordered abelian group  $G$, is called \emph{archimedean} if, for all $a,b>0_G$, there exists an integer $n\in\N$ such that $n\cdot a\geq b$. That is, $G$ is archimedean if it has only one archimedean equivalence class.
\end{defi}

Archimedean totally ordered abelian groups have no nontrivial proper convex subgroups.\par\medskip

The following classical statement is a consequence of Theorem 1 and Proposition 2 in Chapter 4 of \cite{Fuc63}.
\begin{theorem}\label{thm: rank1-emb}
Let $G$ be an archimedean totally ordered abelian group. Then there is an injective homomorphism of totally ordered abelian groups from $G$ into the additive group of $\R$ (with the usual ordering). Moreover, this embedding is unique up to multiplication by a scalar.
\end{theorem}

The following statement is well-known, but the details of the construction can be found in \cite{Gau98}. 
Given a totally ordered abelian group $(G,+,\leq)$, the set $G \cup \{-\infty\}$ equipped with the operations $(\max, +)$ forms an totally ordered semifield. 
This plays a key role in the following theorem, whose proof is a routine exercise. Let $\TOS$ denote the category of totally ordered semifields with semiring homomorphisms, and let $\TOAG$ denote the category of totally ordered abelian groups with order-preserving group homomorphisms.

\begin{thm}\label{thm:TOSAndTOAGEquiv}
The map taking a totally ordered semifield $S$ to the totally ordered abelian group $S^\times$ is an equivalence of categories $\TOS\to\TOAG$. An inverse functor of this functor is given by $G\mapsto G\cup\{-\infty\}$.
\end{thm}

\subsection{Hahn's embedding theorem}\label{app:TotOrdAbGps2}

Let $G$ be a totally ordered abelian group, written additively. Recall that, for $x\in G$ with $x>0_G$, $[x]$ is the archimedean class of $x$. Also, for any $y\in G$, we have $|y|:=\max(y,-y)$.

Let $I$ be a totally ordered set. The \emph{lexicographic function space on $I$}, denoted $\R^{(I)}$, is the subgroup of $\R^I$ consisting of functions $f:I\to\R$ such that 
$$\supp(f):=\{i\in I: f(i)\neq0_{\R}\}\subseteq I$$
is well-ordered. Using the lexicographic order makes $\R^{(I)}$ a totally ordered abelian group. 
This order is given by specifying that $f<g$ exactly when $f(i_0)<g(i_0)$, where $i_0$ is the least element of $\{i\in I: f(i)\neq g(i)\}$.

\begin{thm}[Hahn's embedding theorem$+\eps$]\label{thm:Hahns-embedding}
Let $G$ be a totally ordered abelian group, and let $I$ be the set of archimedean equivalence classes of $G$. Then there is an embedding $\ph:G\into\R^{(I)}$.
Furthermore, suppose that we have an archimedean abelian group $G'$ and embeddings $\iota:G'\to G$, $\psi:G'\to\R$. 
Let $i_0\in I$ be the common archimedean class of all positive elements of $\iota(G')$, let $\pi:\R^{(I)}\to\R$ denote the projection onto the factor indexed by $i_0$, and let $j:\R\to\R^{(I)}$ be the inclusion of the same factor. 
Then we can choose $\ph$ so that the diagrams
$$\vcenter{\vbox{\xymatrix{
G\ar[r]^{\ph}&\R^{(I)}\ar[d]^{\pi}\\
G'\ar[u]^{\iota}\ar[r]^{\psi}&\R
}}}
\text{\hspace*{1cm} and \hspace*{1cm}}
\vcenter{\vbox{\xymatrix{
G\ar[r]^{\ph}&\R^{(I)}\\
G'\ar[u]^{\iota}\ar[r]^{\psi}&\R\ar[u]_{j}
}}}
$$
commute.
\end{thm}\begin{proof}
The first claim is Hahn's embedding theorem, as proven, for example, in \cite{HahnEmbeddingProof}. We briefly outline the construction of $\ph$ given there in order to show how to use it to prove the second claim as well.

First, $G$ has a canonical embedding into the divisible hull $V=G\otimes_{\Z}\Q$ of $G$. There is a unique ordering on $V$ which makes the inclusion $G\into V$ into an embedding of totally ordered abelian groups. Moreover, this embedding induces an order-isomorphism on sets of archimedean equivalence classes. This reduces the problem to finding an embedding $\ph:V\to\R^{(I)}$.

For each $i\in I$, let $H_{\geq i}:=\{x\in V:[|x|]\geq i\}\cup\{0\}$ and $H_{>i}:=\{x\in V:[|x|]> i\}\cup\{0\}$, both of which are convex rational subspaces of $V$. Since $H_{>i}\subseteq H_{\geq i}$ is an inclusion of $\Q$-vector spaces, there is a complement $K_i$ of $H_{>i}$ in $H_{\geq i}$, in the sense that $H_{\geq i}=H_{>i}\oplus K_i$.

The induced order on $K_i$ as a subgroup of $V$ makes $K_i$ archimedean, so there is an embedding $\ph_i:K_i\into\R$. These give an embedding of $\doplus_{i\in I}K_i\subseteq V$ into $\doplus_{i\in I}\R\subseteq\R^{(I)}$, and then it is shown that this can be extended to an embedding of $V$ into $\R^{(I)}$.

We now turn to the second claim. If $G'=\{0\}$ then any choice of $\ph$ as above works, so we assume from now on that $G'$ is nontrivial. Let $V'= G'\otimes_{\Z}\Q$ be the divisible hull of $G'$. The embeddings $G'\toup{\iota}G$ and $G'\toup{\psi}\R$ extend uniquely to embeddings $V'\toup{\iota}V$ and $V'\toup{\psi}\R$. These maps make the diagram
\comd{
G\inj[r]&V\\
G'\inj[r]\ar[u]^{\iota}&V'\ar[u]_{\iota}
}
commute, so it suffices to show that we can choose $\ph:V\to\R^{(I)}$ such that the diagrams
$$\vcenter{\vbox{\xymatrix{
V\ar[r]^{\ph}&\R^{(I)}\ar[d]^{\pi}\\
V'\ar[u]^{\iota}\ar[r]^{\psi}&\R
}}}
\text{\hspace*{1cm} and \hspace*{1cm}}
\vcenter{\vbox{\xymatrix{
V\ar[r]^{\ph}&\R^{(I)}\\
V'\ar[u]^{\iota}\ar[r]^{\psi}&\R\ar[u]_{j}
}}}
$$
commute.

We start by choosing $K_{i_0}$ appropriately. Note that $\iota(V')\subseteq H_{\geq i}$ and $\iota(V')\cap H_{>i}=\{0\}$. We have an inclusion $H_{>i}\oplus\iota(V')\subseteq H_{\geq i}$ of $\Q$-vector spaces, so there is a complement $K'$ of $H_{>i}\oplus\iota(V')$ in $H_{\geq i}$. We let $K_{i_0}=\iota(V')\oplus K'$.

Fix some nonzero $a\in V'$. Choose the embedding $\ph_{i_0}:K_{i_0}\to\R$ so that $\ph_{i_0}(\iota(a))=\psi(a)$. Then continue the construction of $\ph$ as above. We have that $\pi(\ph(\iota(a)))=\ph_{i_0}(\iota(a))=\psi(a)$. Since by Theorem~\ref{thm: rank1-emb} there is a unique embedding of $V'$ into $\R$ up to scalar multiplication, and $\pi\circ\ph\circ\iota$ and $\psi$ are both embeddings of $V'$ into $\R$ which agree on the nonzero element $a$, we have $\pi\circ\ph\circ\iota=\psi$. Because $\ph_{i_0}$ is the restriction of $\pi\circ\ph$ to $K_{i_0}$, we also have $\ph_{i_0}\circ\iota=\psi$.

By the construction of $\ph$, its restriction to $K_{i_0}$ is $j\circ\ph_{i_0}$. So $\ph\circ\iota=j\circ\ph_{i_0}\circ\iota=j\circ\psi$.
\end{proof}


\section{A universal topology on semirings}\label{sect: universal-topology}

Recall that a topological semiring is a semiring that is also a topological space on which the addition and multiplication maps are continuous. 
The main content of this appendix is the following theorem.

\begin{thm}\label{thm:UniversalTopOnSemiring}
Let $\psi:X\to A$ be a function with $X$ a topological space and $A$ a semiring. Then there is a topology $\tau$ on $A$ such that $(A,\tau)$ is a topological semiring and, for any semiring homomorphism $\ph:A\to B$ with $B$ a topological semiring, $\ph$ is continuous with respect to $\tau$ if and only if the composition $X\toup{\psi}A\toup{\ph}B$ is continuous. Moreover, $\psi:X\to(A,\tau)$ is continuous.
\end{thm}

To prove Theorem \ref{thm:UniversalTopOnSemiring}, we need a few elementary lemmas from point-set topology.

If $X\toup{\ph} Y$ is a function from a set $X$ to a topological space $Y$, recall that the initial topology on $X$ (with respect to $\ph$) is the coarsest topology on $X$ which makes $\ph$ continuous. Explicitly, a subset of $X$ is open in this topology if and only if it is of the form $\ph^{-1}(U)$ with $U\subset Y$ open.

\begin{lemma}\label{lemma:InitialMakesTopSemiring}
Let $\ph:A\to B$ be a homomorphism of semirings with $B$ a topological semiring. The initial topology on $A$ with respect to $\ph$ makes $A$ into a topological semiring.
\end{lemma}
\begin{proof}
We show that the multiplication map $m:A\times A\to A$ is continuous; the proof that the addition map is continuous is analogous.

Note that we have a commutative diagram
\comd{
A\times A\ar[r]^{m}\ar[d]_{\ph\times\ph}&A\ar[d]^{\ph}\\
B\times B\ar[r]^{m}&B.
}
If $V\subset A$ is open, then $V=\ph^{-1}(U)$ for some open set $U\subset B$. Then we have that $m^{-1}(V)=(\ph\times\ph)^{-1}(m^{-1}(U))$. Since $m:B\times B\to B$ and $\ph\times\ph:A\times A\to B\times B$ are both continuous, $(\ph\times\ph)^{-1}(m^{-1}(U))$ is open in $A\times A$.
\end{proof}

Given a set $\frakT$ of topologies on a set $X$, the supremum of $\frakT$ is the smallest topology $\tau$ which contains every 
$\calT\in\frakT$. 
Explicitly, if $\frakT$ is nonempty, the open sets of $\tau$ are arbitrary unions of finite intersections of sets in $\bigcup_{\calT\in \frakT}\calT$. That is, $\bigcup_{\calT\in \frakT}\calT$ is a subbase for $\tau$.

\begin{lemma}\label{lemma:SupPreservesContinuity}
Let $\ph:X\to Y$ be a function, let $X$ be a topological space, and let $\frakT$ be a set of topologies on $Y$, each of which make $\ph$ continuous. Let $\tau$ be the supremum of $\frakT$. Then the map $\ph:X\to(Y,\tau)$ is continuous.
\end{lemma}\begin{proof}
Every open set in $(Y,\tau)$ is a union of finite intersections of sets in $\bigcup_{\calT\in \frakT}\calT$. Since unions and intersections are preserved by taking inverse images, and because $\ph^{-1}(U)$ is open in $X$ for each $U\in\bigcup_{\calT\in \frakT}\calT$, we get that the inverse image of any open set in $(Y,\tau)$ is open in $X$.
\end{proof}

For any topology $\calT$ on a set $X$, we let $\calT\times\calT$ denote the product topology on $X\times X$.

\begin{lemma}\label{lemma:SupPreservesTopSemiring}
Let $\frakT$ be a collection of topologies on a semiring $A$ such that $(A,\calT)$ is a topological semiring for every $\calT\in\frakT$. Let $\tau$ be the supremum of $\frakT$. Then $(A,\tau)$ is a topological semiring.
\end{lemma}\begin{proof}
We will show that the multiplication map $m:A\times A\to A$ is continuous with respect to $\tau$; the proof that the addition map is continuous is analogous. Since $\bigcup_{\calT\in\frakT}\calT$ is a subbase for $\tau$, it suffices to check that $m^{-1}(U)$ is in $\tau\times\tau$ for every $U\in\bigcup_{\calT\in\frakT}\calT$. Given such a $U$, fix $\calT\in\frakT$ such that $U\in\calT$. Then $m^{-1}(U)\in\calT\times\calT\subseteq\tau\times\tau$.
\end{proof}

\begin{proof}[Proof of Theorem \ref{thm:UniversalTopOnSemiring}]
Let $\frakT$ be the set of topologies on $A$ which are the initial topology on $A$ with respect to some semiring homomorphism $\ph:A\to B$ with $B$ a topological semiring and $X\toup{\psi}A\toup{\ph}B$ continuous. Note that $\frakT$ is not empty because it contains the trivial topology $\{\emptyset, A\}$. Let $\tau$ be the supremum of $\frakT$.

Lemma~\ref{lemma:InitialMakesTopSemiring} tells us that $(A,\calT)$ is a topological semiring for each $\calT\in\frakT$. Thus, by  Lemma~\ref{lemma:SupPreservesTopSemiring}, we have that $(A,\tau)$ is a topological semiring.

Since $\psi:X\to (A,\calT)$ is continuous for all $\calT\in\frakT$, Lemma \ref{lemma:SupPreservesContinuity} gives us that the map $\psi:X\to(A,\tau)$ is continuous. In particular, if $\ph:(A,\tau)\to B$ is a continuous semiring homomorphism, then $X\toup{\psi}(A,\tau)\toup{\ph}B$ is continuous. For the other direction, suppose that $\ph:A\to B$ is a semiring homomorphism such that $X\toup{\psi}A\toup{\ph}B$  is continuous, and let $\calT$ be the initial topology on $A$ with respect to $\ph$. Since $\tau$ contains $\calT$, then the map $\id_A:(A,\tau)\to(A,\calT)$ is continuous. Because $\ph:(A,\calT)\to B$ is continuous, this gives us that $\ph:(A,\tau)\to B$ is continuous.
\end{proof}

\begin{remark}
The proofs in this appendix do not use the full structure of semirings. The same argument goes through for an arbitrary topological algebraic structure, by which we mean an algebraic structure with a topology on the underlying set which makes the operations continuous. In particular, these results apply to non-idempotent semirings.
\end{remark}


\section{Transcendence degree of a totally ordered semifield extension}\label{app: trdeg}

We propose a definition of transcendence degree for an extension of totally ordered semifields based on the definition of algebraic closure for semifields and we relate it to a quantity that occurs in our study of dimension in \ref{sect: dim-Cnvg}. 

\begin{definition}

Let $L/S$ be an extension of totally ordered semifields. We say that the elements $a_1, \ldots, a_n \in L$ are \textit{algebraically dependent over $S$} if there is a $b\in S$ and some integers $u_1, \ldots, u_n$, not all $0$, such that $a_1^{u_1}\cdots a_n^{u_n}=b$. 
If $a_1, \ldots, a_n \in L$ are not algebraically dependent over $S$, then we say that they are \emph{algebraically independent over $S$}.
\end{definition}

Note that, if $\overline{S}$ 
is 
the algebraic closure of $S$ in $L$, then $a_1, \ldots, a_n $ are algebraically independent over $S$ if and only if they are algebraically independent over $\overline{S}$. 

The extension of $S$ generated by $a_1, \ldots, a_n$ is $\{b a_1^{u_1}\cdots a_n^{u_n} : b\in S,  u_1, \ldots, u_n \in \Z \}$ because any polynomial (sum) will equal one of its terms. We will denote it by $S(a_1, \ldots, a_n)$.

\begin{prop}\label{prop: indepIFFalg}

Let $a_1, \ldots, a_n \in L$ be algebraically independent over $S$. Then $a_{n+1}$ is algebraic over $S(a_1, \ldots, a_n)$ if and only if the set $\{a_1, \ldots, a_n, a_{n+1}\}$ is algebraically dependent over $S$. 
\end{prop}

\begin{proof}
If $a_{n+1}$ is algebraic over $S(a_1, \ldots, a_n)$, then $a_{n+1}^{u_{n+1}}= c$, for some nonzero $u_{n+1} \in \Z$ and $c\in S(a_1, \ldots, a_n)$, then $c =  b a_1^{u_1}\cdots a_n^{u_n}$, for some $b\in S$ and $u_1, \ldots, u_n \in \Z$ and so $b = a_1^{-u_1}\cdots a_n^{-u_n}a_{n+1}^{u_{n+1}}$ implying that $a_1, \ldots, a_n, a_{n+1}$ are algebraically dependent over $S$.

For the other direction, 
say $a_1, \ldots, a_{n+1}$ are algebraically dependent over $S$, so for some $b\in S$ $a_1^{u_1}\cdots a_{n+1}^{u_n+1}=b$  and $u_1, \ldots, u_n \in \Z$, not all 0. Since $a_1, \ldots, a_n$ are algebraically independent over $S$, then $u_{n+1} \neq 0$, and so $a_{n+1}^{u_{n+1}} = b a_1^{-u_1}\cdots a_n^{-u_n}$, 
showing 
that $a_{n+1}^{u_{n+1}}$ is algebraic over $S(a_1, \ldots, a_n)$.
\end{proof}

\begin{defi}
An extension $L/S$ of totally ordered semifields 
\emph{has a finite transcendence degree} 
if there exist $a_1, \ldots, a_n \in L$ such that $L = \overline{S(a_1, \ldots, a_n)}$.
\end{defi}

\begin{prop}

Let $a_1, \ldots, a_n \in L$ be algebraically independent over $S$. Then $L = \overline{S(a_1, \ldots, a_n)}$ if and only if $\{ a_1, \ldots, a_n \}$ is maximal among algebraically independent sets over $S$.
\end{prop}

\begin{proof}

By definition, $L \neq \overline{S(a_1, \ldots, a_n)}$ if and only if there is some $a_{n+1}\in L$ which is not algebraic over $S(a_1, \ldots, a_n)$. By Proposition~\ref{prop: indepIFFalg}, this is equivalent to there being some $a_{n+1}\in L$ such that $\{a_1, \ldots, a_{n+1}\}$ is algebraically independent over $S$. But this exactly says that $\{a_1, \ldots, a_{n}\}$ is not maximal among algebraically independent sets over $S$.
\end{proof}

\begin{prop}
Let $L = \overline{S(a_1, \ldots, a_n)}$. Then  $\{a_1, \ldots, a_{n}\}$ is algebraically independent if and only if $\{a_1, \ldots, a_{n}\}$ is minimal among sets $\{b_1, \ldots, b_m \}\subseteq L$, such that $L = \overline{S(b_1, \ldots, b_m)}$.
\end{prop}

\begin{proof}

The set $\{a_1, \ldots, a_{n}\}$ is not minimal if and only if, for some $i$, $a_i \in \overline{S(a_1, \ldots, \hat{a_i}, \ldots a_{n})}$, where $\{a_1, \ldots, \hat{a_i}, \ldots, a_{n}\}$ denotes the set $\{a_1, \ldots, a_{n}\}$ with $a_i$ removed from it. This is equivalent to saying that there is some $i$ and $u_i\in\Z_{>0}$ such that $a_i^{u_i} = b\prod_{j\neq i} a_j^{-u_j}$, for some $u_j\in\Z$ and $b\in S$. This happens exactly if $b= a_1^{u_1}\cdots a_i^{u_i} \cdots a_n^{u_n}$, with $b\in S$ and $u_1, \ldots, u_n \in \Z$ not all 0, i.e., if $a_1, \ldots, a_n \in S$ are algebraically dependent over $S$.
\end{proof}

\begin{defi}
A \textit{transcendence basis} of $L/S$ is $\{ a_1, \ldots, a_n  \}$ such that any of the following equivalent statements holds:
\begin{itemize}
    \item The set $\{a_1, \ldots, a_{n}\}$ is algebraically independent over $S$ and $L = \overline{S(a_1, \ldots, a_n)}$.
    \item The set $\{a_1, \ldots, a_{n}\}$ is maximally algebraically independent over $S$.
    \item The set $\{a_1, \ldots, a_{n}\}$ is minimal such that $L = \overline{S(a_1, \ldots, a_n)}$.
\end{itemize}
\end{defi}

\begin{lemma}\label{lemma:algIndepUniqueRep}
Let $a_1,\ldots,a_n\in L^\times$. Then $a_1,\ldots,a_n$ are algebraically independent over $S$ if and only if each element of $L^\times$ can be written as $ca_1^{u_1}\cdots a_n^{u_n}$ with $c\in S^\times$ and $u_i\in\Z$ in at most one way. That is, if $ca_1^{u_1}\cdots a_n^{u_n}=da_1^{z_1}\cdots a_n^{z_n}$ implies $c=d$ and $u_i=z_i$ for $1\leq i\leq n$.
\end{lemma}\begin{proof}
Suppose that $a_1,\ldots,a_n$ are algebraically dependent, so there are $u_1,\ldots,u_n\in\Z$ and $c\in S$, not all zero, such that $a_1^{u_1}\cdots a_n^{u_n}=c$. Then $c^{-1}a_1^{u_1}\cdots a_n^{u_n}=1_S=1_S a_1^0\cdots a_n^0$ with some $u_i\neq0$.

Now suppose that $a_1,\ldots,a_n$ are algebraically independent and $ca_1^{u_1}\cdots a_n^{u_n}=da_1^{z_1}\cdots a_n^{z_n}$ with $c,d\in S^\times$ and $u_i,z_i\in\Z$. Then $a_1^{u_1-z_1}\cdots a_n^{u_n-z_n}=\frac{d}{c}\in S^\times$, so $u_i-z_i=0$ for all $i$. Thus $z_i=u_i$ and $\frac{d}{c}=a_1^0\cdots a_n^0$, so $d=c$.
\end{proof}

\begin{prop}\label{prop:BasisExchangeAxiom}
Let $\{a_1, \ldots, a_{n}\}$ and $\{b_1, \ldots, b_{m}\}$ be transcendence bases of $L/S$. Then for some $1\leq j\leq m$ the set $\{b_j, a_2, \ldots, a_{n}\}$ is a transcendence basis.  
\end{prop}

\begin{proof}
We will first show that, if $\{a_1, \ldots, a_{n}\}$ and $\{b_1, \ldots, b_{m}\}$ are transcendence bases of $L/S$, then, for some $1\leq j\leq m$, $b_j^v = c a_1^{u_1}\cdots a_n^{u_n}$ with $c\in S$,  $u_1, \ldots, u_n \in \Z$ and $v, u_1 \neq 0$.
Suppose this were not true. Since each $b_j \in \overline{S(a_1, \ldots, a_n)}$ there is $v_j\in\Z_{>0}$ such that $b_j^{v_j}\in S(a_1,\ldots,a_n)$, so we can write $b_j^{v_j} = c_j a_1^{u_{j,1}}\cdots a_n^{u_{j,n}} = c_j a_2^{u_{j,2}}\cdots a_n^{u_{j,n}},$ where the last equality holds because $u_{j,1} = 0$ by hypothesis. By raising the previous equality to large enough power, we can assume without loss of generality that all $v_j$ are equal to some $v$, i.e., $b_j^{v} = c_j a_2^{u_{j,2}}\cdots a_n^{u_{j,n}}$.
Since $a_1 \in \overline{S(b_1, \ldots, b_m)}$, we have $a_1^w = d b_1^{z_1}\cdots b_m^{z_m}$ for some $d\in S$ and $w, z_1, \ldots, z_m \in \Z$, with $w>0$. Thus,
$a_1^{wv} = d^v (b_1^v)^{z_1}\cdots(b_m^v)^{z_m} = d^v \prod_{j=1}^m (c_j a_2^{u_{j,2}}\cdots a_n^{u_{j,n}})^{z_j},$ so $a_1 \in \overline{S(a_2, \ldots, a_n)}$, contradicting the fact that the set $\{a_1, \ldots, a_{n}\}$ is algebraically independent.

Note that $a_1\in\overline{S(b_j,a_2,\ldots,a_n)}$ because $a_1^{-u_1}=cb_j^{-v}a_2^{u_2}\cdots a_n^{u_n}$ with $u_1\neq0$. Thus 
$$\overline{S(b_j,a_2,\ldots,a_n)}\supseteq\overline{S(a_1,a_2,\ldots,a_n)}=L.$$

Since the set $\{a_2, \ldots, a_{n}\}$ is algebraically independent, to show that $\{b_j, a_2, \ldots, a_{n}\}$ is algebraically independent, it suffices to see that $b_j \not\in \overline{S(a_2, \ldots, a_n)}$. Assume for contradiction that $b_j^w = d a_2^{z_2} \cdots a_n^{z_n}$ for some $d\in S^\times$ and $w,z_2,\ldots,z_n\in\Z$ with $w \neq 0$. Then 
$$
c^w a_1^{wu_1}\cdots a_n^{wu_n}=
b_j^{wv}
=d^v a_2^{vz_2} \cdots a_n^{vz_n}
$$
where $wu_1\neq0$, which is impossible by Lemma~\ref{lemma:algIndepUniqueRep}.
\end{proof}

\begin{coro}
If $\{a_1,\ldots,a_n\}$ and $\{b_1,\ldots,b_m\}$ are transcendence bases of $L/S$, then $m=n$.
\end{coro}\begin{proof}
Proposition~\ref{prop:BasisExchangeAxiom} shows that the transcendence bases of $L/S$ contained in 
$$\{a_1,\ldots,a_n,b_1,\ldots,b_m\}$$ 
form the bases of a matroid. Thus, any two transcendence bases have the same size.
\end{proof}

\begin{defi}
Let $L/S$ be an extension of totally ordered semifields which has a finite transcendence degree. Then the \emph{transcendence degree of $L/S$}, denoted $\trdeg(L/S)$, is the size of any transcendence basis of $L/S$. 
\end{defi}

\begin{prop}\label{prop:TrDegIsRk}
If either $\rk(L^\times/S^\times)$ or $\trdeg(L/S)$ is finite, then $\rk(L^\times/S^\times) = \trdeg(L/S)$.
\end{prop}

\begin{proof}
Let $\trdeg(L/S)=n$ be finite. Let $\{ a_1, \ldots, a_n \}$ be a transcendence basis for $L/S$, and let $\left< a_1, \ldots, a_n\right>$ denote the subgroup of $L^\times/S^\times$ generated by $ a_1, \ldots, a_n$. Then the group $(L^\times/S^\times)/\left< a_1, \ldots, a_n\right> \cong L^\times/S( a_1, \ldots, a_n)^\times$ is torsion and $\rk(L^\times/S^\times) \leq n = \trdeg(L/S).$ 

On the other hand, since $a_1, \ldots, a_n$ are algebraically independent over $S$, they are $\Z$-linearly independent in $L^\times/S^\times$, so $\rk(L^\times/S^\times) \geq n = \trdeg(L/S)$.
Now assume that $\rk(L^\times/S^\times)$ is finite; we need to show that $\trdeg(L/S)$ is also finite. 
There are elements $ a_1, \ldots, a_m \in L^\times$ such that $(L^\times/S^\times)/\left< a_1, \ldots, a_m\right> \cong L^\times/S( a_1, \ldots, a_m)^\times$ is torsion. 
This shows that $L=\overline{S(a_1, \ldots, a_{m})}$, showing that $L/S$ has a finite transcendence degree.
\end{proof}

\section{The adic analytification of an affine variety}\label{app:AdicAnalytification}
\newcommand{\ad}{\mathrm{ad}}
\newcommand{\spa}{\operatorname{Spa}}
\renewcommand{\Spa}{\spa}
\newcommand{\fraka}{\mathfrak{a}}
\newcommand{\im}{\operatorname{Im}}

Let $X=\Spec(B)$ be an affine variety over a field with a nontrivial valuation $v:K\onto S_v\subseteq\T$. We endow $K$ with the topology induced by this valuation and let $K^\circ$ be the valuation ring of $K$. The goal of this appendix is to prove the following result which is well known to experts but for which we could not find a reference.

\begin{prop}\label{prop:AdicAnalytification}
The adic analytification $X^{\ad}$ of $X$ can be seen as the set of valuations $w:B\to S$ where $S$ is a totally ordered semifield, $w|_K$ is continuous, and $w(a)\leq1_S$ for all $a\in K^\circ$, modulo equivalence of valuations.
\end{prop}

We first recall the definition of $X^{\ad}$. Our presentation combines pieces from \cite[Section 5.4]{Bos14} 
and 
\cite[Section 4]{Hub94}. 

For any affinoid ring\footnote{An \emph{affinoid ring} is a pair $(R,R^+)$ where $R$ is an f-adic ring and $R^+\subseteq R$ is a subring that is open and integrally closed in $R$ and every element of $R^+$ is power bounded. The only aspects of this that we will need to know are that $R$ is a topological ring and that $R^+$ is a subring.} $(R,R^+)$, the \emph{adic spectrum} of $(R,R^+)$ is the set $\spa(R,R^+)$ of valuations $w:R\to S$ where $S$ is a totally ordered semifield, $w$ is continuous, and $w(b)\leq 1_S$ for all $b\in R^+$, modulo equivalence of valuations. 
If $\ph:R_1\to R_2$ is a continuous ring homomorphism with $\ph(R_1^+)\subseteq R_2^+$ then $\ph$ induces a map $\ph^*:\Spa(R_2,R_2^+)\to\Spa(R_1,R_1^+)$.

In light of the previous paragraph, we can restate Proposition~\ref{prop:AdicAnalytification} as follows: $X^{\ad}$ can be seen as the set of those valuations $w$ on $B$ such that $w|_K\in\Spa(K,K^\circ)$.

Fix a presentation $B\cong K[x_1,\ldots,x_n]/\fraka$ of $B$.

Given $r\in\T^\times$, let $\|\bullet\|_r:K[x_1,\ldots,x_n]\to\T$ be given by 
$$\left\|\dsum_{I=(i_1,\ldots,i_n)\in\N^n}a_Ix^I\right\|_r=\dsum_{I=(i_1,\ldots,i_n)\in\N^n}v(a_I)r^{i_1+\cdots+i_n}.$$
Note that, if $r\in S_v$, then $\|\bullet\|_r$ takes values in $S_v$.
We let $A_r$ denote $K[x_1,\ldots,x_n]$ with the topology induced by $\|\bullet\|_r$ and $A_r^+=\{f\in A\;:\;\|f\|_r\leq 1_{\T}\}$. Then $\Spa(A_r,A_r^+)$ is the adic ball of radius $r$.

If $r_1\leq r_2$ then, for any $f\in K[x_1,\ldots,x_n]$, $\|f\|_{r_1}\leq\|f\|_{r_2}$; this shows that the identity map $A_{r_2}\to A_{r_1}$ is continuous and $A_{r_2}^+\subseteq A_{r_1}^+$. 
We therefore have inclusion maps $\spa(A_{r_1},A_{r_1}^+)\to \spa(A_{r_2},A_{r_2}^+)$ which make the collection of $\Spa(A_r,A_r^+)$s into a directed system; the direct limit of this system is $(\A^n)^{\ad}$. Similarly, there are inclusion maps $\spa\left(\dfrac{A_{r_1}}{\fraka}, \dfrac{A_{r_1}^+}{\fraka\cap A_{r_1}^+} \right)\to \spa\left(\dfrac{A_{r_2}}{\fraka}, \dfrac{A_{r_2}^+}{\fraka\cap A_{2_1}^+} \right)$ making the collection of $\spa\left(\dfrac{A_{r}}{\fraka}, \dfrac{A_{r}^+}{\fraka\cap A_{r}^+} \right)$s into a directed system; the direct limit of this system is $X^{\ad}$.
Because the maps in these directed systems are inclusions, these direct limits are unions.

We now observe that it suffices to show Proposition~\ref{prop:AdicAnalytification} in the case where $X=\A^n$. To see this, note that the effect of quotienting by $\fraka$ is that we only consider valuations on $K[x_1,\ldots,x_n]$ which vanish on $\fraka$. Since this is the same for both the union of the $\Spa(A_r,A_r^+)$s and the set of valuations $w:K[x_1,\ldots,x_n]\to S$ with $w|_K\in\Spa(K,K^\circ)$, Proposition~\ref{prop:AdicAnalytification} will follow for $X=\Spec(K[x_1,\ldots,x_n]/\fraka)$ once we show it for $X=\Spec(K[x_1,\ldots,x_n])$.

The inclusion $(K,K^\circ)\into(A_r,A_r^+)$ induces the restriction map $\Spa(A_r,A_r^+)\to\Spa(K,K^\circ)$. So if $w\in X^{\ad}=\dcup_{r\in\T^\times}\Spa(A_r,A_r^+)$ then $w|_K\in\Spa(K,K^\circ)$.

Towards seeing the other inclusion, let $w:K[x_1,\ldots,x_n]\to S$ be any valuation with $w|_K\in\Spa(K,K^\circ)$. We will write $|f|$ in place of $w(f)$. We start by relating $w$ and $v$ directly.

\begin{lemma}\label{lemma:EmbedValueGroups}
We can factor $w|_K:K\to S$ as $K\toup{v} S_v \toup{\theta} S$ with $\theta:S_v\to S$ a continuous homomorphism of totally ordered semifields.
\end{lemma}
\begin{proof}
We start by observing that, if $a\in K$ has $v(a)=1$, then $|a|=1$. To see this, note that $a,a^{-1}\in K^\circ$, so we must have $|a|\leq 1$ and $|a^{-1}|\leq 1$, i.e., $|a|\leq 1$ and $|a|\geq 1$.


Consider the quotient $K/\{a\in K\;:\;v(a)=1\}$ of multiplicative monoids. Recall that $S_v\cong K/\{a\in K\;:\;v(a)=1\}$ and that, putting the quotient topology on $K/\{a\in K\;:\;v(a)=1\}$, this identifies the topologies as well (see footnote~\ref{footnote:HuberTopology} on page \pageref{footnote:HuberTopology}). Since $w|_K:K\to S$ is a morphism of multiplicative monoids and $\{a\in K\;:\;v(a)=1\}$ is contained in the kernel of this morphism, $w|_K$ factors as $K\to K/\{a\in K\;:\;v(a)=1\}\toup{\theta} S$ with $\theta$ a homomorphism of multiplicative monoids. Because we have endowed $K/\{a\in K\;:\;v(a)=1\}$ with the quotient topology, $\theta$ is continuous.

It remains only to show that $\theta$ is additive. Say $b,c\in K$ and let $[b]$ and $[c]$ denote their classes in $K/\{a\in K\;:\;v(a)=1\}$; we want to show that $\theta([b]+[c])=\theta([b])+\theta([c])$. If $b$ or $c$ is $0_K$ this is trivial, so assume that $b$ and $c$ are both nonzero. Note that $\theta([b])=|b|$ and $\theta([c])=|c|$. 

Since $K/\{a\in K\;:\;v(a)=1\}\cong S_v$ is totally ordered, we may assume without loss of generality that $[b]+[c]=[b]$, i.e., $[c]\leq[b]$. This says that $[cb^{-1}]\leq 1$, meaning that $cb^{-1}\in K^\circ$. Therefore, $|cb^{-1}|\leq 1$, i.e., $|c|\leq|b|$. So we have
$$\theta([b]+[c])=\theta([b])=|b|=|b|+|c|,$$
as desired.
\end{proof}

In order to find $r$ for which we will show that $w\in\Spa(A_r,A_r^+)$, we use the following brief observation.

\begin{lemma}\label{lemma:SvHasEnoughElements}
For any $\alpha\in S^\times$ there are $\beta_1,\beta_2\in S_v^\times$ with $\theta(\beta_1)<\alpha<\theta(\beta_2)$.
\end{lemma}\begin{proof}
Since $\theta:S_v\to S$ is continuous,  Lemma~\ref{lemma:ContsHomOfTotOrdSemifields} tells us that there are $\beta,\beta'\in S_v^\times$ with $\theta(\beta)<\alpha$ and  $\theta(\beta')<\alpha^{-1}$. So, taking $\beta_1=\beta$ and $\beta_2=(\beta')^{-1}$, we have the desired result.
\end{proof}

If $|x_\ell|=0_S$ for all $\ell=1,\ldots,n$ then $w$ is in $\Spa(A_r,A_r^+)$ for all $r$, so $w\in X^{\ad}$. So we consider the case where $|x_\ell|\neq0_S$ for some $\ell$, i.e., $\rho:=\dsum_{\ell=1}^n|x_\ell|\in S^\times$. By Lemma~\ref{lemma:SvHasEnoughElements} there is some $r\in S_v$ with $\theta(r)>\rho$. Fix such an $r$, so $\theta(r)>|x_\ell|$ for all $\ell=1,\ldots,n$. This allows us to relate $w$ and $\|\bullet\|_r$.

\begin{lemma}\label{lemma:BoundTheNorm}
For all $f\in K[x_1,\ldots,x_n]$, $|f|\leq\theta(\|f\|_r)$.
\end{lemma}\begin{proof}
Write $f=\dsum_{I=(i_1,\ldots,i_n)\in\N^n}a_Ix^I$. Then
\begin{align*}
|f|&\leq \dsum_{I=(i_1,\ldots,i_n)\in\N^n}|a_I|\dprod_{\ell=1}^n |x_\ell|^{i_\ell} 
=\dsum_{I=(i_1,\ldots,i_n)\in\N^n}\theta(v(a_I))\dprod_{\ell=1}^n |x_\ell|^{i_\ell} 
\\
&\leq\dsum_{I=(i_1,\ldots,i_n)\in\N^n}\theta(v(a_I))\dprod_{\ell=1}^n \theta(r)^{i_\ell} 
=\theta\left(\dsum_{I=(i_1,\ldots,i_n)\in\N^n}v(a_I)\dprod_{\ell=1}^n r^{i_\ell}\right)\\
&=\theta(\|f\|_r).
\end{align*}
\end{proof}

We are now in a position to prove the two criteria for $w$ to be in $\Spa(A_r,A_r^+)$.

\begin{lemma}\label{lemma:ThisRMakesCont}
The map $w:A_r\to S$ is continuous.
\end{lemma}\begin{proof}
Let $A^w$ denote $K[x_1,\ldots,x_n]$ with the topology induced by $w$. Then $w:A^w\to S$ is continuous, so it suffices to show that the identity map $A_r\to A^w$ is continuous. Fix $\eps\in S^\times$. By Lemma~\ref{lemma:SvHasEnoughElements} there is some $\delta\in S_v^\times$ with $\theta(\delta)<\eps$. Then, for any $f,g\in A_r$ with $\|f-g\|_r<\delta$, we have $|f-g|\leq\theta(\|f-g\|_r)\leq\theta(\delta)<\eps$.
\end{proof}

\begin{lemma}
For any $f\in A_r^+$, $|f|\leq 1_S$.
\end{lemma}\begin{proof}
We have $|f|\leq \theta(\|f\|_r)\leq\theta(1_{S_v})=1_S$.
\end{proof}

Thus, $w|_K\in\Spa(A_r,A_r^+)\subseteq X^{\ad}$. This completes the proof of Proposition~\ref{prop:AdicAnalytification}.







\renewcommand{\bibliofont}{\small}
\bibliographystyle{alpha}

\end{document}